\newif\ifdraftmode
\draftmodefalse
\documentclass[10pt]{article}

\ifdraftmode
\fi

\newcommand\thisShorttitle{Weighted sparsity and sparse tensor networks for least squares approximation}
\newcommand\thisTitle{\thisShorttitle}
\newcommand\thisSubject{math.NA, cs.LG}
\newcommand\thisAuthor{Philipp Trunschke, Martin Eigel, Anthony Nouy}
\newcommand\thisKeywords{empirical least-squares, sampling complexity, sparse tensor networks, alternating least squares}

\usepackage{arxiv}

\usepackage[utf8]{inputenc}      % allow utf-8 input
\usepackage[T1]{fontenc}         % use 8-bit T1 fonts
\usepackage[dvipsnames]{xcolor}  % \color command
\usepackage{float}
\usepackage{hyperref}            % hyperlinks
\usepackage{url}                 % simple URL typesetting
\usepackage{doi}                 % hyperlink doi numbers
\usepackage{microtype}           % microtypography
\usepackage{graphicx}
\hypersetup{
    breaklinks,
    raiselinks=true,
    colorlinks,
    citecolor=gray,
    filecolor=gray,
    linkcolor=gray,
    urlcolor=gray,
    bookmarksopenlevel=1,    % Gliederungstiefe der Bookmarks
    bookmarksopen=true,      % Expandierte Untermenues in Bookmarks
    bookmarksnumbered=true,  % Nummerierung der Bookmarks
    bookmarkstype=toc,       % Art der Verzeichnisses
    pdfdisplaydoctitle=true, % Display document title instead of file name in title bar.
    pdfstartview=FitH,       % Starting view of PDF document: Fits the width of the page to the window.
    pdfpagemode=UseOutlines, % Bookmarks im Viewer anzeigen
    pdfpagelayout=OneColumn, % Displays the document in one column; continuous scrolling.
    pdftitle={\thisTitle},
    pdfsubject={\thisSubject},
    pdfauthor={\thisAuthor},
    pdfkeywords={\thisKeywords},
}

\usepackage{lipsum}              % Can be removed after putting your text content

\usepackage{enumitem}            % provides user control over the layout of list environments
\usepackage[ruled,linesnumbered]{algorithm2e}  % algorithm environment
\usepackage{caption}             % Allow numbered figures without caption.
\usepackage{subcaption}          % Subfigues

\usepackage{csquotes}  % Needed for biblatex+babel
\usepackage[
    bibstyle=alphabetic,
    citestyle=alphabetic,
    sorting=nyt, % ynt,
    sortcites=true,
    giveninits=true,
    maxbibnames=99,
    backend=biber,
    maxalphanames=5,
    backref=true,
    natbib=true,     % \citet, \citep
    url=false,
    date=year, % terse,
    abbreviate=false,
]{biblatex}

\DeclareFieldFormat{eprint:hal}{%
    \mkbibacro{HAL}\addcolon\space
    \ifhyperref
        {\href{https://hal.archives-ouvertes.fr/#1}{%
        \nolinkurl{#1}
        \iffieldundef{eprintclass}
            {}
            {\addspace\UrlFont{\mkbibbrackets{\thefield{eprintclass}}}}}}
        {%
        \nolinkurl{#1}
        \iffieldundef{eprintclass}
            {}
            {\addspace\UrlFont{\mkbibbrackets{\thefield{eprintclass}}}}}}
\DeclareFieldAlias{eprint:HAL}{eprint:hal}

\DeclareSourcemap{
  \maps[datatype=bibtex]{
    \map{
      \step[fieldset=issn, null]
    }
  }
}
\renewbibmacro{in:}{}  % Do not print "In."

\def\keywordname{{\bfseries \emph{Key words.}}}%

\def\mscclassname{{\bfseries \emph{AMS subject classifications.}}}%
\def\mscclasses#1{\par\addvspace\medskipamount{\rightskip=0pt plus1cm
\def\and{\ifhmode\unskip\nobreak\fi\ $\cdot$
}\noindent\mscclassname\enspace\ignorespaces#1\par}}
\def\codename{{\bfseries \emph{Code.}}}%
\def\code#1{\par\addvspace\medskipamount{\rightskip=0pt plus1cm\noindent\codename\enspace\ignorespaces\url{#1}\par}}

\ifdraftmode
    \definecolor{amaranth}{rgb}{0.9, 0.17, 0.31}%
    \definecolor{americanrose}{rgb}{1.0, 0.01, 0.24}
    \colorlet{alertcolor}{amaranth}
    \colorlet{notecolor}{MidnightBlue}

    \makeatletter
    \newcommand{\todo}[1]{\marginpar{\tiny\color{alertcolor}#1}\@latex@warning{#1}\xspace}
    \makeatother
    \newcommand{\note}[1]{\marginpar{\tiny\color{notecolor}#1}}
    
    \usepackage[notref,notcite]{showkeys}
    
    \definecolor{bleudefrance}{rgb}{0.19, 0.55, 0.91}
    \newcommand{\numRevisions}{2}
    \newcommand{\revision}[2][0]{%
    \begingroup%
        \newcount\colorRatio%
        \colorRatio=\numexpr(100*(#1+1))/\numRevisions\relax%
        \colorlet{revisionColor}{bleudefrance!\the\colorRatio!black}\color{revisionColor}#2%
    \endgroup}

\else
    \newcommand{\todo}[1]{}
    \newcommand{\note}[1]{}
\fi

\usepackage{pgf,tikz}
\usetikzlibrary{positioning}
\usetikzlibrary{calc}
\usetikzlibrary{intersections}
\usetikzlibrary{decorations.pathreplacing}  % braces

\tikzset{core/.style={inner sep=0pt}}
\tikzset{contraction/.style={line width=0.75}}
\tikzset{contractionDots/.style={contraction, dotted}}

\usepackage{amsmath}             % basic math symbols, operators and equation environments
\usepackage{amsfonts}            % blackboard math symbols
\usepackage{amssymb}             % \lesssim
\usepackage{amsthm}              % an enhanced version of LATEX’s \newtheorem command
\usepackage{thmtools}      % frontend for amsthm, Erst nach hyperref laden...
\usepackage[fixamsmath,disallowspaces]{mathtools} % extends amsmath and fixes some bugs, defines DeclareMathOperator and DeclarePairedDelimiters
\mathtoolsset{showonlyrefs=true}
\usepackage{newtxtext}
\usepackage[varvw]{newtxmath}  % mathptmx is very old and does not support bold fonts.

\colorlet{dimgray}{black!35!white}
\colorlet{lightgray}{dimgray!35!white}
\declaretheoremstyle[bodyfont=\itshape, mdframed={backgroundcolor=lightgray, linecolor=dimgray, linewidth=0.75pt, innertopmargin=1.5ex}]{claim}
\declaretheorem[style=claim]{theorem}
\declaretheorem[style=claim, numberlike=theorem]{lemma}
\declaretheorem[style=claim, numberlike=theorem]{proposition}
\declaretheorem[style=claim, numberlike=theorem]{corollary}

\declaretheoremstyle[mdframed={backgroundcolor=lightgray, linecolor=dimgray, linewidth=0.75pt, innertopmargin=1.5ex}]{definition}
\declaretheorem[style=definition, numberlike=theorem]{definition}
\declaretheoremstyle[bodyfont=\itshape, mdframed={backgroundcolor=white, linecolor=dimgray, linewidth=0.75pt, innertopmargin=1.5ex}]{remark}
\declaretheorem[style=remark, numberlike=theorem]{remark}
\declaretheoremstyle[mdframed={backgroundcolor=white, linecolor=dimgray, linewidth=0.75pt, innertopmargin=1.5ex}]{example}
\declaretheorem[style=example, numberlike=theorem]{example}

\newcommand{\indep}{\perp\kern-0.6em\perp}

\newcommand{\bfu}{\boldsymbol{u}}
\newcommand{\bfv}{\boldsymbol{v}}
\newcommand{\bfe}{\boldsymbol{e}}
\newcommand{\bff}{\boldsymbol{f}}
\newcommand{\bfx}{\boldsymbol{x}}
\newcommand{\bfy}{\boldsymbol{y}}
\newcommand{\bfomega}{{\boldsymbol{\omega}}}
\newcommand{\bfalpha}{{\boldsymbol{\alpha}}}
\newcommand{\bfbeta}{{\boldsymbol{\beta}}}
\newcommand{\bfrho}{{\boldsymbol{\rho}}}
\newcommand{\bfsigma}{{\boldsymbol{\sigma}}}

\newcommand*{\mbb}[1]{\mathbb{#1}}
\newcommand*{\mcal}[1]{\mathcal{#1}}

\newcommand*{\dd}{\ensuremath{\mathrm{d}}}
\newcommand*{\dx}[1][x]{\ensuremath{\,\dd{#1}}}

\DeclareMathOperator{\supp}{supp}

\let\inf\relax  % remove the definition of \inf before redeclaring it
\DeclareMathOperator*{\inf}{inf\vphantom{\sup}}
\DeclareMathOperator*{\argmin}{arg\,min}

\DeclarePairedDelimiter{\pars}{\ensuremath{(}}{\ensuremath{)}}
\DeclarePairedDelimiter{\bracs}{\ensuremath{[}}{\ensuremath{]}}
\DeclarePairedDelimiter{\braces}{\ensuremath{\{}}{\ensuremath{\}}}
\DeclarePairedDelimiter{\ceil}{\lceil}{\rceil}
\DeclarePairedDelimiter{\floor}{\lfloor}{\rfloor}

\DeclarePairedDelimiter{\norm}{\|}{\|}
\DeclarePairedDelimiter{\abs}{\lvert}{\rvert}

\makeatletter
\newcommand{\opnorm}{\@ifstar\@opnorms\@opnorm}
\newcommand{\@opnorms}[1]{%
  \left|\mkern-1.5mu\left|\mkern-1.5mu\left|
   #1
  \right|\mkern-1.5mu\right|\mkern-1.5mu\right|
}
\newcommand{\@opnorm}[2][]{%
  \mathopen{#1|\mkern-1.5mu#1|\mkern-1.5mu#1|}
  #2
  \mathclose{#1|\mkern-1.5mu#1|\mkern-1.5mu#1|}
}
\makeatother

\mathchardef\texthyphen="2D

\usepackage{dsfont}              % \mathds{1}

\let\oldbullet\bullet
\newlength{\raisebulletlen}
\setbox1=\hbox{$\bullet$}\setbox2=\hbox{\tiny$\bullet$}
\renewcommand\bullet{\raisebox{\raisebulletlen}{\,\tiny$\oldbullet$}\,}

\makeatletter
\newcommand*{\rom}[1]{\expandafter\@slowromancap\romannumeral #1@}
\makeatother

\providecommand\given{} % ensure it exists
\newcommand\givensymbol[1]{%
  \nonscript\;\delimsize#1\allowbreak\nonscript\;\mathopen{}%
}
\DeclarePairedDelimiterX\Set[1]\{\}{%
  \renewcommand\given{\givensymbol{\vert}}%
  #1%
}

\def\multiset#1#2{\ensuremath{\left(\kern-.3em\left(\genfrac{}{}{0pt}{}{#1}{#2}\right)\kern-.3em\right)}}
\newcommand*{\xhat}[2][0.3em]{#2\kern-#1\hat{\vphantom{#2}}\kern#1}
\title{\thisTitle} % \thanks{A preliminary version of this manuscript is available on \texttt{arXiv}. The manuscript is neither published nor submitted elsewhere.}}
\renewcommand{\undertitle}{}
\renewcommand{\shorttitle}{\thisShorttitle}
\date{}

\author{
\href{https://orcid.org/0000-0002-2995-126X}{\includegraphics[height=0.7em]{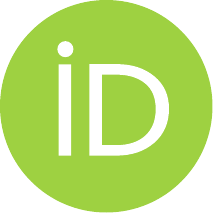}\hspace{1mm}\textcolor{black}{Philipp Trunschke\thanks{Corresponding Author. 
}}} \\
    Centrale Nantes \\
    Nantes Universit\'e \\
    Laboratoire de Math\'ematiques Jean Leray \\
    CNRS UMR 6629 \\
    France \\
    \href{mailto:philipp.trunschke@univ-nantes.fr}{\texttt{philipp.trunschke@univ-nantes.fr}} \\
\And
\href{https://orcid.org/0000-0003-2687-4497}{\includegraphics[height=0.7em]{orcid.pdf}\hspace{1mm}\textcolor{black}{Martin Eigel
}} \\
    Weierstrass Institute \\
    Mohrenstrasse 39 \\
    10117 Berlin, Germany \\
    \href{mailto:martin.eigel@wias-berlin.de}{\texttt{martin.eigel@wias-berlin.de}} \\
\And
\href{https://orcid.org/0000-0002-2149-2986}{\includegraphics[height=0.7em]{orcid.pdf}\hspace{1mm}\textcolor{black}{Anthony Nouy
}} \\
    Centrale Nantes \\
    Nantes Universit\'e \\
    Laboratoire de Math\'ematiques Jean Leray \\
    CNRS UMR 6629 \\
    France \\
    \href{mailto:anthony.nouy@ec-nantes.fr}{\texttt{anthony.nouy@ec-nantes.fr}}
}

\begin{document}
\maketitle
% \vspace{-5em}

\begin{abstract}
The approximation of high-dimensional functions is a ubiquitous problem in many scientific fields that is only feasible practically if advantageous structural properties can be exploited.
One prominent structure is sparsity relatively to some basis.
For the analysis of these best $n$-term approximations a relevant tool is the Stechkin's lemma.
In its standard form, however, this lemma does not allow to explain convergence rates for a wide range of relevant function classes.
This work presents a new weighted version of Stechkin's lemma that improves the best $n$-term rates for weighted $\ell^p$-spaces and associated function classes such as Sobolev or Besov spaces.
For the class of holomorphic functions, which for example occur as solutions of common high-dimensional parameter dependent PDEs, we recover exponential rates that are not directly obtainable with Stechkin's lemma.

This sparsity can be used to devise weighted sparse least squares approximation algorithms as known from compressed sensing.
However, in high-dimensional settings, classical algorithms for sparse approximation suffer the curse of dimensionality.
We demonstrate that sparse approximations can be encoded efficiently using tensor networks with sparse component tensors.
This representation gives rise to a new alternating algorithm for best $n$-term approximation with a complexity scaling polynomially in $n$ and the dimension.

We also demonstrate that weighted $\ell^p$-summability not only induces sparsity of the tensor but also low ranks.
This is not exploited by the previous format.
We thus propose a new low-rank tensor train format with a single weighted sparse core tensor and an ad-hoc algorithm for approximation in this format.
To analyse the sample complexity for this new model class we derive a novel result of independent interest that allows to transfer the restricted isometry property from one set to another sufficiently close set.
We then prove that the new model class is close enough to the set of weighted sparse vectors such that the restricted isometry property transfers.

Numerical examples illustrate the theoretical results for a benchmark problem from uncertainty quantification.

Although they lead up to the analysis of our final model class, our contributions on weighted Stechkin and the restricted isometry property are of independent interest and can be read independently.
\end{abstract}

%NOTE: Also change the pdfkeywords in \hypersetup!
\vspace{-1em}
\mscclasses{15A69 \and 41A30 \and 62J02 \and 65Y20 \and 68Q25}
\vspace{-1em}
\code{https://github.com/ptrunschke/sparse_als}

\section{Introduction}
\label{sec:introduction}

Approximating an unknown function from data is a fundamental problem in computational science and machine learning.
In many applications, the sought function may depend on a large number of parameters, rendering the approximation task susceptible to the \textit{curse of dimensionality} (CoD), i.e.\ an exponential complexity in the dimension of the problem or the amount of sample points required to obtain an accurate approximation.
This is particularly problematic when the amount of data available is limited due to practical constraints.
Nevertheless, many practically relevant functions can be approximated efficiently using a judiciously chosen set of functions.
Given a function that can be well approximated by a sparse expansion on some basis,  results in compressive sensing guarantee an accurate approximation from a small number of sample points.
The required sets (or dictionaries) can be found by exploiting regularity properties of the sought function.
A common characterisation of smooth functions is given in terms of the decay of their Fourier series.
This can be viewed as promoting a structured sparsity where low-order Fourier modes are more likely to contribute to the total norm of the function.
As a consequence, smooth functions admit approximately sparse representations, enabling an efficient numerical reconstruction.

Another form of low-dimensional structure induced by smoothness is low-rank approximability. 
This structure is exploited e.g.\ in reduced basis methods or proper orthogonal decomposition~\cite{cohen_devore_2015,nouy:2017_morbook} and in a more general form in hierarchical tensor formats such as the popular tensor trains (TT)~\cite{oseledets_tensor-train_2011}.

The aim of this paper is to develop a sparse approximation algorithm which simultaneously exploits sparsity and low-rank properties, thus enabling an efficient approximation of large function sets.
Central tools for this are a new weighted version of the well-known Stechkin's lemma and a novel analysis of the restricted isometry property which allows to accommodate any space which can be approximated by weighted sparse expansions.
These developments should be of independent interest in the study of sparse and general nonlinear least squares approximations.
We carry out the theoretical analysis of convergence rates for weighted $\ell^p$ functions.
Moreover, we present representations of these sparse vectors (or sequences) in sparse low-rank formats and discuss the application to parametric PDEs.

This is not the first work that proposes the utilisation of sparsity in the component tensors of a tensor network.
In~\cite{grelier2019learning} and~\cite{michel2021learning}, the authors consider the abstract setting of empirical risk minimisation on bounded model classes of potentially sparse tensor networks.
They present model selection strategies for the network topology and sparsity pattern.
Due to the use of empirical risk minimisation, they obtain standard error bounds for arbitrary risk functions satisfying boundedness and Lipschitz continuity assumptions (cf.~\cite{michel2021learning}, which relies on approximation results from \cite{ali2021approximation,ali2023approximation}).
% In contrast to the RIP-based bounds in Proposition~\ref{prop:empirical_projection_error}, Corollary~\ref{cor:sparse_sample_complexity} and~\ref{cor:semi-sparse_sample_complexity}, the risk bound presented in these works manuscript work .
However, in the case of least squares risk, these strong assumptions restrict the application to bounded model classes.
Moreover, they do not guarantee an equivalence of errors, which translates to the error decreasing with a slow Monte Carlo rate.
This is intolerable when striving for small relative errors, which is often the case in numerical schemes.

In~\cite{Chevreuil2015sparseRank1} the authors propose an algorithm that computes a sparse best approximation in the model class of sparse rank-$1$ tensors.
Conceptually, this algorithm is very similar to our Algorithm~\ref{alg:sparse_als} but is restricted to a sum of unweighted sparse rank-$1$ tensors.
The restriction to a sum of rank-$1$ tensors implies a suboptimal convergence with respect to the rank and the use of unweighted sparsity means that, in the worst case, vastly more sample points may be required than are actually necessary.
% Contrary to our work, the authors in~\cite{Chevreuil2015sparseRank1} do not adapt the rank by adding small perturbations but by computing sparse rank-$1$ updates.
% Although it is known, that such a sum of best approximations can lead to a suboptimal rank (cf.~\cite{Stegeman_2010} and the references therein), convergence can be guaranteed (cf.~\cite{ehrlacher_2021_greedy_tt}).
% In contrast to Algorithm~\ref{alg:rals}, greedy algorithms do not require explicit rank adaptation, which simplifies the implementation and alleviates any concerns about stability.
% Moreover, since the representation of a rank-$1$ tensor is unique up to scaling factors, the regularization proposed in Section~\ref{sec:restricted_als} is unique.
% This holds the promise to combine the numerical performance of Algorithm~\ref{alg:rals} with the potential efficiency benefits of the greedy algorithm from~\cite{Chevreuil2015sparseRank1}.

A very similar ($s^2$-PGD) method is also proposed in~\cite{sancarlos2021pgdbased}.
The method optimises with regard to the same model class as our Algorithm~\ref{alg:sparse_als} but does not orthogonalise the component tensors in between the micro steps.
It is not clear if this lack of orthogonalisation can result in numerical instabilities as it would in the classical ALS method.
Moreover, the lack of orthogonalisation prevents the use of the correct weight sequence in the micro steps of their sparse ALS and does not allow for the same automatic rank adaptation as our algorithm.

Finally, block-sparse tensor networks are a well-known tool in the numerics of quantum mechanics~\cite{singh_2010_block_sparse_dmrg} and were recently introduced to the mathematics community by~\cite{bachmayr2022particle}.
This theory is already used in~\cite{goette2021blocksparse} to perform least squares regression in a model class of tensor trains restricted to subspaces of homogeneous polynomials of fixed degree.
The basis selection performed in our second algorithm is conceptually very similar to the restriction to eigenspaces used in~\cite{bachmayr2022particle} to ensure block-sparsity.
We believe that our algorithm can be interpreted as a generalisation of the regression on block-sparse tensor trains.
In contrast to this approach, where the sparsity structure has to be known in advance,
our algorithms explores the sparsity automatically.
%\todo{@ME/PT: Depending on the outcome of our experiments:
%- The downside of this is that the sparsity structure can not be interpreted as a restriction to a linear subspace of the ambient tensor space, which reduces the interpretability and increases the degrees of freedom.
%- We observe that this structure can be recovered by our algorithms.}
%}

\subsection{Weighted sparsity}
Sparse approximability of a function can be expressed by the $\ell^q$-summability of the coefficients sequence of its basis or frame representation.
The following central result, commonly attributed to Stechkin~\cite{DeVore_1998,COHEN_2011}, is a key tool to provide convergence rates for sparse approximations.
\begin{lemma}[Stechkin]
\label{lem:stechkin} 
    Let $0 < q < p \le \infty$ and let $\bfv\in\ell^q$.
    Define $J_n$ as the set of indices corresponding to the $n$ largest elements of the sequence $\abs{\bfv}$ and $P_{J_n}\bfv = \sum_{j \in J_n} \bfv_j \bfe_j$, where $\bfe_j$ is the sequence with $1$ at index $j$ and $0$ everywhere else.
    Then
    \begin{equation}
        \norm{\bfv - P_{J_n}\bfv}_{\ell^p} \le \pars{n+1}^{-s} \norm{\bfv}_{\ell^q},
        \quad s := \frac{1}{q} - \frac{1}{p} .
    \end{equation}
\end{lemma}
Given a normalized basis (or frame) $\braces{B_k}_{k\in\mbb{N}}$ of a Banach space $V$ of functions, each element $u$ of $V$ can be identified with its coefficient sequence $\bfu$ with respect to this basis.
Lemma~\ref{lem:stechkin} hence yields the convergence estimate for the best $n$-term approximation $u_n$ of $u$:
\begin{equation}
    \norm{u - u_n}_{V} \le \norm{\bfu - P_{J_n} \bfu}_{\ell^p} \le \pars{n+1}^{-s} \norm{\bfu}_{\ell^q}, \quad s=1/q - 1/p,
\end{equation}
for $p=1>q$ in the general Banach case, or $p=2>q$ when $V$ is a Hilbert space and $\braces{B_k}_{k\in\mbb{N}}$ is an orthonormal basis of $V$.
%For the Hilbert case and when $\braces{B_k}_{k\in\mbb{N}}$ is a basis of $V$, the above estimate can be improved into 
%\begin{equation}
%    \norm{u - u_n}_{V} = \norm{\bfu - P_{J_n} \bfu}_{\ell^2} \le \pars{n+1}^{-s} \norm{\bfu}_{\ell^q}, \quad s=1/q - 1/2.
%\end{equation}
%Assuming an orthonormal basis $\braces{B_k}_{k\in\mbb{N}}$ of $L^2(D)$ on some domain $D$, we can identify every function $u\in L^2$ with its coefficient sequence $\bfu$ with respect to this basis.
%This results in an isometry $\norm{u}_{L^2} = \norm{\bfu}_{\ell^2}$.
%For $q<2$, Lemma~\ref{lem:stechkin} hence yields the convergence estimate
%% \begin{align}
%%     \norm{u - u_n}_{L^\infty} \le \pars{n+1}^{1-1/q} \norm{u}_{\ell^q}
%%     \quad\text{ and }\quad
%%     \norm{u - u_n}_{L^2} \le \pars{n+1}^{1/2-1/q} \norm{u}_{\ell^q}.
%% \end{align}
%\begin{equation}
%    \norm{(I - P_{J_n}) u}_{L^2(D)} = \norm{(I - P_{J_n})\bfu}_{\ell^2} \le \pars{n+1}^{-s} \norm{\bfu}_{\ell^q} .
%\end{equation}
A disadvantage of the standard Stechkin's lemma is that it can only predict algebraic approximation rates, and these rates are suboptimal for some relevant classes of functions.

\paragraph*{Contributions}
To overcome these issues, we introduce for any sequence $\bfomega\in\bracs{0,\infty}^{\mbb{N}}$ the $\bfomega$-weighted sequence space
\begin{equation}
    \ell_\bfomega^q := \braces{\bfv\in\mbb{R}^{\mbb{N}} \,:\, \norm{\bfv}_{\ell^q_\bfomega} := \norm{\bfomega \bfv}_{\ell^q} < \infty} ,
\end{equation}
where $\bfomega \bfv = (\bfomega_\nu \bfv_\nu)_{\nu}$ denotes the element-wise multiplication of the two sequences $\bfomega$ and $\bfv$. 
With these spaces, a corresponding weighted version of Stechkin's lemma is derived, which enables to better exploit classical regularity in terms of convergence rates, significantly improving results of the classical lemma.
Indeed, in Section~\ref{sec:weighted_stechkin} we recall that the weighted $\ell^q_\bfomega$-spaces correspond to a variety of function spaces such as Barron, Besov and Sobolev spaces.
This makes possible to relate the summability $\bfu\in\ell^q$ directly to more natural regularity assumptions such as $u$ being in the Sobolev space $H^k$ (cf.~Example~\ref{ex:Hk}).
% For instance, Example~\ref{ex:Hk} shows a weighted summability property of the Hermite coefficients of functions in $H^k\pars{\mbb{R},\rho}$ for any $k\ge 3$.
%The weighted version of Stechkin's lemma
% in Theorem~\ref{thm:sparse_RIP_weighted}
%improves the rates over those of the unweighted version and achieves close to optimal bounds.
We compare the obtained results to previous works and discuss the relation of the weighted $\ell^p$ spaces to unweighted and monotone $\ell^p$ spaces (cf.~\cite{adcock_2022_isMCbad}).
If the employed weight sequence increases super-algebraically, the new weighted bound has a significantly faster decay than the corresponding unweighted bound.
Moreover, there are also improvements for algebraically increasing weight sequences, namely a significant reduction of multiplicative constants in the approximation estimate.
Because of this, the analysis provides immediate convergence bounds even in the non-asymptotic setting, i.e.\ for finite-dimensional linear spaces.

\subsection{Application to parametric partial differential equations}
\label{sec:model problem}

The numerical solution of high-dimensional parametric operator equations has become a highly active research field in the last decade, particularly in the area of Uncertainty Quantification (UQ) and in relation to modern (scientific) machine learning, see for instance \cite{cohen_devore_2015,schwab2011sparse} and references therein.
% In UQ one often models uncertainty in input parameters in terms of probability distributions.
% Hence, randomness in the coefficients of a partial differential equation may account for the fact that the true data is not known exactly in practice.
% Having a full parametric solution of the equation at hand, it is straightforward to (re)insert the probabilistic model and to compute statistical quantities such as the expected solution, covariances and higher order moments.
The parameter domain is often high- or even infinite-dimensional, making it computationally challenging to approximate the solution in linear spaces due to the CoD.
% However, while staying independent of the problem dimension, the classical Monte Carlo method for estimating the expected solution of a PDE with random input converges at best at a rate of $m^{-1/2}$ for $m$ sample evaluations.
% The theoretical and numerical treatment of parametric PDEs with a large number of deterministic or stochastic parameters has been the subject of intensive investigation in recent years, see for instance \cite{cohen_devore_2015,schwab2011sparse} and references therein.
% One principal objective is to design numerical approximation methods that alleviate the curse of dimensionality in the sense that approximation rates can be established even in the case of countably infinite many variables.
% This is of course only possible if the target functions exhibit certain properties that can be exploited in the reconstruction.
It hence is mandatory to exploit structural properties of the respective functions.
When relying on sparsity as we do, the required summability constraints can be deduced from smoothness.
Indeed, it is shown in Section~\ref{sec:weighted_stechkin} that weighted summability with algebraically increasing weight sequences can often be derived from standard regularity assumptions.
Morevover, certain assumptions on the data allow us to derive even stronger summability properties for the solution of parametric PDEs as shown e.g.\ in~\cite{Bachmayr_2016_I,Bachmayr_2016_II}.
% Strengthening the smoothness assumption to e.g.\ holomorphy allows us .
%The derivation however is often technical and elaborated.

% Note that the case of finitely many variables (i.e. $a_j = 0$ for $j > d$) typically leads to exponential convergence rates of the form $\exp(-cn^{1/d})$.
% Such rates were for instance obtained in~\cite{beck_2012_ppde_collocation,beck_2014_ppde} by application of the above Stechkin estimate for all $0 < q < 1$ and judiciously choosing the value of $q$ with respect to $n$.
% \todo{add collocation references}
% These values can also be obtained directly through the available upper bounds on the coefficients as shown in~\cite{tran_2017_ppde}.

% This work is devoted to the design and analysis of concrete numerical methods that provably achieve this goal in the challenging but also well-understood setting of linear elliptic partial differential equation (PDE).
% While our analysis can be applied to other models, the focus lies on the following prototypical linear second order problem:
We recall a prototypical parametric linear second order elliptic problem and its solution properties as a motivation for the results of this work.
For a given bounded Lipschitz domain $D\subseteq\mbb{R}^d$ with $d\in\mbb{N}$ and some source function $f\in L^2\pars{D}$, consider the linear elliptic PDE
\begin{equation}
\label{eq:darcy}
    \begin{aligned}
        -\operatorname{div}_{\!x}\pars{a\pars{x,y}\nabla_{\!x} u\pars{x,y}} &= f(x), &\mbox{in }D,\\
        u\pars{x,y} &= 0, &\mbox{on }\partial D,
    \end{aligned}
\end{equation}
where $y\in\mbb{R}^L$ is a high-dimensional ($L\in\mbb{N}$) or infinite-dimensional ($L=\infty$) parameter vector determining the coefficient field $a$ and hence the solution $u$.
With typical applications in modelling stochastic flow through porous media (such as groundwater flow~\cite{wang_1995_groundwater}), the diffusion coefficient is often defined by a Karhunen--Lo\`eve type expansion~\cite{todor2007convergence,lord2014introduction}, which can be constructed to represent random fields with bounded variance and typically takes the form
\begin{alignat}{4}
\label{eq:a_uniform}
    a(x,y) &= \sum_{j=1}^{L} a_j(x) y_j + a_0(x)
    &\qquad\text{with }
    y&\sim\mcal{U}\pars{\bracs{-1, 1}}^{\otimes L}
    \qquad\text{ or} \\
\label{eq:a_lognormal}
    \ln\pars{a(x,y)} &= \sum_{j=1}^{L} a_j(x) y_j
    &\qquad\text{with }
    y&\sim\mcal{N}\pars{0, 1}^{\otimes L} .
\end{alignat}
In these applications, the functions $a_j:D\to\mathbb{R}$ are scaled $L^2\pars{D}$-orthogonal eigenfunctions of the covariance operator of $a$ or $\ln\pars{a}$.
This specific choice is not necessary for the application of our theory and other more advantageous expansions (cf.~\cite{bachmayr_2017_notKL}) may be considered as well.
% Assuming that the coefficient $a\pars{y} \ge \check{a}\pars{y} > 0$ is bounded from below for every $y\in\mbb{R}^L$, Lax--Milgram theory allows us to define the solution $u(y)$ in the space $H^1_0(D)$ through the variational formulation
% \begin{equation}
% \label{eq:variational_form}
%     \int_D a(x,y)\nabla_{\!x} u(x,y)^\intercal \nabla_{\!x} v(x) \dx
%     = \int_D f(x)v(x) \dx
%     \qquad\text{for all}\qquad
%     v\in H^1_0\pars{D} .
% \end{equation}
% \todo{Define $u\pars{y} := u\pars{\bullet, y}$.}
% Moreover, the standard Lax-Milgram a-priori estimate tells us that
% \begin{equation}
%     \norm{u\pars{y}}_{H^1_0}
%     \le \check{a}\pars{y}^{-1} \norm{f}_{H^{-1}\pars{D}}
%     \le \check{a}\pars{y}^{-1} C_{\mathrm{P}} \norm{f}_{L^2\pars{D}},
% \end{equation}
% where $C_{\mathrm{P}}$ is the Poincar\'e constant of $D$.
% For an in-depth discussion of the well-posedness of~\eqref{eq:darcy} with coefficents~\eqref{eq:a_uniform} and~\eqref{eq:a_lognormal}, we refer to the works~\cite{schwab2011sparse,cohen_devore_2015,gittelson_2011_sparse_ppde,gittelson_2010_lognormal_SGFEM,galvis-sarkis,mugler_2013_SGFEM}.

% \subsection{Approximation results}
% \label{sec:intro approximation}

We recall some results from \cite{Bachmayr_2016_I,Bachmayr_2016_II} on the analysis and approximation of the parameter-to-solution map
\begin{equation}
    y\mapsto u(y) := u\pars{\bullet, y},
\end{equation}
induced by the model \eqref{eq:darcy}--\eqref{eq:a_lognormal}.
% One key observation is that the summability of the data transfers to the solution.
For~\eqref{eq:a_uniform}, the following result was shown recently.
\begin{theorem}[Theorem~3.1~in~\cite{Bachmayr_2016_I}]
\label{thm:Legendre_decay_Bachmayr}
    Consider problem~\eqref{eq:darcy} with affine coefficients~\eqref{eq:a_uniform}.
    Assume that there exists a sequence $\bfrho\in\pars{1,\infty}^L$ such that
    \begin{equation}
        \sup_{x\in D} \frac{1}{a_0} \sum_{j\in\bracs{L}} \bfrho_j\abs{a_j\pars{x}} < 1 .
    \end{equation}
    Then the map $y\mapsto u(y)$ belongs to $L^k\pars{ \bracs{-1, 1}^L, \dx[\gamma] ;  H^1_0\pars{D}}$ for all $k\in\mbb{N}\cup\braces{\infty}$, where $\gamma$ is the uniform measure.
    Hence, there exists an expansion of $u\pars{x,y} = \sum_{\nu\in\mcal{F}} u_\nu\pars{x} L_\nu\pars{y}$ in terms of Legendre polynomials $(L_\nu)_{\nu\in \mcal{F}}$, where $\mcal{F}$ is the set of multi-indices in $\mbb{N}^L$ with finite support.
    Moreover, the sequence of coefficients satisfies
    \begin{equation}
    \label{eq:weighted_coefs_Legendre}
        \sum_{\nu\in\mcal{F}} \bfomega_\nu^2\norm{u_\nu}_{H^1_0\pars{D}}^2 < \infty,
    \end{equation}
    with $\bfomega_\nu := \prod_{j\in\bracs{L}}\pars{2\nu_j+1}^{-1/2}\bfrho_j^{\nu_j}$.
\end{theorem}
For the case of unbounded Gaussian parameters~\eqref{eq:a_lognormal}, we recall the following result.
\begin{theorem}[Theorems~2.2,~3.3~and~4.2~in~\cite{Bachmayr_2016_II}]
\label{thm:Hermite_decay_Bachmayr}
    Consider the model~\eqref{eq:darcy} with affine coefficients~\eqref{eq:a_lognormal}.
    Assume that there exists an $r\in\mbb{N}$ and a sequence $\bfrho\in\pars{0,\infty}^L$ such that
    \begin{equation}
        \sup_{x\in D}\sum_{j\in\bracs{L}} \bfrho_j \abs{a_j\pars{x}}
        < \frac{\ln 2}{\sqrt{r}}
        \qquad\text{and}\qquad
        \sum_{j\in\bracs{L}} \exp\pars{-\bfrho_j^2}
        < \infty .
    \end{equation}
    % Then $\check{a}\pars{y} > 0$ for almost all $y\in\mbb{R}^L$ and $\mbb{E}\bracs{\check{a}^{-k}\pars{y}} < \infty$ for all $k\in\mbb{N}$.
    Then the map $y\mapsto u(y)$ belongs to $L^k\pars{\mbb{R}^L,\dx[\gamma] ; H^1_0\pars{D} }$ for all $k\in\mbb{N}$, where $\gamma$ is the Gaussian measure on $\mbb{R}^L$.
    Hence, there exists an expansion of $u\pars{x,y} = \sum_{\nu\in\mcal{F}} u_\nu\pars{x} H_\nu\pars{y}$ in terms of Hermite polynomials $(H_\nu)_{\nu\in \mcal{F}}$, where $\mcal{F}$ is the set of multi-indices in $\mbb{N}^L$ with finite support.
    Moreover, the sequence of coefficients satisfies
    \begin{equation}
    \label{eq:weighted_coefs_Hermite}
        \sum_{\nu\in\mcal{F}} \omega_\nu^2\norm{u_\nu}_{H^1_0\pars{D}}^2 < \infty,
    \end{equation}
    with $\bfomega_\nu := \prod_{j\in\bracs{L}} \pars{\sum_{l=0}^r \binom{\nu_j}{l}\bfrho_j^{2l}}^{1/2}$.
\end{theorem}

\paragraph*{Contributions}
Using the weighted sequence spaces and Stechkin's lemma, we propose an alternative method of proof for the summability of the solution of parametric PDEs in Section~\ref{sec:sparse approximation of parametric PDEs}.
We show that in the case of a single parameter ($L=1$), these summability properties already follow from the analyticity of the parameter to solution map.
The new derivation only relies on the weighted version of Stechkin's lemma and elementary techniques.
This results in similar bounds to those in Theorem~\ref{thm:Legendre_decay_Bachmayr} and~\ref{thm:Hermite_decay_Bachmayr} with exponential decay of the basis coefficients.

% {\color{MidnightBlue}
% \begin{itemize}
%     \item Introduce the general setting, motivating the use of weighted sparsity.
%     \item Shorten the following deliberations.
%     \item Just state the results of Theorems 3 and 4? (I.e.\ don't present them in the theorem format.) 
% \end{itemize}}

\subsection{Numerical methods for weighted sparse approximation using sparse tensor train format}

Equally important as the approximation error analysis is the availability of actual computational methods.
For a probability measure $\gamma$ on some set $Y$, let $\mcal{M} \subseteq L^2\pars{Y,\gamma}$ be a \emph{model class} of functions in which $u\in \mcal{V}$ should be approximated.
Defining the norms
\begin{equation}
\label{eq:norms}
    \norm{\bullet} := \norm{\bullet}_{L^2\pars{Y,\gamma}}
    \qquad\text{and}\qquad
    \norm{\bullet}_{w,\infty} := \norm{w^{1/2}\bullet}_{L^\infty\pars{Y,\gamma}},
\end{equation}
where $w$ is some positive  \emph{weight function}, 
the problem of determining the best-approximation of $u$ in $\mcal{M}$ can be formulated as
\begin{equation}
\label{eq:min}
    u_{\mcal{M}} \in \argmin_{v\in \mcal{M}}\ \norm{u - v} .
\end{equation}
% Since the energy norm corresponding to~\eqref{eq:variational_form} is usually not equivalent to the standard $L^2$-norm, $u_{\mcal{M}}$ is not related to the Galerkin projection.
% This can be resolved by changing the product measure $\gamma^{\otimes M}$ (due to an independence assumption of the data parameter distributions) either in the weak formulation of~\eqref{eq:darcy} or in the optimisation problem~\eqref{eq:min}.
% To resolve this issue, one may introduce a weight function $w$ satisfying $\norm{\bullet}_B \lesssim \norm{\bullet}_{L^2_w}$.
% For the parametric diffusion equation~\eqref{eq:darcy}, such a weight function is given by $w\pars{y} = \hat{a}\pars{y} = \norm{a\pars{y}}_{L^\infty\pars{D}}$.}
% Given an optimal set of basis functions, indexed by $\Lambda$, the coefficients of the best approximation are given by the …
% In the considered model problems, the integration requires…
Since the $L^2$-norm cannot be computed exactly in high-dimensional settings, a popular remedy is to introduce an empirical estimator with samples $\boldsymbol{y}:=\{y^i\}_{i=1}^n$ and the respective weighted least-squares minimisation, namely
% Given point-evaluations $\braces{\pars{y^i, u\pars{y^i}}}_{i=1}^n$ we want to find a (not necessarily unique) best approximation
\begin{equation} \label{eq:min_emp}
     u_{\mcal{M},n} \in \argmin_{v\in\mcal{M}}\ \norm{u - v}_{n}
    \quad\text{with}\quad
      \norm{v}_{n} := \pars*{\frac{1}{n}\sum_{i=1}^n w\pars{y^i} \abs{v\pars{y^i}}^2}^{1/2} .
\end{equation}
A natural choice is to take $w$ such that $\int_Y w^{-1} \dx[\gamma] = 1$, and to draw the points $y_i$
independently from  the measure $w^{-1}\gamma$ for all $i=1,\ldots, n$.
This approach results in approximations with guaranteed error bounds when assuming the \emph{restricted isometry property} (RIP) that is known from compressed sensing~\cite{CANDES2008589,adcock2022sparse}.
It is defined for a given set of functions $A$ by 
\begin{equation}\label{eq:rip}
    \operatorname{RIP}_A\pars{\delta} :\Leftrightarrow \pars{1-\delta}\norm{u}^2 \le \norm{u}_{n}^2 \le \pars{1+\delta}\norm{u}^2 \qquad \forall u\in A .
\end{equation}
If satisfied for a parameter $\delta\in\pars{0,1}$, the error of estimator~\eqref{eq:min_emp} can be bounded as follows.
\begin{proposition}[Theorem~3 in~\cite{trunschke_2022_convergence}]
\label{prop:empirical_projection_error}
    If $\operatorname{RIP}_{\braces{u_{\mcal{M}}}-\mcal{M}}\pars{\delta}$ holds, then
    \begin{equation}
        \norm{u - u_{\mcal{M},n}}
        \le \norm{u - u_{\mcal{M}}} + \tfrac{2}{\sqrt{1-\delta}} \norm{u - u_{\mcal{M}}}_{w,\infty}\ .
    \end{equation}
\end{proposition}

The assumption of $\operatorname{RIP}_{\braces{u_{\mcal{M}}} - \mcal{M}}\pars{\delta}$ is a weaker version of the standard assumption $\operatorname{RIP}_{\mcal{M} - \mcal{M}}\pars{\delta}$, which is often used when considering nested sequences $(\mcal{M}_r)_{r\ge 1}$ of model classes, like $r$-sparse vectors or rank-$r$ tensors, which satisfy  $\mcal{M}_r - \mcal{M}_r \subseteq\mcal{M}_{2r}$ . 
We show that such a nestedness property is also satisfied for the model classes of sparse low-rank tensors considered in this paper.

Because $\operatorname{RIP}_{\mcal{M}}\pars{\delta}$ is a random event, a sufficient number of sample points has to be used to guarantee that it holds true.
For theoretical reasons and since obtaining new sample points may be costly, a practical goal is to achieve this property with a minimal number $n$.
%Therefore, it seems natural to use least squares or compressive sensing methods for the reconstruction, for which quasi optimal recovery guarantees exists~\cite{cohen_2017_optimal,rauhut_2016_weighted_l1}.

To leverage the existing results from least squares methods~\cite{cohen_2017_optimal} in the development of numerical methods, one may rely on explicit bounds on the coefficients or a weighted summability property of the form $\bfu\in\ell^2_\bfomega$.
Given such a bound, we may define the sets $\Lambda_n$ corresponding to  the $n$ smallest weights and prove that
$$
    \norm{\pars{I-P_{\Lambda_n}}\bfu} \lesssim n^{-s},
$$
where $s$ relates to the summability of the sequence $\bfomega$.
An example of this can be seen in~\cite{cohen_2021_ppde}.
From a statistical point of view, this has the advantage that there exist bounds that guarantee that a small number of parameter evaluations is sufficient to result in a quasi-best sparse approximation with high probability.
Finding sets $\Lambda_n$ with the prescribed error bounds, however, relies on the knowledge of an exponentially increasing weight sequences $\bfomega$.
Such sequences do not exist for every PDE and their existence is not always easy to prove.
Moreover, due to the reliance on optimal sampling, the cited work also requires the ability to draw new samples from a problem adapted measure.

An alternative approach that mitigates these issues is the use of weighted sparsity~\cite{bouchot_2015_CSPG,rauhut_2016_cspg}.
Let $\bfv$ denote the sequence of coefficients of $v\in \mcal{V}$ with respect to a given basis.
Then the set of $\bfomega$-weighted $r$-sparse sequences is given by the ball
\begin{equation}
\label{eq:sparse_model_class}
    B_{\ell_\bfomega^0}\pars{0, r} = \braces{\bfv \,:\, \norm{\bfv}_{\ell^0_\bfomega}\le r} ,
\end{equation}
where the $\ell^0_\bfomega$-``norm'' is a generalisation of the standard $\ell^0$-``norm'' and is defined in Section~\ref{sec:weighted_stechkin}.
% Both ``norms'' coincide for the choice $\bfomega\equiv 1$.
In~\cite{rauhut_2016_weighted_l1} the authors show that a significantly improved bound for the probability of the RIP of $B_{\ell_\bfomega^0}\pars{0, r}$ can be derived when the $\ell^0$-``norm'' is replaced by its weighted version.
Although the shown a priori convergence rates still rely on weighted summability assumptions, the method itself does not.
The only requirement is an upper bound on the $L^\infty$-norm of the basis functions.
As a consequence, it can be applied easily in practice and is less reliant on the rate of decay of the sequence $\bfomega$.
The following theorem is a slight generalisation of this result.

% \begin{theorem}[]
% \label{thm:sparse_RIP}
% 	Fix parameters $\delta,\gamma\in\pars{0,1}$.
% 	Let $\braces{B_j}_{j\in\bracs{D}}$ be orthonormal with respect to the measure $\rho$, let $\beta_j \ge \norm{B_j}_{L^\infty}$ and fix
% 	\begin{equation}
% 		n \ge C \delta^{-2} s\max\braces{\log^3\pars{s}\log\pars{D}, -\log\pars{\gamma}} .
% 	\end{equation}
% 	Let $w\equiv 1$ and $y_1,\ldots, y_n$ be drawn independently from $\rho$.
% 	% and define $\norm{f}_{m}^2 := \sum_{i=1}^m f\pars{x_i}^2$ for any function $f\in L^2\pars{\rho}$.
% 	Then the probability of the event
% 	\begin{equation}
% 		\forall x\in\ell^2: \norm{x}_{\beta,0} \le s \Rightarrow \pars{1-\delta} \norm{B^\intercal x}_{L^2\pars{\pi}}^2 \le \norm{B^\intercal x}_{\boldsymbol{y}}^2 \le \pars{1+\delta} \norm{B^\intercal x}_{L^2\pars{\pi}}^2
% 	\end{equation}
% 	exceeds $1-\gamma$.
% \end{theorem}

\begin{theorem}
\label{thm:sparse_RIP_weighted}
	Fix parameters $\delta,p\in\pars{0,1}$.
	Let $\braces{B_j}_{j\in\bracs{D}}$ be orthonormal in $L^2\pars{Y,\gamma}$ and let $w\ge 0$ be any weight function satisfying $\norm{w^{-1}}_{L^1\pars{Y,\gamma}} = 1$.
	Assume the weight sequence $\bfomega$ in the definition \eqref{eq:sparse_model_class} of the model class $\mcal{M}$ satisfies $\bfomega_j \ge \norm{w^{1/2} B_j}_{L^\infty(Y,\gamma)}$ and fix
	\begin{equation}
		n \ge C \delta^{-2} r\max\braces{\log^3\pars{r}\log\pars{D}, -\log\pars{p}} .
	\end{equation}
	Let $y_1,\ldots, y_n$ be drawn independently from $w^{-1}\gamma$.
	% and define $\norm{f}_{w,m}^2 := \sum_{i=1}^m w\pars{x_i}f\pars{x_i}^2$ for any function $f\in L^2\pars{\rho}$.
	Then the probability of $\operatorname{RIP}_{B_{\ell^0_\bfomega}\pars{0,r}}\pars{\delta}$
    % the event
    % \begin{equation}
    % 	\forall x\in\ell^2: \norm{x}_{\ell^0_\beta} \le s \Rightarrow \pars{1-\delta} \norm{B^\intercal x}_{L^2\pars{\rho}}^2 \le \norm{B^\intercal x}_{\boldsymbol{y}}^2 \le \pars{1+\delta} \norm{B^\intercal x}_{L^2\pars{\rho}}^2
    % \end{equation}
	exceeds $1-p$.
\end{theorem}
\begin{proof}
    To make the weight function $w$ that is used explicit, we define $\norm{v}_{n,w} := \frac{1}{n} \sum_{i=1}^n w\pars{y_i} v\pars{y_i}^2$.
    Applying Theorem~5.2 from~\cite{rauhut_ward} to the $L^2\pars{Y,w^{-1}\gamma}$-orthonormal basis $\braces{\psi_j := w^{1/2} B_j}_{j\in\bracs{D}}$ shows that the probability of the event 
    \begin{equation}
		\forall v\in\ell^2: \norm{v}_{\ell^0_\bfomega} \le s
  \end{equation}
  exceeds $1-p$, which implies that 
  \begin{equation}
   \pars{1-\delta} \norm{\psi^\intercal v}_{L^2\pars{Y,w^{-1}\gamma}}^2 \le \norm{\psi^\intercal v}_{n,1}^2 \le \pars{1+\delta} \norm{\psi^\intercal v}_{L^2\pars{Y,w^{-1}\gamma}}^2
	\end{equation}
	holds with probability higher than $1-p$.
	The claim follows, since $\norm{\psi^\intercal v}_{L^2\pars{Y,w^{-1}\gamma}} = \norm{B^\intercal v}_{L^2\pars{Y,\gamma}}$ and $\norm{\psi^\intercal v}_{n,1} = \norm{B^\intercal v}_{n,w}$.
\end{proof}

% \rephrase{
% MERGE:
% for the coefficients and -- in the case of high-dimensional functions $u$ -- involves exploring a potentially large tensor-product basis.
% }
The index set $J_n$ that contains the $n$ largest coefficients of the function can in principle contain arbitrary large indices.
This is an issue for algorithmic realisations, where $J_n$ must be restricted to lie in a finite set of candidate indices $\Lambda$.
The set $\Lambda$ should be large enough to ensure that $J_n\subseteq\Lambda$ but not too large as to blow up the time complexity of the numerical algorithm, which scales at least linearly with $\abs{\Lambda}$.
Specifically, it has to be chosen carefully not to re-introduce the CoD.
Without assumptions on the summability of the coefficients, such a set is difficult to find.
For the model problems considered in this work, this is not a problem since appropriate candidate sets $\Lambda$ can be designed based on the summability conditions in Theorem~\ref{thm:Legendre_decay_Bachmayr}.
For other problems such conditions are not known, which impedes the application of compressed sensing algorithms.

\paragraph*{Contributions}

% An alternative to least squares methods is the compressed sensing approach.
% While it is less restrictive, it still requires the definition of a basis in which the solution can be approximated sparsely.
% And since the time complexity of the algorithms depends at least linearly on the size of this basis.
% Although this set has a manageable size of $\binom{M+d}{d} \sim \max\braces{\pars{1+\frac{d}{M}}^M, \pars{1+\frac{M}{d}}^d}$, it is inconvenient to construct in practice and does not take into account any anisotropy.
% But doing this would lead us essentially to the set $\Lambda_n$ and to the same drawbacks as discussed before.
Without prior knowledge, the exponentially large candidate set $\Lambda=\braces{0, \ldots, d-1}^M$ is a natural choice but classical algorithms for sparse approximation would yield a complexity polynomial in $\vert\Lambda \vert = d^M$, hence the CoD.
We propose to use tensor trains~\cite{oseledets_tensor-train_2011,holtz_alternating_2012} to alleviate this CoD.
Building upon results from~\cite{li_2022_sparse_tt}, we show that the best $n$-term approximation can be represented in a sparse tensor train format with rank $n$.
A sparse version of the ALS algorithm, which we call SALS, can be used to optimise over the sparse components. The complexity becomes polynomial in $n$ and $d$ and linear in $M$.
This allows almost the same sample complexity bounds as in Theorem~\ref{thm:sparse_RIP_weighted} while also admitting an admissible algorithmic realisation.
We emphasise that this new algorithm is a feasible alternative to sparse approximation algorithm, not for the approximation in low-rank tensor format.

% ==========================================================% Note that the generalised lasso $\norm{F - QC}_{\ell^2}^2 + \lambda\norm{QC}_{\ell^1}$ would promote stability directly but would suffer from the curse of dimensionality.
% ==========================================================

\subsection{Numerical methods for weighted sparse and low-rank tensor train approximation}

The results of Section~\ref{sec:sparse_als} provide an approach to express sparse tensors as TTs with a rank that is bounded by the number of nonzero entries of the component tensors.
We demonstrate that tensors with weighted sparsity are not only sparse but have also low rank, which is not exploited by the model class of sparse tensor trains from the previous section.
Due to the special structure of these sparse tensor trains, the basis of every core tensor is strongly overparameterised.
As a result, the linear systems arising in the microsteps of the SALS may become very large and the optimisation becomes very costly.
Another consequence of this overparameterisation is that the sample size that is required for an accurate microstep is larger than it would have to be if only a minimal basis would have been used.

\paragraph*{Contributions}
% we define a model class of sparse tensors and an adapted ALS algorithm in Section~\ref{sec:sparse approximation of parametric PDEs}.
% As a result, the representation of the tensor is no longer minimal.
% This leads to significantly larger linear systems in the microsteps of the ALS algorithm and pessimistic estimates regarding the sampling complexity.
% On the other hand, simply ignoring the sparsity would give us standard tensor and current RIP bounds for tensors only hold locally~\cite{trunschke21} or are very pessimistic~\cite{trunschke_2022_convergence}.
% The RIP result in Theorem~\ref{thm:sparse_RIP_weighted} requires sparse vectors.
% This can in principle be used in a sparse ALS algorithm~\ref{alg:sparse_als} as described in Section~\ref{sec:sparse approximation of parametric PDEs}.
% However, since the used spaces are not minimal, the variation constants can be way larger than necessary.
% This leads to significantly linear systems in the microsteps of the ALS and pessimistic estimates regarding the sampling complexity.
% Moreover, current RIP bounds for tensors only hold locally~\cite{trunschke21} or are very pessimistic~\cite{trunschke_2022_convergence}.
As a possible solution to the overparametrisation issue we propose to round the sparse tensor back to minimal rank.
Although this destroys the sparsity of the orthogonal component tensors, it retains the weighted sparsity of the core tensor.
This yields a new model class of sparse and low-rank tensor trains.
% This however optimises over the entire class of TTs for which the sample complexity is extremely pessimistic.
Investigating the probability of the RIP for this hybrid model class is the focus of Section~\ref{sec:semisparse_als}.
To do this, we show in Theorem~\ref{thm:rip_close_BA} that the RIP on $B_{\ell^0_\bfomega}\pars{0,r}$ induces a RIP for any model class that is close enough with respect to an appropriate distance.
This is a promising novel result that applies to any model class.
In particular, we show in Theorem~\ref{thm:RIP_MRromega} that our hybrid model class satisfies the conditions of Theorem~\ref{thm:rip_close_BA}.
It thereby inherits the RIP from sparse vectors that is guaranteed by the stability result in Theorem~\ref{thm:sparse_RIP_weighted}.
The results improve upon the previously developed theory for tensor reconstruction of solutions of high-dimensional parametric PDEs as presented in~\cite{trunschke21,trunschke_2022_convergence}.

% In particular this means that there exists a basis for the core space such that a weighted sparse core produces a full tensor which can be well-approximated by a weighted sparse tensor.
% This result mainly hinges on Theorem~\ref{thm:rip_close_BA}, which is valid in general and does not depend on the specific chosen model class.
% We consider this a quite promising novel result.

% \todo{Add a remark: It would be interesting to see if the reverse is also true: I.e.\ if a set satisfies the rip then it is close to a union of sparse sets?}
\section{A weighted version of Stechkin's lemma}
\label{sec:weighted_stechkin}

% The considered solutions of parametric PDEs as well as more general holomorphic functions exhibit a beneficial sparsity structure.
% Our goal is to exploit this structure in a reconstruction method with low-rank tensors.
% A central tool is the weighted Stechkin lemma that is presented in this section together with interesting properties and relations of weighted sequence spaces.

In what follows, we are concerned with coefficient sequences $\bfx$ indexed by a set $\Lambda$, which may be finite or countably infinite.
If not specified otherwise, we always assume that $\Lambda = \mathbb{N}$.
Since any operation defined on the coefficients can be extended to an element-wise operation defined on sequences, we write e.g.\ $\pars{\bfx\bfy}_j = \bfx_j\bfy_j$ as the element-wise product of two sequences $\bfx$ and $\bfy$ and $\pars{\bfx/\bfy}_j = \bfx_j/\bfy_j$ as the element-wise division.
For any such sequence $\bfx$ and any subset $J\subseteq \Lambda$, define $P_J \bfx$ via
\begin{equation}
    \pars{P_J \bfx}_j := \begin{cases}
        \bfx_j & j\in J\\
        0      & j\in\Lambda\setminus J .
    \end{cases}
\end{equation}
In other words, 
$P_J$ is the canonical projection onto the linear space $\operatorname{span}\braces{\bfe_j\,:\,j\in J}$, with $\bfe_j$ the canonical sequence having $1$ at index $j$ and $0$ everywhere else.
Moreover, let $\supp\pars{\bfx} := \braces{j\in\Lambda\,:\,\bfx_j\ne0}$ denote the support of $\bfx$.
To a vector $\bfomega \in \bracs{0,\infty}^\Lambda$ of weights we associate for $0<p\le\infty$ the weighted $\ell^p$ spaces
\begin{equation}
% \label{eq:ell_p_omega}
    \ell^p_{\bfomega}
    := \braces*{ \bfx \in \mbb{R}^\Lambda\ :\ \norm{\bfx}_{\ell^p_\bfomega} := \norm{\bfomega \bfx}_{\ell^p} < \infty} .
\end{equation}
Central to our analysis is the weighted $\ell_0$-``norm'' given by
\begin{equation}
    \norm{\bfx}_{\ell^0_\bfomega} = \sum_{j\in\supp\pars{\bfx}} \bfomega^2_{j},
\end{equation}
which counts the squared weights of the non-zero entries of $\bfx$.
% We also define the weighted cardinality of a set $S$ to be $\omega\pars{S} := \sum_{j \in S} \omega^2_{j}$.
% Since $\omega_{j} \geq 1$ by assumption, we always have $\omega(S) \geq |S|$, the cardinality of $S$.
When $\bfomega \equiv 1$, these weighted norms reproduce the standard $\ell^p$ norms.

In sparse approximation theory, the assumption $\bfv\in\ell^q$ is often central for the analysis.
However, it provides no guarantee for the position of the largest elements in the sequence.
For the purpose of numerical discretisation this is problematic since a truncation after the first $n$ terms of the sequence is not guaranteed to contain the largest elements.
Without an explicit bound on the decay of $\bfv$, no bounds for the discretisation error can be given.
We hence argue that it is natural to require an ordering of the terms that is induced by a weight sequence $\bfomega$.
Such a weighting exists for instance in the coefficients of solutions of parametric PDEs~\cite{Bachmayr_2016_I,Bachmayr_2016_II}.
Moreover, we will show later that such a weighting occurs naturally for many classical regularity classes like Sobolev and Besov spaces and for certain bases.

A very elegant proof of Stechkin's Lemma (Lemma~\ref{lem:stechkin}) is provided in~\cite[Lemma~3.6]{cohen_devore_2015}, which relies on a basic bound for the decay of any $\bfv\in\ell^q$ and an application of H\"older's inequality.
The same reasoning can be applied to obtain a proof for the Stechkin inequality in the weighted setting below.

\begin{lemma} \label{lem:weighted_stechkin}
    Let $0 < q < p \le \infty$ and $\bfalpha,\bfsigma,\bfomega\in\bracs{0,\infty}^{\mbb{N}}$ be sequences satisfying $\bfalpha^p = \bfsigma^{p-q} \bfomega^q$ (or $\bfalpha = \bfsigma$ in the case $p=\infty$).
	For a sequence $\bfv\in\mbb{R}^{\mbb{N}}$ with $\norm{\bfsigma \bfv}_{\ell^\infty} < \infty$ and $\norm{\bfomega \bfv}_{\ell^q} < \infty$ let $J_n$ be the set of indices corresponding to the $n$ largest elements of the sequence $\bfsigma\abs{\bfv}$.
	% Finally, define the linear space $V_n := \operatorname{span}\braces{e_j\,:\,j\in J_n}$ and let $P_{V_n}$ be the $\ell^2$-orthogonal projector onto this set.
	Then
    \begin{alignat}{2}
		\norm{ \bfv -P_{J_n}\bfv}_{\ell^p_\bfalpha}
		% \le \norm{P_{V_{n+1}} \pars*{\tfrac{\omega}{\sigma}}^{q}}_{\ell^{1}}^{-s} \norm{\omega v}_{\ell^q},
		&\le \norm{P_{J_{n+1}} \tfrac{\bfomega}{\bfsigma}}_{\ell^q}^{-sq} \norm{\bfv}_{\ell^q_\bfomega},
		&\quad
		s &:= \tfrac{1}{q} - \tfrac{1}{p} .
    \intertext{For $\bfalpha\equiv 1$ it holds that $\bfsigma = \bfomega^{q/(p-q)}$ and the inequality simplifies to}
		\norm{\bfv  -P_{J_n} \bfv}_{\ell^p}
		&\le \norm{P_{J_{n+1}} \bfomega}_{\ell^{1/s}}^{-1} \norm{\bfv}_{\ell^q_\bfomega},
        &\quad
		s &:= \tfrac{1}{q} - \tfrac{1}{p} .
    \end{alignat}
\end{lemma}
\begin{proof}
    We start by proving the assertion for $p=\infty$.
	Without loss of generality, we can assume that $v$ is ordered such that the sequence $\bfsigma\abs{\bfv}$ is decreasing.
	Under this assumption $J_n = \bracs{n}$.
    The choice $p=\infty$ implies $\bfalpha=\bfsigma$ and $\norm{\pars{I-P_{J_n}} \bfalpha \bfv}_{\ell^p} = \norm{\pars{I-P_{\bracs{n}}} \bfsigma \bfv}_{\ell^\infty} = \bfsigma_{n+1}\abs{\bfv_{n+1}}$.
    Now the bound $\bfsigma_{n}\abs{\bfv_{n}} \le \norm{P_{J_n} \frac{\bfomega}{\bfsigma}}_{\ell^q}^{-1} \norm{\bfomega \bfv}_{\ell^q}$ follows from
	\begin{align}
		\sum_{k=1}^{n} \pars*{\tfrac{\bfomega_{k}}{\bfsigma_{k}}}^q \pars{\bfsigma_{n} \abs{\bfv_{n}}}^q
		\le \sum_{k=1}^{n} \pars*{\tfrac{\bfomega_{k}}{\bfsigma_{k}}}^q \bfsigma_{k}^q \abs{\bfv_k^{\vphantom{q}}}^q
		= \sum_{k=1}^{n} \bfomega_k^q \abs{\bfv_k^{\vphantom{q}}}^q
		\le \norm{\bfomega \bfv}_{\ell^q}^q .
	\end{align}
	This proves the claim for $p=\infty$.
	The case $p<\infty$ can be reduced to $p=\infty$ using H\"older's inequality via
	% For $p < \infty$, H\"older's inequality implies
	\begin{equation} % \label{eq:hoelder_stechkin_weighted_alt}
	    \norm{\pars{I-P_{J_n}}\bfalpha \bfv}_{\ell^p}^p
	    = \norm{\pars{\pars{I-P_{J_n}}\bfsigma \bfv}^{p-q} \pars{\pars{I-P_{J_n}}\bfomega \bfv}^q}_{\ell^1}
	   % \le \norm{\pars{I-P_{J_n}}\sigma v}_{\ell^\infty}^{p-q} \norm{\pars{I-P_{J_n}}\omega v}_{\ell^q}^{q} .
	    \le \norm{\pars{I-P_{J_n}}\bfsigma \bfv}_{\ell^\infty}^{p-q} \norm{\bfomega \bfv}_{\ell^q}^{q} .
	\end{equation}
	The claim follows by using the weighted Stechkin bound for the factor $\norm{\pars{I-P_{J_n}}\bfsigma \bfv}_{\ell^\infty}
	\le \norm{P_{J_{n+1}} \frac{\bfomega}{\bfsigma}}_{\ell^q}^{-1}\norm{\bfomega \bfv}_{\ell^q}$.
% 	By H\"older's inequality it holds that
% 	\begin{equation}\label{eq:hoelder_stechkin_weighted_alt}
% 	    \norm{\pars{I-P_{J_n}}\alpha v}_{\ell^p}^p
% 	    = \norm{\pars{\pars{I-P_{J_n}}\sigma v}^{p-q} \pars{\pars{I-P_{J_n}}\omega v}^q}_{\ell^1}
% 	    \le \norm{\pars{I-P_{J_n}}\sigma v}_{\ell^\infty}^{p-q} \norm{\pars{I-P_{J_n}}\omega v}_{\ell^q}^{q} .
% 	\end{equation}
% 	Without loss of generality, we can assume that $v$ is ordered such that the sequence $\sigma\abs{v}$ is decreasing.
% 	Then the definition of $J_n$ implies $\norm{\pars{I-P_{J_n}} \sigma v}_{\ell^\infty} = \sigma_{n+1}\abs{v_{n+1}}$.
% 	To bound $\sigma_{n+1}\abs{v_{n+1}} \le \norm{P_{\bracs{n+1}} \frac{\omega}{\sigma}}_{\ell^q}^{-1} \norm{\omega v}_{\ell^q}$, observe that
% 	\begin{align}
% 		\sum_{k=1}^{n+1} \pars*{\tfrac{\omega_{k}}{\sigma_{k}}}^q \pars{\sigma_{n+1} \abs{v_{n+1}}}^q
% 		\le \sum_{k=1}^{n+1} \pars*{\tfrac{\omega_{k}}{\sigma_{k}}}^q \sigma_{k}^q \abs{v_k^{\vphantom{q}}}^q
% 		= \sum_{k=1}^{n+1} \omega_k^q \abs{v_k^{\vphantom{q}}}^q
% 		\le \norm{\omega v}_{\ell^q}^q .
% 	\end{align}
% 	Substitution into equation~\eqref{eq:hoelder_stechkin_weighted} yields the result.
\end{proof}

The preceding lemma is a weighted generalisation of Stechkin's Lemma~\ref{lem:stechkin} with which it coincides for the choice $\bfalpha=\bfsigma=\bfomega\equiv 1$.
In this setting, the parameter $q$ has to be chosen as small as possible to exploit the decay of the sequence $\bfv$ and increase the rate of convergence $s$.
When using the weighted Stechkin estimate of Lemma~\ref{lem:weighted_stechkin} this is not necessary since the decay of the sequence can be measured by means of the sequence $\bfomega$.

To get a better intuition of the derived results, note that each of the sequences $\bfalpha$, $\bfsigma$ and $\bfomega$ controls a different aspect of the estimate.
The sequence $\bfalpha$ determines how the truncation error is measured, $\bfsigma$ controls the truncation strategy and $\bfomega$ measures the decay of the sequence.
However, due to the constraint $\bfalpha^p = \bfsigma^{p-q}\bfomega^q$ only two of these sequences can be chosen freely.
Typically, these are $\bfalpha$ and $\bfomega$.

In the remainder of this section, the roles of the different parameters that occur in Lemma~\ref{lem:weighted_stechkin} are discussed with illustrative examples.
We start with an examination of $p$ and $\bfalpha$, which should be chosen to obtain an appropriate error norm $\norm{\bullet}_{\ell^p_\bfalpha}$ as in the following four examples.

\begin{example}[Sobolev and spectral Barron spaces on the torus]
\label{ex:1d-spaces}
    Suppose that $f$ is a function on the $1$-torus $\mbb{T}$ and let $v$ be its sequence of Fourier coefficients.
    Then $p=1$ and $p=2$ together with the weight sequence $\bfalpha(k)_j := \pars{1+j^2}^{k/2}$ provide a natural choice of parameters since the Sobolev and spectral Barron norms (cf.~\cite{chen_2022_spectral_barron}) of $f$ are then defined by
    \begin{equation}
        \norm{f}_{H^k\pars{\mbb{T}}} := \norm{\bfv}_{\ell^2_{\bfalpha(k)}}
        \qquad\text{and}\qquad
        \norm{f}_{B^k\pars{\mbb{T}}} := \norm{\bfv}_{\ell^1_{\bfalpha(k)}}\qquad \text{for any }k\in\mbb{R} .
    \end{equation}
\end{example}

\begin{example}[Sobolev and Besov spaces]
\label{ex:sparse_grids}
    Consider the Sobolev space $W^{k,p}$ of functions defined on the interval $[0,1]$ equipped with the Lebesgue measure, with $k\ge 1$ and $1 \le p \le \infty$.  
    A natural basis for this space is the hierarchical spline basis of degree $k$.
    %If a function $v\in W^{k,p}\pars{\lambda}$ is suitably represented with respect to such a basis, 
    It can be shown (see e.g.~\cite{ali_2021_part3}) that for any $v\in W^{k,p}$,
    $$
        \norm{v}_{W^{k,p}}
        \asymp \norm{\bfv}_{\ell^p_{\bfomega\pars{k}}}
        \qquad\text{with}\qquad
        \bfomega\pars{k}_{\ell, j} := 2^{k\ell}.
    $$
    A simple example for such a basis is provided in Appendix~\ref{app:sparse_grid_example}.
    These results can be extended to the wider class of Besov spaces $B^k_q\pars{L^p}$ for $0 < p = q \le \infty$ \cite{ali_2021_part3,leisner_2003_besov}.
\end{example}
    \begin{example}   
    Another useful choice of $p$ and $\bfalpha$ can be made when $f\in L^\infty\pars{\mcal{X},\gamma}$ for any measurable set $\mcal{X}$ and probability measure $\gamma$.
    Let $\bfv$ be the sequence of coefficients of $f$ with respect to the basis $\braces{B_j}_{j\in\mbb{N}}$ and define the sequence $\bfalpha_j := \norm{B_j}_{L^\infty\pars{\mcal{X},\gamma}}$.
    Then, by triangle inequality,
    \begin{equation}
        \norm{f}_{L^p\pars{\mcal{X},\gamma}} \le \norm{f}_{L^\infty\pars{\mcal{X},\gamma}} \le \norm{\bfv}_{\ell^1_\bfalpha} .
    \end{equation}
    By choosing weights so that $\bfalpha_j:= \norm{B_j}_{L^\infty\pars{\mcal{X},\gamma}} + \norm{B_j'}_{L^\infty\pars{\mcal{X},\gamma}}$, one may also arrive at bounds of the form $\norm{f}_{L^\infty\pars{\mcal{X},\gamma}} + \norm{f'}_{L^\infty\pars{\mcal{X},\gamma}} \le \norm{\bff}_{\ell^1_\bfalpha}$, reflecting how steeper weights encourage more smoothness (cf.~\cite{rauhut_2016_weighted_l1}).
    This bound for instance is used in the proof of Theorem~6.1 in~\cite{adcock_2022_isMCbad}, which provides dimension independent convergence rates for unweighted least squares approximation in high dimensions.
    However, the proof relies on a suboptimal weighted version of Stechkin's lemma, which we discuss further in Section~\ref{sec:relation_to_rauhut2016weighted}. %and is hamstrung by the condition $k \ge \norm{P_\Lambda \alpha}_{\ell^\infty}$ ($\Lambda$ being the necessarily finite set of candidates in the sparse regression).
\end{example}

\begin{example}
    Another useful application of Lemma~\ref{lem:weighted_stechkin} is given by the choice $p=\infty$ and $\bfalpha\equiv 1$.
    The choice $p=\infty$ requires $\bfsigma=\bfalpha\equiv 1$ and since $\bfalpha\equiv 1$, the bound simplifies to
    $$
        \norm{(I - P_{J_n}) \bfv}_{\ell^\infty} \le \norm{P_{J_{n+1}} \bfomega}_{\ell^q}^{-1} \norm{\bfv}_{\ell^q_\bfomega} .
    $$
    Replacing $\bfv$ by the monotonisation
    $$
        \bfv^{\mathrm{min}}_k := \min_{j\le k} \abs{\bfv_j}
        \quad\text{for all}\quad
        k\in\mbb{N}
    $$
    yields $J_n = [n]$, $\norm{(I - P_{[n]}) \bfv^{\mathrm{min}}}_{\ell^\infty} = \bfv^{\mathrm{min}}_{n+1}$ and $\norm{\bfv^{\mathrm{min}}}_{\ell^q_\bfomega} \le \norm{\bfv}_{\ell^q_\bfomega}$ and simplifies the bound even further to
    $$
        \bfv^{\mathrm{min}}_{n+1}
        \le \norm{P_{[n+1]} \bfomega}_{\ell^q}^{-1} \norm{\bfv}_{\ell^q_\bfomega} .
    $$
\end{example}

Next, we examine the parameters $q$ and $\bfomega$ and illustrate the benefits of the weighted version of Stechkin's lemma in terms of convergence.
The subsequent two examples aim to provide an intuition for the choice of $q$ and $\bfomega$, which should be chosen to capture the asymptotic decay of the sequence $\bfv$ in the reference norm $\norm{\bullet}_{\ell^q_\bfomega}$.
% Given $\bfalpha$ and $\bfomega$, the sequence $\bfsigma = \pars{\frac{\bfalpha^p}{\bfomega^q}}^{1/\pars{p-q}}$ is uniquely defined and determines how the coefficients are selected.
% Moreover, the fraction $\frac{\bfomega}{\bfsigma}$ characterises the excess regularity.
% According to $\bfomega$, this regularity is exhibited by $\bfv$ but it is not used to measure the error according to $\bfalpha$, hence leading to a sub-optimal convergence rate.

\begin{example}[The choice of $q$ and $\bfomega$ for algebraic decay]
\label{ex:q_omega_algebraic_decay}
    Consider the algebraically decaying sequence $\bfv_j = j^{-\rho_{\mathrm{alg}}}$ for some $\rho_{\mathrm{alg}}> 1$.
    To compare Lemma~\ref{lem:stechkin} and Lemma~\ref{lem:weighted_stechkin}, let $\bfalpha \equiv 1$ and $0 < q < p \le \infty$ be arbitrary but fixed.
    Moreover, define $\bar{q} := \frac1q$ and $\bar{p} := \frac1p$.
    Then Lemma~\ref{lem:s-harmonic-rate} provides the equivalence
    \begin{equation}
        \norm{\pars{1-P_{J_n}} \bfv}_{\ell^p}
        \sim \pars{n+1}^{-\pars{\rho_{\mathrm{alg}} - \bar{p}}} .
    \end{equation}
    
    This rate is a benchmark against which both versions of Stechkin's lemma can be compared.
    By Lemma~\ref{lem:s-harmonic-rate} it holds that
    \begin{equation}
        \frac{\bar{q}}{\rho_{\mathrm{alg}} - \bar{q}}
        \le \norm{\bfv}_{\ell^q}^q
        \le \frac{\rho_{\mathrm{alg}}}{\rho_{\mathrm{alg}} - \bar{q}}
    \end{equation}
    for any $\bar{q} \in \pars{\bar{p}, \rho_{\mathrm{alg}}}$.
    Stechkin's lemma thus yields the bound
    % This implies that $v\in\ell^q$ if and only if $\bar{q} := 1/q < \rho_{\mathrm{alg}}$, leading to the bound
    \begin{equation}
        \norm{\pars{1-P_{J_n}}\bfv}_{\ell^p}
        \le \pars{n+1}^{-s} \norm{\bfv}_{\ell^q}
        \le \pars{n+1}^{-\pars{\bar{q}-\bar{p}}} \pars*{\frac{\rho_{\mathrm{alg}}}{\rho_{\mathrm{alg}} - \bar{q}}}^{\bar{q}}.
    \end{equation}
    As $\bar{q}$ approaches the upper bound $\rho_{\mathrm{alg}}$, the predicted rate of convergence approaches the optimal rate $\rho_{\mathrm{alg}}-\bar{p}$.
    However, at the same time the factor $\norm{\bfv}_{\ell^q}$ diverges to infinity.
    This makes the bound only useful for large $n$ or small $\bar{q}$, i.e.\ small $s = \bar{q} - \bar{p}$.
    
    Although a different weight sequence $\bfomega$ cannot provide a faster rate of convergence, it can change the asymptotic constant.
    To this end we define the algebraically increasing sequence $\bfomega_j = j^{r}$ for some $r<\rho_{\mathrm{alg}}-\bar{q}$.
    % % \rephrase{Since we were already able to recover the optimal rate of convergence, we can not expect to obtain faster rates by using the weighted version.
    % % We can however hope for a better asymptotic constant.}
    % To measure the decay of $v$, we define the algebraically increasing sequence $\omega_j = j^{r}$ for some $r<\rho_{\mathrm{alg}}-1/q$.
    % % Since $v$ decays algebraically it seems natural to measure this decay of the sequence with an algebraically decaying sequence $\omega_j = j^{r}$.}

    This choice implies $\bfsigma_j = j^{-r/\pars{sp}}$.
    % s == 1/q - 1/p == (p - q)/(pq)
    % a^p == s^{p-q} o^q --> (a=1) s^{q-p} == o^q --> s == o^{q/(q-p)} == o^{-1/(sp)}
    % Test: j^{-r/(sp)(p-q)} j^{rq} == j^{-rq} j^{rq} == 1
    %       (p-q)/(sp) == (p-q)/(pq) q/s == q
    %, for $r < \rho_{\mathrm{alg}} - 1/q$, 
    The technical Lemmas~\ref{lem:s-harmonic-partial} and~\ref{lem:s-harmonic-rate} in the appendix yield the bounds
    \begin{equation}
        \norm{P_{J_{n+1}} \tfrac{\bfomega}{\bfsigma}}_{\ell^q}^{q}
        % = \sum_{j=1}^{n+1} j^{\pars{1 + 1/\pars{sp}}qr}
        % s == 1/q - 1/p --> sp == p/q - 1 --> sp + 1 == p/q
        % --> 1 + 1/sp == (sp + 1)/sp == (p/q)/(sp) == p/(spq) == p/(p - q)
        % = \sum_{j=1}^{n+1} j^{r/s}
        \ge \frac{\pars{n+1}^{r/s+1}}{r/s+1}
        \qquad\text{and}\qquad
        \frac{\bar{q}}{\rho_{\mathrm{alg}}-r-\bar{q}}
        \le \norm{\bfv}_{\ell^q_\bfomega}^q
        % = \sum_{j=1}^\infty j^{\pars{r-\rho_{\mathrm{alg}}}q}
        \le \frac{\rho_{\mathrm{alg}} - r}{\rho_{\mathrm{alg}} - r - \bar{q}} .
    \end{equation}
    Applying Lemma~\ref{lem:weighted_stechkin}, we obtain the bound
    \begin{equation}
        \norm{\pars{1-P_{J_n}} \bfv}_{\ell^p}
		\le \norm{P_{J_{n+1}} \tfrac{\bfomega}{\bfsigma}}_{\ell^q}^{-sq} \norm{\bfv}_{\ell^q_\bfomega}
        \le \frac{\pars{n+1}^{-\pars{r+s}}}{\pars{r/s+1}^{-s}} \pars*{\frac{\rho_{\mathrm{alg}} - r}{\rho_{\mathrm{alg}} - r - \bar{q}}}^{\bar{q}} .
    \end{equation}
    As in the unweighted case, the factor $\norm{\bfv}_{\ell^q_\bfomega}$ diverges as $r$ increases or $q$ decreases.
    But in contrast to the unweighted case, we can choose different values for $q$ than in the classical Stechkin estimate while maintaining the same rate of convergence.
    % there exists an entire family of parameters $r$ and $q$ that achieve a given rate of convergence $s < \rho_{\mathrm{alg}} - \bar{p}$.
    % This allows to choose % a parameter combination which minimises the asymptotic constant.
    % While difficult to state this analytically, it is actually rather simple to perform numerically.
    % The weight sequence $\omega$ allows us to choose
    Denote by $\tilde{q}$ the value of $\bar{q}$ chosen in the standard Stechkin estimate and recall that $\tilde{q} < \rho_{\mathrm{alg}}$.
    We can hence choose $r = \tilde{q} - \bar{q}$ and take the limit $q\to p$, leading to the estimate
    % Since $r+s = \tilde{q}-\bar{q}+\bar{q}-\bar{p} = \tilde{q} - \bar{p}$, $r+\bar{q} = \tilde{q}$, $\lim_{q\to p} s = 0$, $\lim_{q\to p}\bar{q} = \bar{p}$ and $\rho_{\mathrm{alg}} - \tilde{q} + \bar{p} \ge 0 + \bar{p}$, it follows that
    \begin{equation}
        \norm{\pars{1-P_{J_n}}\bfv}_{\ell^p}
        % \le \pars{n+1}^{-\pars{\tilde{q}-\bar{p}}} \pars*{\frac{\rho_{\mathrm{alg}} - \tilde{q} + \bar{p}}{\rho_{\mathrm{alg}} - \tilde{q}}}^{\bar{p}}
        \le \pars{n+1}^{-\pars{\tilde{q}-\bar{p}}} \pars*{\frac{\bar{p}}{\rho_{\mathrm{alg}} - \tilde{q}}}^{\bar{p}} .
    \end{equation}
    Compared to the unweighted case, the asymptotic constant in the weighted case is significantly smaller.
    A comparison of these constants is given in Figure~\ref{fig:w_bound}.
    These constants makes the Stechkin bound viable even for small values of $n$.
    % \rephrase{However, this is not surprising, given the fact that the unweighted Stechkin was already able to recover the precise rate of convergence and the weighting essentially only enforces an ordering in the coefficients.}
    % This is also not surprising, since the weighted $\ell^q$-norm just diverges in this case.
    % The more interesting question is: If we fix an $s_1$ in the unweighted case, can we find $s_2$ and $r$ such that $s_1 = s_2 + r$ but the asymptotic constant is significantly smaller?
    % This is difficult to compare on paper but easy to implement and the following comparison shows an astonishing improvement.
    % To choose $s$, we fix $p$ and choose $q_1$ and $q_2$.
    % Assume that $p > 1$. Then we can choose $q_2 = p/\pars{p-1}$ such that $s_2 = 1$.
    % In this case we must choose $r = s_1 - 1$, which only works if $s_1 > 1$, i.e.\ $q_1 < p/\pars{p+1}$.
    % Then we obtain
    % \begin{equation}
    %     \pars{r/s_2+1}^{s_2}
    %     \pars*{\frac{\pars{\rho_{\mathrm{alg}}-r}q_2}{\pars{\rho_{\mathrm{alg}}-r}q_2-1}}^{1/q_2}
    %     \pars*{\frac{\rho_{\mathrm{alg}}q_1}{\rho_{\mathrm{alg}}q_1 - 1}}^{-1/q_1}
    %     \le s_1 \ldots
    % \end{equation}
\end{example}

%\todo{ME Define the notation $\norm{\pars{1-P_{J_n}} v}_{\ell^p} \sim \pars{n+1}^{-\pars{\rho_{\mathrm{alg}} - p^{-1}}}$ as $\lesssim$ and $\gtrsim$ at the same time.}

\begin{figure}
    \centering
    \includegraphics[width=\textwidth]{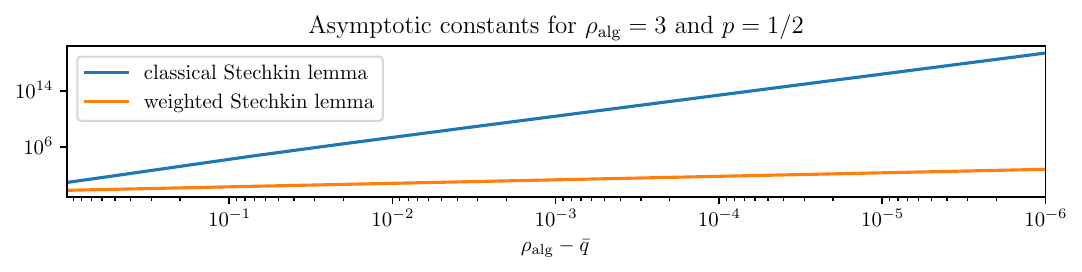}
    \caption{Behaviour of the asymptotic constants of the classical and weighted Stechkin estimates as the desired rate of convergence $s$ approaches the optimal rate $s^* := \rho_{\mathrm{alg}} - p^{-1}$.}
    \label{fig:w_bound}
\end{figure}
% \makeatletter
% fontsize: \f@size
% \makeatother
% width: \the\textwidth
% % \printinunitsof{in}\prntlen{\textwidth}

\begin{example}[The choice of $q$ and $\bfomega$ for exponential decay]
\label{ex:q_omega_exponential_decay}
    Consider the exponentially decaying sequence $\bfv_j = \rho_{\mathrm{exp}}^{j-1}$ with $\rho_{\mathrm{exp}}\in\pars{0,1}$.
    To compare Lemma~\ref{lem:stechkin} and Lemma~\ref{lem:weighted_stechkin}, let $\bfalpha \equiv 1$ and $0 < q < p \le \infty$ be arbitrary but fixed.
    Moreover, define $\bar{q} := \frac1q$ and $\bar{p} := \frac1p$.
    Then % the exact rate of convergence is given by
    \begin{equation}
        \norm{\pars{1-P_{J_n}} \bfv}_{\ell^p} = \rho_{\mathrm{exp}}^n \pars{1-\rho_{\mathrm{exp}}^p}^{-\bar{p}} .
    \end{equation}
    This rate is a benchmark against which both versions of Stechkin's lemma can be compared.
    The classical Stechkin lemma yields the bound
    \begin{equation}
        \norm{\pars{1-P_{J_n}} \bfv}_{\ell^p}
        \le \pars{n+1}^{-s} \norm{\bfv}_{\ell^q}
        = \pars{n+1}^{-s} \pars{1-\rho_{\mathrm{exp}}^q}^{-\bar{q}} , \quad \text{with } s= \bar q - \bar p.
    \end{equation}
    Notably, even though the factor $\norm{\bfv}_{\ell^q}$ does no longer impose a lower limit on $q$, the optimal exponential rate of convergence cannot be recovered.
    Moreover, the asymptotic constant still grows without bounds when $q$ decreases.
    This illustrates that Lemma~\ref{lem:stechkin} cannot fully exploit the decay of the sequence, which renders the estimates only useful as an asymptotic statement or for small values of $s$.

    To compare the preceding bound with the weighted bound of Lemma~\ref{lem:weighted_stechkin}, we choose the exponentially growing weight sequence $\bfomega_j = r^{-\pars{j-1}}$ for some $r\in\pars{\rho_{\mathrm{exp}} ,1}$.
    This choice implies $\bfsigma_j = r^{\pars{j-1}/\pars{sp}}$ and consequently
    \begin{equation}
        \norm{P_{J_{n+1}} \tfrac{\bfomega}{\bfsigma}}_{\ell^q}^{q}
        = \frac{r^{-\pars{n+1}/s}-1}{r^{-1/s}-1}
        \ge r^{-n/s}
        \qquad\text{and}\qquad
        \norm{\bfv}_{\ell^q_\bfomega}^q
        = \frac{1}{1 - \pars{\rho_{\mathrm{exp}}/r}^q} .
    \end{equation}
    Lemma~\ref{lem:weighted_stechkin} then yields
    \begin{equation}
        \norm{\pars{1-P_{J_n}} \bfv}_{\ell^p}
		\le \norm{P_{J_{n+1}} \tfrac{\bfomega}{\bfsigma}}_{\ell^q}^{-sq} \norm{\bfv}_{\ell^q_\bfomega}
        % \le \pars*{\frac{r^{-\pars{n+1}/s}-1}{r^{-1/s}-1}}^{-s}
        % \pars{1 - \pars{\rho_{\mathrm{exp}}/r}^q}^{-1/q}
        \le r^n \pars{1 - \pars{\rho_{\mathrm{exp}}/r}^q}^{-\bar{q}}.
    \end{equation}
    This shows that in contrast to the classical Stechkin inequality, the weighted Stechkin inequality can actually recover an exponential rate of convergence.
    This rate is even independent of $q$ but the asymptotic constant grows without bounds when $r$ approaches $\rho_{\mathrm{exp}}$.
\end{example}

\subsection{Relation to previous results}
\label{sec:relation_to_rauhut2016weighted}

Lemma~\ref{lem:weighted_stechkin} is not the first extension of Stechkin's lemma to the weighted case.
Another extension was proposed in~\cite{Rauhut2016weighted_l1}, which we briefly recall.
For a fixed parameter $p$ and sequence $\tilde{\bfomega}$
% instead of considering the weighted $n$-term approximation error $\norm{\pars{1-P_{J_n}}\tilde\omega v}_{\ell^p}$,
they consider the \emph{weighted $r$-sparse approximation error}
\begin{equation}
    \sigma_r\pars{v}_{\ell^p_\bfalpha}
    % := \min_{\substack{I\subseteq\mbb{N}\\\norm{P_I \tilde\omega}_{\ell^2}^2 \le r}}
    % \norm{\pars{1 - P_I} \tilde\omega^{\pars{2-p}/p}v}_{\ell^p} .
    := \min_{\substack{I\subseteq\mbb{N}\\
    \tilde{\bfomega}\pars{I} \le r}}
    \norm{\pars{1 - P_I} \bfv}_{\ell^p_\bfalpha},
    \qquad
    \tilde\bfomega := \bfalpha^{p/\pars{2-p}},
\end{equation}
where for any subset $I\subseteq\mbb{N}$ the weighted cardinality is defined by $\tilde\bfomega\pars{I} := \sum_{i\in I}\tilde\bfomega_i^2$.
To bound this error, for a given sequence $\bfv$ and a threshold $r>0$ the set of indices $J_n$ corresponding to the $n$ largest elements of the sequence $\tilde\bfomega\abs{\bfv}$ is considered.
Moreover, let
$
% \begin{equation}
    % n\pars{r} := \max\braces{n\in\mbb{N} : \norm{P_{J_n}\tilde\bfomega}_{\ell^2}^2 \le r} %,
    n\pars{r} := \max\braces{n\in\mbb{N} : \tilde{\bfomega}\pars{J_n} \le r} %,
    % \qquad\text{and}\qquad
    % \tilde\sigma_r\pars{\bfv}_{\tilde\bfomega,p} := \norm{\pars{1 - P_{J_{n\pars{r}}}} \tilde\bfomega^{\pars{2-p}/p}\bfv}_{\ell^p} .
% \end{equation}
$.
By maximality of $n\pars{r}$, it holds that $\tilde{\bfomega}\pars{J_{n\pars{r}+1}} > r$.
Consequently,
\begin{equation} \label{eq:weighted_r-sparse_bound}
    \sigma_r\pars{\bfv}_{\ell^p_\bfalpha}
    \le \norm{\pars{1 - P_{J_{n\pars{r}}}} \bfv}_{\ell^p_\bfalpha}.
 %    \tilde\sigma_r\pars{\bfv}_{\bfomega,p}
	% \le r^{-s}
	% \norm{\tilde{\bfomega}^{\pars{2-q}/q} \bfv}_{\ell^q},
	% \quad s := \frac{1}{q} - \frac{1}{p} .
\end{equation}
Using the unweighted version of Stechkin's lemma, it is concluded in~\cite[Theorem~3.2]{rauhut_2016_weighted_l1} that
\begin{equation}
    \sigma_r\pars{\bfv}_{\ell^p_\bfalpha}
    \le \pars{r-\norm{\tilde{\bfomega}}_{\ell^\infty}^2}^{-s}\norm{\tilde{\bfomega}^{\pars{2-q}/q}}_{\ell^q},
	\quad s := \frac{1}{q} - \frac{1}{p},
\end{equation}
for all $0 < q < p \le 2$ and $r > \norm{\tilde{\bfomega}}_{\ell^\infty}^2$.
The main weakness of this statement comes from the requirement $r > \norm{\tilde{\bfomega}}_{\ell^\infty}^2$.
This condition requires that the sequence $\tilde{\bfomega}$ is bounded and implies at the same time that the bound only applies for a large threshold $r$, i.e.\ asymptotically.
% By Lemma~3.1 in~\cite{rauhut_2016_weighted_l1}, it holds that $\sigma_r\pars{v}_{\tilde{\omega},p} \ge \norm{\pars{1-P_{J_{n\pars{3r}}}} \tilde\omega^{\pars{2-p}/p} v}_{\ell^p}$ which implies that 
% \rephrase{Considering the set $J_{n\pars{r}}$ we see that that a large $r$ also induces a large number of terms $n\pars{r}$, which suggests that the bound is only relevant in an asymptotic regime.}
% On the one hand, this means that the bound only applies for a large number of terms $n$, i.e.\ asymptotically.
% On the other hand, it requires the sequence $\tilde{\omega}$ to be bounded.
Hence, either the sparse vectors are part of a finite dimensional space, which contradicts the asymptotical nature of the result, or the sparse vectors are in an infinite dimensional space but the sequence is asymptotically constant, which only results in a very limited generalisation of the classical Stechkin lemma.
% {\definecolor{lightgray}{gray}{0.75}
% \color{lightgray}
% From the point of view of our weighted Stechkin result, this is surprising since Lemma~\ref{lem:weighted_stechkin} can be applied with $\alpha = \tilde{\omega}^{\pars{2-p}/p}$, $\sigma = \tilde{\omega}^{-1}$ and $\omega = \tilde{\omega}^{\pars{2-q}/q}$ to yield the bound
% \begin{equation}
% 	\norm{\pars{1-P_{J_n}}\tilde{\omega}^{\pars{2-p}/p} v}_{\ell^p}
% 	\le \norm{P_{J_{n+1}} \tilde{\omega}}_{\ell^2}^{-2s}
% 	\norm{\tilde{\omega}^{\pars{2-q}/q} v}_{\ell^q},
% 	\quad s := \frac{1}{q} - \frac{1}{p}
% \end{equation}
% without requiring $\norm{P_{J_n} \tilde{\omega}}_{\ell^2} > \norm{\tilde{\omega}}_{\ell^\infty}$.
% }
These shortcomings can be eliminated by using our weighted version of Stechkin's lemma.
In fact, applying Lemma~\ref{lem:weighted_stechkin} with $\bfalpha = \tilde{\bfomega}^{\pars{2-p}/p}$ and $\bfomega = \tilde{\bfomega}^{\pars{2-q}/q}$ to the bound~\eqref{eq:weighted_r-sparse_bound} yields the subsequent corollary of Lemma~\ref{lem:weighted_stechkin}.
\begin{corollary} \label{cor:ward_weighted_stechkin}
    For $\tilde\bfomega\in\bracs{0,\infty}^{\mbb{N}}$ and $0 < q < p \le 2$ define $\bfalpha := \tilde{\bfomega}^{\pars{2-p}/p}$, $\bfsigma := \tilde{\bfomega}^{-1}$ and $\bfomega := \tilde{\bfomega}^{\pars{2-q}/q}$.
    Let $\bfv\in \ell^\infty_\bfsigma \cap\ell^q_\bfomega$ and let $J_n$ be the set of indices corresponding to the $n$ largest elements of $\bfsigma\abs{\bfv}$.
    Then, for any $r \ge 0$,
    \begin{equation}
        \bfsigma_r\pars{\bfv}_{\ell^p_\bfalpha}
        \le \norm{\pars{1 - P_{J_{n\pars{r}}}} \bfv}_{\ell^p_\bfalpha}
    	\le r^{-s}
    	\norm{\bfv}_{\ell^q_\bfomega},
    	\quad s := \frac{1}{q} - \frac{1}{p} .
    \end{equation}
    % \begin{equation}
    %     \sigma_r\pars{v}_{\omega,p}
    %     \le \tilde\sigma_r\pars{v}_{\omega,p}
    %     \le r^{-s}
    %     \norm{\tilde{\omega}^{\pars{2-q}/q} v}_{\ell^q},
    %     \quad s := \frac{1}{q} - \frac{1}{p}
    % \end{equation}
    % for any $r \ge 0$.
    % This is summarised in the subsequent corollary of Lemma~\ref{lem:weighted_stechkin}.
\end{corollary}

This new weighted version of Stechkin's lemma results in improved bounds in the original work~\cite{rauhut_2016_weighted_l1} as well as derived works such as~\cite{adcock_2022_isMCbad}.

\begin{example}
\label{ex:error_bounds}
    Let $\braces{B_j}_{j\in\mbb{N}}$ be an $L^2\pars{Y, \rho}$-orthonormal basis and identify $v\in L^2\pars{Y,\rho}$ with its sequence of coefficients $\bfv\in\ell^2$.
    Moreover, define the weight sequence $\bfomega_j := \norm{B_j}_{w,\infty}$ and the model class
    \begin{equation}
        \mcal{M} := \braces{\bfv\in\ell^2 \,:\, \norm{\bfv}_{\ell^0_\bfomega}\le r} .
    \end{equation}
    Combining Proposition~\ref{prop:empirical_projection_error} and Theorem~\ref{thm:sparse_RIP_weighted} with Corollary~\ref{cor:ward_weighted_stechkin} yields the bound
    \begin{align}
        \norm{v - v_{\mcal{M},\boldsymbol{y}}}
        &\le \norm{v - v_{\mcal{M}}} + \tfrac{2}{\sqrt{1-\delta}} \norm{v - v_{\mcal{M}}}_{w,\infty} \\
        &\le \norm{\bfv - \bfv_{\mcal{M}}}_{\ell^2} + \tfrac{2}{\sqrt{1-\delta}} \norm{\bfv - \bfv_{\mcal{M}}}_{\ell^1_\bfomega} \\
        &\le \pars{1 + \tfrac{2}{\sqrt{1-\delta}}} \norm{\bfv - \bfv_{\mcal{M}}}_{\ell^1_\bfomega} \\
        &\le r^{-1} \pars{1 + \tfrac{2}{\sqrt{1-\delta}}} \norm{\bfv}_{\ell^{1/2}_{\bfomega^{3/2}}} ,
    \end{align}
    which holds with high probability if $n \gtrsim r\log^3\pars{r}\delta^{-2}$.
\end{example}

\subsection{The relation of \texorpdfstring{$\ell^p_\bfomega$}{weighted lp spaces} to other spaces}
% \paragraph{\textbf{Properties of $\ell^p_\omega$ and relations to the other spaces.}}\mbox{}\\
It is of general interest to examine the relation of weighted sequence spaces depending on exponents and the weight sequences, which is the topic of this section.
We first substantiate our claim that $\bfomega$ measures the decay of the sequence $\bfv$ by noting that a sequence in $\ell^q_\bfomega$ decays in modulus with a rate of $\bfomega^{-1}$.
\begin{lemma}
\label{lem:wlp_components}
    Let $\bfv\in\ell^q_\bfomega$.
    Then $\abs{\bfv_k} \le \bfomega_k^{-1}\norm{\bfv}_{\ell^q_\bfomega}$.
\end{lemma}
\begin{proof}
    Obviously, $\bfomega_k\abs{\bfv_k} \le \norm{\bfomega \bfv}_{\ell^q}$.
\end{proof}

The regularity in the $\ell^q_\bfomega$ space is described by the two parameters $q$ and $\bfomega$.
The preceding lemma implies that a sequence $\bfv\in\ell^p_{\tilde\bfomega}$ also lies in $\ell^q_{\bfomega}$, if its upper bound lies in this space, i.e.
$$
    \norm{\bfv}_{\ell^q_\bfomega}
    \le \norm{\tilde\bfomega^{-1} \norm{\bfv}_{\ell^p_{\tilde\bfomega}}}_{\ell^q_\bfomega}
    = \norm{\tilde\bfomega^{-1}}_{\ell^q_\bfomega} \norm{\bfv}_{\ell^p_{\tilde\bfomega}} .
$$
This indicates that the exponent parameter $q$ can be increased by simultaneously increasing the weight sequence parameter $\bfomega$.
This is made precise in the subsequent lemma.

\begin{lemma}
\label{lem:trade_sparsity}
    Let $0 \le q \le p \le\infty$ and $\bfomega,\tilde{\bfomega}\in\bracs{0,\infty}^{\mbb{N}}$.
    Then for any sequence $\bfv\in\ell^p_{\tilde{\bfomega}}$, it holds that
    $$
        \norm{\bfv}_{\ell^q_\bfomega} \le \norm{\tilde\bfomega^{-1}}_{\ell^{1/s}_{\bfomega}} \norm{\bfv}_{\ell^p_{\tilde\bfomega}},
        \quad
        s = \tfrac1q - \tfrac1p .
    $$
\end{lemma}
\begin{proof}
    By H\"older's inequality
    \begin{align}
        \norm{\bfv}_{\ell^q_\bfomega}^q
        = \norm{\bfomega^q\bfv^q}_{\ell^1}
        = \norm{\pars{\tfrac{\bfomega}{\tilde\bfomega}}^q\pars{\tilde\bfomega\bfv}^q}_{\ell^1}
        \le \norm{\pars{\tfrac{\bfomega}{\tilde\bfomega}}^q}_{\ell^r} \norm{\pars{\tilde\bfomega\bfv}^q}_{\ell^t}
        = \norm{\tfrac{\bfomega}{\tilde\bfomega}}_{\ell^{rq}}^q \norm{\tilde\bfomega\bfv}_{\ell^{tq}}^q,
    \end{align}
    where $r\in\bracs{1,\infty}$ and $t\in\bracs{1, \infty}$ satisfy $\frac{1}{r} + \frac{1}{t} = 1$.
    Choosing $t = \frac pq$ yields the claim.
\end{proof}

% In the special case of $q=1$ and $p=2$, the preceding result can be strengthened.
% % \declaretheorem[style=claim, name=Lemma~\ref*{lem:trade_sparsity}*]{trade_sparsity_star}
% % \begin{trade_sparsity_star}
% % \end{trade_sparsity_star}
% \begin{mdframed}[backgroundcolor=lightgray, linecolor=dimgray, linewidth=0.75pt, innertopmargin=1.5ex]
% \makeatletter
% \protected@edef\@currentlabelname{\ref*{lem:trade_sparsity}$^\star$}
% \protected@edef\@currentlabel{\ref*{lem:trade_sparsity}$^\star$}
% \makeatother
% \label{lem:trade_sparsity_star}
% \textbf{Lemma~\ref*{lem:trade_sparsity}$^\star$.\hspace{0.2em}}\itshape
%     Let $\bfomega,\tilde{\bfomega}\in\bracs{0,\infty}^{\mbb{N}}$.
%     Then $\norm{\bfv}_{\ell^1_\bfomega} \le \norm{\tilde{\bfomega}^{-1}}_{\ell^2_\bfomega} \norm{\bfv}_{\ell^2_{\tilde{\bfomega}}}$ for any sequence $v\in\ell^2_{\tilde{\bfomega}}$.
% \end{mdframed}
% \begin{proof}
% The claim follows from
% $\norm{v}_{\ell^1_\omega}
%     = \pars{\omega, \abs{v}}_{\ell^2}
%     = \pars{\tfrac{\omega}{\tilde{\omega}}, \tilde{\omega}\abs{v}}_{\ell^2}
%     \le \pars{\tfrac{\omega}{\tilde{\omega}}, \tilde{\omega}\abs{v}}_{\ell^2}
%     \le \norm{\tilde{\omega}^{-1}}_{\ell^2_{\omega}} \norm{v}_{\ell^2_{\tilde{\omega}}}$.
% \end{proof}

\begin{example}[Sparse polynomial approximation rates in Gaussian Sobolev spaces]
\label{ex:Hk}
    Let $\gamma$ be the standard Gaussian measure on $\mbb{R}$ and $\braces{B_j}_{j\in\mbb{N}}$ be the basis of normalised Hermite polynomial in $L^2\pars{\mbb{R}, \gamma}$.
    Since these polynomials constitute an Appell sequence, it holds that
    \begin{equation}
        \norm{v}_{H^k\pars{\mbb{R}, \gamma}} = \norm{\bfv}_{\ell^2_{\tilde \bfomega\pars{k}}}
        \qquad\text{with}\qquad
        % \omega\pars{k}_j^2 := \sum_{\ell=0}^{j\wedge k} \frac{j!}{\pars{j-\ell}!} .
        \tilde \bfomega\pars{k}_j
        := \sqrt{\textstyle \sum_{\ell=0}^{\min\braces{j,k}} \frac{\Gamma(j+1)}{\Gamma\pars{j-\ell+1}}}
        \asymp j^{k/2} := \bfomega\pars{k}_j.
    \end{equation}
    Applying Lemma~\ref{lem:weighted_stechkin} yields the following bound for the best $n$-term approximation $v_n$ of $v$:
    \begin{equation}
        \norm{v-v_n}_{L^2\pars{\mbb{R},\gamma}}
        = \norm{  \bfv - P_{J_n}\bfv }_{\ell^2}
        \le \norm{P_{J_{n+1}} \bfomega\pars{k-\varepsilon}}_{\ell^2}^{-1} \norm{\boldsymbol{v}}_{\ell^1_{\bfomega\pars{k-\varepsilon}}}
        \lesssim \pars{n+1}^{-\pars{k+1-\varepsilon}/2} \norm{\boldsymbol{v}}_{\ell^1_{\bfomega\pars{k-\varepsilon}}} .
    \end{equation}
    The required weighted summability $\bfv \in \ell^1_{\bfomega\pars{k-\varepsilon}}$ can usually not be inferred directly from the smoothness of the function.
    However, since $\norm{\tilde \bfomega\pars{k+1}^{-1}}_{\ell^2_{\bfomega\pars{k-\varepsilon}}}$ is finite, Lemma~\ref{lem:trade_sparsity} can be used to obtain the more natural summability condition $\bfv \in \ell^2_{\tilde \bfomega\pars{k+1}}$.
    % {\color{BurntOrange}
    % With the convention that $x! = 1$ for $x\le 0$, it holds that
    % $$
    %     \norm{\omega\pars{k+2}^{-1}}_{\ell^2_{\omega\pars{k}}}
    %     \le k\sum_j \pars{\tfrac{j!}{\pars{j-k}!}} \pars{\tfrac{j!}{\pars{j-k-2}!}}^{-1}
    %     = k\sum_j \tfrac{\pars{j-k-2}!}{\pars{j-k}!}
    %     = k\sum_j \tfrac{1}{j\pars{j-1}}
    %     < \infty .
    % $$
    % }
    Consequently, for arbitrary $\varepsilon>0$, 
    \begin{equation}
        \norm{v-v_n}_{L^2\pars{\mbb{R},\gamma}}
        \lesssim \pars{n+1}^{-\pars{k+1-\varepsilon}/2} \norm{v}_{H^{k+1}\pars{\mbb{R},\gamma}} .
    \end{equation}
\end{example}

\todo{
In one dimension, we know that for increasing functions left Riemann sums are lower bounds and right Riemann sums are upper bounds (the roles are reversed for decreasing functions).
In multiple dimensions, we can use these arguments inductively on each variable.
}
\begin{remark}[Best $n$-term rates in higher dimensions]
\label{rmk:best-n-term_hd}
    To briefly discuss the best $n$-term rates in higher dimensions, we consider isotropic weight sequences of the form $\bar{\bfomega}\pars{a} = \bfomega\pars{a}^{\otimes M}$, where $a\in\pars{0, \infty}$ determines growth of $\bfomega\pars{a}$ 
    (rates for anisotropic product weight sequences should follow by similar arguments). 
    To obtain worst-case rates for the approximation, we apply Lemmas~\ref{lem:weighted_stechkin} and~\ref{lem:trade_sparsity}
    $$
        \norm{\bfu - P_{J_n\bfu } \bfu}_{\ell^2}
        \le \norm{P_{J_{n+1}} \bar\bfomega\pars{a}}_{\ell^2}^{-1} \norm{\bfu}_{\ell^1_{\bar\bfomega\pars{a}}}
        \le \norm{P_{J_{n+1}} \bar\bfomega\pars{a}}_{\ell^2}^{-1} \norm{\bar\bfomega\pars{A}^{-1}}_{\ell^2_{\bar\bfomega\pars{a}}} \norm{\bfu}_{\ell^2_{\bar\bfomega\pars{A}}}
    $$
    and compute an upper bound for the decay rate $\varepsilon(n) := \norm{P_{J_{n+1}} \bar\bfomega\pars{a}}_{\ell^2}^{-1}$.
    Then, for fixed $a$, we choose the parameter $A>a$ as small as possible while ensuring that $\norm{\bar\bfomega\pars{A}^{-1}}_{\ell^2_{\bar\bfomega\pars{a}}}$ is finite.
    \paragraph{Exponential decay (analytic regularity)} 
    Consider the weight sequence $\bfomega\pars{a}_j = f_a(j)$ with $f_a(x) := \exp(ax)$.
    In this case, there exists a constant $c_R$ such that
    $$
        \varepsilon(n)
        \lesssim n^{-(M-1)/(2M)} \exp\pars{-c_R a n^{1/M}}
    $$
    and it holds that $\norm{\bar{\bfomega}\pars{A}^{-1}}_{\ell^2_{\bar\bfomega\pars{a}}} < \infty$ for any $a<A$.
    \paragraph{Algebraic decay (mixed Sobolev regularity)}
    Consider the weight sequence $\bfomega\pars{a}_j = g_a(j)$ with $g_a(x) := (x+1)^a$.
    In this case, we obtain the bound
    $$
        \varepsilon(n) \lesssim n^{-(a + 1/2)} \ln\pars{n}^{a(M-1)}.
    $$
    and it holds that $\norm{\bar{\bfomega}\pars{A}^{-1}}_{\ell^2_{\bar\bfomega\pars{a}}} < \infty$ for any $a<A-\tfrac{1}2$.
    
    Proofs for these statements can be found in appendix~\ref{sec:n-term}.
    Note that interestingly, the best $n$-term approximation rate seems to depend on the dimension $M$ in the exponential case while it is independent of $M$ in the algebraic case.
    
    Together with Example~\ref{ex:Hk}, this provides best $n$-term $L^2$-approximation rates in the Sobolev spaces $H^{k,\mathrm{mix}}\pars{\mbb{R}^R, \gamma^{\otimes M}}$ for the tensor product Hermite polynomial basis.
    It can also be shown~\cite{Hansen2010BestMA} that the similar rate $\norm{\pars{I-P_{J_n}}u}_{L^2} \le n^{-k} \norm{u}_{H^{k,\mathrm{mix}}}$ also holds (up to logarithmic factors) for the hierarchical tensor product spline basis in the Sobolev spaces $H^{k,\mathrm{mix}}([0,1]^M, \lambda^{\otimes M})$ from Example~\ref{ex:sparse_grids}.
\end{remark}

\todo{For Sobolev spaces we can always choose a spectral basis. $\to$ Is the hierarchical spline basis only needed for Besov spaces?}

Finally, we use these insights to highlight the relation of weighted sequences spaces to other well-knwon sequence spaces.
We start our discussion with the relation of $\ell^p_\bfomega$ for different values of $p$ and $\bfomega$.
In particular, we show that $\ell^q$ can not be embedded into $\ell^p_\omega$ for any $p$ and any unbounded $\bfomega$.
It is clear that this is not possible, because otherwise Lemma~\ref{lem:wlp_components} would provide decay rates for sequences in $\ell^p$.
A concrete counterexample is provided in the proof of the subsequent lemma.

\begin{lemma} \label{lem:lq_not_in_wlp}
    Let $0 < q \le p \le\infty$ and $\bfomega \lesssim \tilde{\bfomega}\in\bracs{0,\infty}^{\mbb{N}}$.
    Then 
    \begin{itemize}
    \item[(i)] $\ell^q_\bfomega \subseteq \ell^p_\bfomega$ and $\ell^p_{\tilde{\bfomega}} \subseteq \ell^p_{\bfomega}$.
    \end{itemize}
    Moreover,
    \begin{itemize}
        \item[(ii)] if $1\lesssim\bfomega$ is bounded, then $\ell^p_\bfomega \simeq \ell^p$ and
        \item[(iii)] if $\bfomega$ is unbounded, then $\ell^q \not\subseteq \ell^p_\bfomega \subseteq \ell^p$ for any $q>0$.
    \end{itemize}
\end{lemma}
\begin{proof}
    The two inclusions $\ell^q_\bfomega \subseteq \ell^p_\bfomega$ and $\ell^p_{\tilde{\bfomega}} \subseteq \ell^p_{\bfomega}$ follow by definition and the assertion (ii) holds because $\norm{\bfv}_{\ell^p} \lesssim \norm{\bfomega \bfv}_{\ell^p} \le \norm{\bfomega}_{\ell^\infty}\norm{\bfv}_{\ell^p}$.
    Hence, the mapping $\bfv \mapsto \bfomega^{-1}\bfv$ provides an isometry between $\ell^p$ and $\ell^p_\bfomega$.
    To show the assertion (iii), assume that $\bfomega$ is unbounded.
    Then there exists a strictly increasing function $\sigma:\mbb{N}\to\mbb{N}$, which defines a subsequence of $\bfomega$ such that $\bfomega_{\sigma\pars{k}} \ge 2^k$ for all $k\in\mbb{N}$.
    Now let $\varepsilon > 0$ and define the sequence $\bfv$ by
    \begin{equation}
        % \bfv_k := \begin{cases}
        %     j^{-\pars{1 + \varepsilon}/q} & \exists j\in\mbb{N} : k = \sigma\pars{j} \\
        %     0 & \text{otherwise}
        % \end{cases} .
        \bfv_{\sigma\pars{j}} = j^{-\pars{1 + \varepsilon}/q}
        \quad\text{and}\quad
        \bfv_{k} = 0
        \quad\text{otherwise.}
    \end{equation}
    This sequence satisfies $\bfv\in\ell^q$ and $\bfv\not\in\ell^p_\bfomega$ since
    \begin{align}
        \norm{\bfv}_{\ell^q}^q
        &= \sum_{k\in\mbb{N}} \abs{\bfv_k}^q
        = \sum_{j\in\mbb{N}} \abs{\bfv_{\bfsigma\pars{j}}}^q
        = \sum_{j\in\mbb{N}} j^{-\pars{1+\varepsilon}} < \infty
        \quad\text{and} \\
        \norm{\bfv}_{\ell^p_\bfomega}^p
        &= \sum_{k\in\mbb{N}} \abs{\bfomega_k \bfv_k}^p
        = \sum_{j\in\mbb{N}} \abs{\bfomega_{\sigma\pars{j}} \bfv_{\sigma\pars{j}}}^p
        \ge \sum_{j\in\mbb{N}} 2^{jp} j^{-\pars{1+\varepsilon}p/q} = \infty . \qedhere
    \end{align}
\end{proof}

Finally, we examine the relation of the weighted $\ell^q_\bfomega$ spaces to the \emph{monotone $\ell^{q,\mathrm{mon}}$ spaces}  (cf.~\cite{adcock_2022_isMCbad}).
For any sequence $\bfv$, define the \emph{minimal monotone majorant} $\bfv^{\mathrm{mon}}$ by
\begin{equation}
    \bfv^{\mathrm{mon}}_j := \sup_{k\ge j} \abs{\bfv_k}
    \quad\text{for all}\quad
    j\in\mbb{N} .
\end{equation}
The space $\ell^{q,\mathrm{mon}}$ is then defined as the set of all sequences for which the norm $\norm{\bfv}_{\ell^{q,\mathrm{mon}}} := \norm{\bfv^{\mathrm{mon}}}_{\ell^{q}}$ is finite.

\begin{lemma}
\label{lem:lqmon_in_wlp}
    Let $0 < p,q \le\infty$ and $\bfomega\in\bracs{0,\infty}^{\mbb{N}}$ and define $\varkappa_k := (k+1)^{-1/p}$.
    Then,
    \begin{itemize}
        \item[(i)] $\ell^q_\bfomega \subseteq \ell^{p,\mathrm{mon}}$ if $\bfomega^{-1}\in\ell^{p,\mathrm{mon}}$ and
        \item[(ii)] $\ell^{p,\mathrm{mon}} \subseteq \ell^q_\bfomega$ if $\varkappa\in\ell^q_\bfomega$.
    \end{itemize}
\end{lemma}
\begin{proof}
    For the first assertion, assume that $\bfv\in\ell^q_\bfomega$.
    Then, by Lemma~\ref{lem:wlp_components}, $\abs{\bfv_k^{\mathrm{mon}}} = \sup_{j\ge k} \abs{\bfv_j} \le \sup_{j\ge k} \bfomega_j^{-1} \norm{\bfv}_{\ell^q_\bfomega} = \pars{\bfomega^{-1}}^{\mathrm{mon}}_k \norm{\bfv}_{\ell^q_\bfomega}$.
    To show the second assertion, let $\bfv\in\ell^{p,\mathrm{mon}}$.
    Then, by Lemma~\ref{lem:stechkin}, $\abs{\bfv_k} \le \abs{\bfv_k^{\mathrm{mon}}} \le \varkappa_k \norm{\bfv}_{\ell^{p,\mathrm{mon}}}$ and consequently $\norm{\bfomega \bfv}_{\ell^q}\le\norm{\bfomega\varkappa}_{\ell^q}\norm{\bfv}_{\ell^{p,\mathrm{mon}}}$.
\end{proof}

\todo{Weak $\ell^p$? For all $p,\varepsilon>0$ it holds that $\ell^p \subset \mathrm{w}\!\ell^p \subset \ell^{p+\varepsilon}$.
The advantage is that $\mathrm{w}\!\ell^p$ are always quasi-normed spaces.
Maybe Stechkin could be generalised to these spaces.
But I don't see an immediate necessity.}
\todo{Lorentz sequence spaces? Our weighting makes the spaces stronger, Lorentz spaces are weaker...}
\section{Sparse approximation of parametric PDEs}
\label{sec:sparse approximation of parametric PDEs}

This section is concerned with an application of the weighted Stechkin lemma for a popular class of functions where weighted sparsity is encountered naturally.
In what follows, we consider solutions of parametric PDEs that have become popular in the field of Uncertainty Quantification.
We restrict our attention to two prototypical examples mentioned above in Theorems~\ref{thm:Legendre_decay_Bachmayr} and~\ref{thm:Hermite_decay_Bachmayr} that exhibit a holomorphic dependence on the parameter $y$.
The proofs of these bounds are typically rather involved and e.g.\ make use of techniques from complex analysis.
With the weighted version of Stechkin's lemma deduced in the preceding section, alternative proofs for such bounds can be derived with more elementary techniques.
The principle is demonstrated in this section for the one-dimensional case $L=1$ as a use case of Lemma~\ref{lem:weighted_stechkin}.

Assuming that the coefficient $a\pars{y} \ge \check{a}\pars{y} > 0$ is finite and bounded from below for every $y\in\mbb{R}$ and $f \in H^{-1}(D)$, Lax--Milgram theorem allows us to define the solution $u(y)$ in the space $H^1_0(D)$ through the variational formulation
\begin{equation}
\label{eq:variational_form}
    \int_D a(x,y)\nabla_{\!x} u(x,y)\cdot \nabla_{\!x} v(x) \dx
    = \int_D f(x)v(x) \dx
    \qquad\text{for all}\qquad
    v\in H^1_0\pars{D} .
\end{equation}
Moreover, a standard Lax--Milgram a priori estimate tells us that
\begin{equation}
    \norm{u\pars{y}}_{H^1_0}
    \le \check{a}\pars{y}^{-1} \norm{f}_{H^{-1}\pars{D}} .
    % \le \check{a}\pars{y}^{-1} C_{\mathrm{P}} \norm{f}_{L^2\pars{D}},
\end{equation}
% where $C_{\mathrm{P}}$ is the Poincar\'e constant of $D$.
% \anthony{PT: check if Poincaré inequalité is really needed, or if we can just work with  $\norm{f}_{H^{-1}\pars{D}}$.}

Using the machinery of weighted $\ell^p_\bfomega$-spaces developed in Section~\ref{sec:weighted_stechkin}, we now derive a priori best $n$-term convergence bounds for the solution of~\eqref{eq:darcy} from first principles.
Both results rely on the holomorphy of the solution map $y\mapsto u\pars{y} := u\pars{\bullet,y}$ and make use of the following extension of Cauchy's inequality to Banach spaces.

\begin{theorem}[Lemma~2.4 in~\cite{COHEN_2011}]
\label{thm:cauchy_banach}
    Let $\rho\in\pars{1,\infty}$, $X$ be a Banach space and $v : B_{\mbb{C}}\pars{0, \rho} \to X$ be holomorphic  such that $\sup_{y\in B_{\mbb{C}}\pars{0,\rho}} \norm{v\pars{y}}_X \le M < \infty$.
    Then the power series coefficients $\boldsymbol{v}\in X^{\mathbb{N}}$ of $v$ satisfy
    \begin{equation}
        \norm{\boldsymbol{v}_k}_X \le M \rho^{-k} .
    \end{equation}
\end{theorem}

\subsection{Affine coefficients}
\label{sec:affine coefficients}

We first consider the model problem~\eqref{eq:darcy} with affine coefficients~\eqref{eq:a_uniform}.
Summability of the power series of solution $u$ can be shown based on its holomorphy.
% {\color{MidnightBlue}\begin{itemize}
%     \item We essentially arrive at the same result using only the Cauchy theorem for power series obtaining what is usually derived using Chebyshev series as an intermediate step.
% \end{itemize}}
\begin{theorem}
\label{thm:affine_holomorphy}
    Let $\rho > 1$ and the \emph{uniform ellipticity assumption} (UEA)
    \begin{equation}
        C := \inf_{x\in D} a_0\pars{x} - \rho \sum_{j\ge 1} \vert a_j(x) \vert  > 0
    \end{equation}
    be satisfied.
    Moreover, let $u$ be the solution of the diffusion equation~\eqref{eq:darcy} with affine coefficients~\eqref{eq:a_uniform}.
    Then the map $y\mapsto u\pars{y}$ is holomorphic from $\bracs{-1, 1}$ to $H^1_0(D)$ and belongs to $L^p\pars{[-1, 1], \gamma; H^1_0\pars{D}}$ for all $p\in\mbb{N}\cup\braces{\infty}$, where $\gamma$ is the uniform measure.
    Moreover, for any $r\in\pars{0, \rho}$, the power series coefficients $\bfu$ of $u$ satisfy the bound
    \begin{equation}
        % \norm{\boldsymbol{u}_k}_{H^1_0(D)} \le C_{\mathrm{P}} \norm{f}_{L^2\pars{D}} C^{-1} r^{-k} .
        \norm{\boldsymbol{u}_k}_{H^1_0(D)} \le \norm{f}_{H^{-1}\pars{D}} C^{-1} r^{-k} .
    \end{equation}
\end{theorem}
% \begin{proof}
%     Observe that 
%     $$
%         \inf_{x\in D} a_0\pars{x} - \norm{a\pars{x}}_{\ell^1}\rho \ge C > 0
%         % \qquad\Leftrightarrow\qquad
%         % \forall x\in D: a_0\pars{x} - \norm{a\pars{x}}_{\ell^1}\rho \ge C > 0
%         \quad\Rightarrow\quad
%         \forall x\in D: \frac{a_0\pars{x} - C}{\norm{a\pars{x}}_{\ell^1}} \ge \rho
%         \quad\Leftrightarrow\quad
%         \inf_{x\in D} \frac{a_0\pars{x} - C}{\norm{a\pars{x}}_{\ell^1}} \ge \rho
%     $$
%     The rest follows from similar arguments as Theorem~\ref{thm:logaffine_holomorphy}.
% \end{proof}
We omit the proof of this theorem since it follows by the same arguments as the one of the more interesting log-affine case in Theorem~\ref{thm:logaffine_holomorphy}.
Moreover, we note that in higher dimensions an anisotropic choice of $\rho$ can be used to reduce the regularity assumptions on $a$ as it is done in the log-affine case.

The preceding lemma guarantees a decay of the power series coefficients of $u$.
However, in numerical applications an expansion in terms of an orthonormal basis is preferable.
For the diffusion equation~\eqref{eq:darcy} with affine coefficients~\eqref{eq:a_uniform}, a suitable basis is given by the Legendre polynomials.
The subsequent two lemmas show how the decay of the power series coefficients translates into a decay of the Legendre coefficients.

\begin{lemma}[see~\cite{Iserles2010}]
\label{lem:legendre_coefficients}
    Let $v$ satisfy the conditions of Theorem~\ref{thm:cauchy_banach} and let $\bfv$ be the power series coefficients of $v$.
    % Given that $f\pars{z} = \sum_{n\in\mbb{N}} v_nz^n$ is analytic in an open domain $\Omega\subseteq\mbb{C}$ with $\bracs{-1, 1}\subset\Omega$,
    % it is true that
    Then
    \begin{equation}
        v\pars{z} = \sum_{m\in\mbb{N}} \hat{\bfv}_m L_m\pars{z},
    \end{equation}
    where $L_m$ is the $m$\textsuperscript{th} normalised Legendre polynomial and $\hat{\bfv}_m$ satisfies
    \begin{equation}
        \hat{\bfv}_m = \sum_{n\in\mbb{N}} \frac{\pars{2m + 1}^{1/2}\pars{m+2n}!}{2^{m+2n}n!\pars{\tfrac{3}{2}}_{m+n}} \bfv_{m+2n} ,
    \end{equation}
    where the Pochhammer symbol $\pars{a}_k$ is defined as $\pars{a}_0 = 1$, $\pars{a}_{k+1} = \pars{a}_k\pars{a+k}$.
\end{lemma}

% \begin{lemma}
% \label{lem:coefficient_trafo}
%     Let $v\in\ell^q_\omega$ for some $q < 1$ and let $\tilde\alpha\preceq\alpha\in\bracs{0,\infty}^{\mbb{N}\times\mbb{N}}$ be a double sequences such that
%     \begin{equation}
%         \hat{v}_m = \sum_{n\in a+b\mbb{N}} \tilde\alpha_{m,n} v_{n}
%     \end{equation}
%     for some constants $a,b\in\mbb{N}$.
%     Moreover, assume that $\alpha_m\omega^{-q}$ is decreasing for all $m$.
%     % s = 1/q - 1 = (1-q)/q
%     % sq = 1-q
%     % o^(1-q) = o^(sq) = a w^(-q)
%     % (w/o)^(sq) = w^(sq) / (a w^(-q)) = w^(sq+q) / a
%     %            = w / a
%     Then
%     \begin{equation}
%         \abs{\hat{v}_m} \le \alpha_{m,a}\omega_a^{-1} \norm{v}_{\ell^q_\omega} .
%     \end{equation}
% \end{lemma}
% \begin{proof}
%     By triangle inequality and Lemma~\ref{lem:monotone_weighted_stechkin} it holds that
%     \begin{equation}
%         \abs{\hat{v}_m}
%         \le \norm{P_{a+b\mbb{N}} v}_{\ell^1_{\alpha_m}}
%         \le \pars*{\tfrac{\omega_{a}}{\sigma_{a}}}^{-\pars{1-q}} \norm{v}_{\ell^q_{\omega}} .
%     \end{equation}
%     % Observe, that $\min a+b\mbb{N} = a$.
%     The assertion follows, since $\sigma^{1-q} = \alpha_m \omega^{-q}$.
% \end{proof}

\begin{theorem}
\label{thm:legendre_coefficient_bounds}
    Let $u$ be the solution of the diffusion equation~\eqref{eq:darcy} with affine coefficients~\eqref{eq:a_uniform}.
    Moreover, let the parameters $\rho>1$ and $C>0$ be defined as  in Theorem~\ref{thm:affine_holomorphy}.
    % Moreover, let $\hat{\bfv}$ be defined as in Lemma~\ref{lem:legendre_coefficients}.
    Then, for any $r\in\pars{1, \rho}$, the Legendre basis coefficients $\hat\bfu$ of $u$ satisfy
    $$
        % \norm{\hat{\bfu}_m}_{H^1_0(D)} \le C_{\mathrm{P}} \norm{u}_{L^2\pars{D}} C^{-1} \sqrt{2m+1}\frac{r^{-m}}{r^2-1} .
        \norm{\hat{\bfu}_m}_{H^1_0(D)} \le \norm{f}_{H^{-1}\pars{D}} C^{-1} \sqrt{2m+1}\frac{r^{-m}}{r^2-1} .
    $$
    This implies that $\hat{\bfu}\in\ell^q_{\bfomega}$ for all $q\in(0,\infty]$ and $\bfomega_n := p(n) r^n$ with $r\in(1,\rho)$ and any positive function $p$ growing at most polynomially.
    % \color{MidnightBlue}
    % Let the sequences $v$ and $\tilde{v}$ be defined as in Lemma~\ref{lem:legendre_coefficients}.
    % Moreover, assume that $v\in\ell^q_\omega$ for some $q < 1$ and an increasing sequence $\omega$.
    % Then
    % \begin{equation}
    %     \abs{\hat{v}_m} \le \pars{2m+1}^{-1/2}\omega_{m+1}^{-1}\norm{v}_{\ell^q_\omega} .
    % \end{equation}
\end{theorem}
\begin{proof}
    Denote by $\bfu$ the power series coefficients of $u$ and define the double sequence
    \begin{equation}
        \bfalpha_{m,k}
        := \sqrt{2m+1} \frac{k!}{2^k\pars{\pars{k-m}/2}!\pars{\tfrac{3}{2}}_{\pars{k+m}/2}} .
    \end{equation}
    By Lemma~\ref{lem:legendre_coefficients} it holds that $\hat{\bfu}_m = \sum_{k\in m+2\mbb{N}} \bfalpha_{m,k}\bfu_k$ and hence
    \begin{equation}
    \label{eq:legendre_coefficient_bounds}
        \norm{\hat{\bfu}_m}_{H^1_0(D)}
        \le \norm{P_{m+2\mbb{N}} \bfu}_{\ell^1_{\bfalpha_m}} ,
        % \le \norm{P_{2\pars{\tilde{m}+\mbb{N}}} v}_{\ell^1_{\alpha_m}}
        % \le \norm{\pars{1-P_{\bracs{\tilde{m}-1}}} v}_{\ell^1_{\alpha_m}} .
        % \le \norm{\pars{1-P_{\bracs{m-1}}} v}_{\ell^1_{\alpha_m}} .
    \end{equation}
    where the $\norm{\bfu}_{\ell^1_{\alpha_m}} = \sum_{k\in\mbb{N}} \alpha_{m,k}\norm{\bfu_k}_{H^1_0(D)}$.
    This bound is tight, since equality holds for $\bfu_{k+1} = \bfv_k\bfu_k$ and any non-negative sequence $0 \le \bfv \in\ell^1_{\alpha_m}$.
    % By Theorem~\ref{thm:affine_holomorphy} it holds that $\norm{\bfu_k}_{H^1_0(D)}\le c r^{-k}$ with $c := C_{\mathrm{P}} \norm{v}_{L^2\pars{D}} C^{-1}$.
    By Theorem~\ref{thm:affine_holomorphy} it holds that $\norm{\bfu_k}_{H^1_0(D)}\le c r^{-k}$ with $c := \norm{f}_{H^{-1}\pars{D}} C^{-1}$.
    Moreover, expressing the Pochhammer symbol in terms of the Gamma function yields the bound
    \begin{equation}
        \pars{\tfrac{3}{2}}_{k}
        = \frac{\Gamma\pars{k+\tfrac{3}{2}}}{\Gamma\pars{\tfrac32}}
        > \underbrace{\pars*{\min_{z\in[0,\infty)} \frac{\Gamma\pars{z+\tfrac{3}{2}}}{\Gamma\pars{z+1}}}}_{= \Gamma\pars{\tfrac{3}{2}}} \frac{\Gamma\pars{k+1}}{\Gamma\pars{\tfrac32}}
        = k!\ .
    \end{equation}
    Substituting both bounds into~\eqref{eq:legendre_coefficient_bounds} yields
    \begin{equation}
        \norm{\hat{\bfu}_m}_{H^1_0(D)}
        \le c\sqrt{2m+1} r^{-m} \sum_{n\in\mbb{N}} \underbrace{2^{-\pars{m+2n}} \binom{m+2n}{n}}_{=p_{m+2n}\pars{n} \le 1} r^{-2n}
        \le \frac{c\sqrt{2m+1} r^{-m}}{r^2-1} ,
    \end{equation}
    where $p_{m+2n}$ is the probability mass function of a binomial distribution with $m+2n$ trials and a success probability of $\tfrac{1}{2}$.
    The final claim follows directly from this bound and the ratio test for series convergence.
\end{proof}

\begin{remark}
    It is possible to use Lemma~\ref{lem:weighted_stechkin} to bound~\eqref{eq:legendre_coefficient_bounds}.
    But as shown in Examples~\ref{ex:q_omega_algebraic_decay} and~\ref{ex:q_omega_exponential_decay},
    Stechkin's lemma cannot fully exploit the decay of a weight sequence.
    In fact, if possible it is preferable to derive a bound directly as in the previous proof.
\end{remark}

\begin{example}
    Let $u$ be the solution of the diffusion equation~\eqref{eq:darcy} with affine coefficients~\eqref{eq:a_uniform}.
    We denote by $L_m$ the $m$\textsuperscript{th} normalised Legendre polynomial and by $\hat\bfu$ the Legendre basis coefficients of $u$.
    % Let $v$ satisfy the conditions of Theorem~\ref{thm:cauchy_banach}, let $\braces{L_j}_{j\in\mbb{N}}$ be the basis of normalised Legendre polynomials and identify $v$ with its sequence of Legendre series coefficients $\bfv$.
    Moreover, define the weight sequence $\bfomega_j := \norm{L_j}_{\infty} = \sqrt{2j + 1}$ and the model class
    \begin{equation}
        \mcal{M} := \braces{\bfv\in\ell^2 \,:\, \norm{\bfv}_{\ell^0_\bfomega}\le r} .
    \end{equation}
    We know from Example~\ref{ex:error_bounds} that
    \begin{equation}
        \norm{u - u_{\mcal{M},\boldsymbol{y}}}
        \le r^{-1} \pars{1 + \tfrac{2}{\sqrt{1-\delta}}} \norm{\bfu}_{\ell^{1/2}_{\bfomega^{3/2}}}
    \end{equation}
    holds with high probability if $n \gtrsim r\log^3\pars{r}\delta^{-2}$.
    Theorem~\ref{thm:legendre_coefficient_bounds} guarantees that
    $$
        \norm{\bfu}_{\ell^{1/2}_{\bfomega^{3/2}}}^{1/2}
        \lesssim \sum_{m\in\mbb{N}} (2m+1)^{3/4} \pars*{\sqrt{2m+2} r^{-m}}^{1/2}
        = \sum_{m\in\mbb{N}} (2m+1) r^{-m/2}
        = \frac{\sqrt{r}(\sqrt{r}+1)}{(\sqrt{r}-1)^2}
    $$
    is indeed finite.
    Note that we can probably obtain better rates by using a faster growing weight sequence $\bfomega$ and Lemma~\ref{lem:weighted_stechkin} instead of Corollary~\ref{cor:ward_weighted_stechkin}. 
\end{example}

\subsection{Log-affine coefficients}
\label{sec:log-affine coefficients}

%{\color{MidnightBlue}
%\begin{itemize}
%     \item Later we use~\cite{cohen_devore_2015}.
%     In this work, they only consider the affine case. Moreover, they never explicitly prove the inequality $\abs{u_\nu} \le ...$.
%     \item Emphasise that it is typically assumed that holomorphy is not enough to obtain decay in the Hermite coefficients. (Citing from~\cite{Bachmayr_2016_II}: ``The above Theorem 1.1 {\color{BurntOrange}($j\norm{a_j}_{L^\infty} \in \ell^p \implies \norm{u_\nu}\in\ell^p$)} is comparable to those obtained in [9] {\color{BurntOrange}(\cite{COHEN_2011})} for Taylor and Legendre coefficients in the affine case, except for the appearance of the additional factor $j$ in front of $\norm{\psi_j}_{L^\infty}$ {\color{BurntOrange}($\norm{a_j}_{L^\infty}$)}, which makes the summability assumption more restrictive. \textbf{On the one hand, it is quite remarkable that such a result exists, since the complex variable arguments used in the affine case, based solely on holomorphy of the solution map, are not sufficient to control the Hermite coefficients.}'' --- Maybe this is true in the infinite-dimensional case?
% \end{itemize}
% }

The analysis of~\eqref{eq:darcy} with log-affine coefficient~\eqref{eq:a_lognormal} is much more involved from a theoretical and practical side than the affine case.
% \rephrase{
% As noted in~\cite{Bachmayr_2016_II}, the transfer of smoothness from the coefficient to the solution is quite remarkable.
% However, when compared to the affine case, the required coefficient summability $j\norm{a_j}_{L^\infty(D)}\in\ell^p$ is more restrictive due to the additional factor $j$.
% This suffices to control the otherwise non-converging Hermite coefficients of the solution.
% }
As in the affine case, we begin by showing holomorphy of the solution $u$.
The analysis is based on the approach in~\cite{cohen_devore_2015}, where only the affine case was considered and an explicit decay of the coefficients $\norm{\bfu_\nu}_{H^1_0(D)}$ was not shown.

\begin{figure}[H]
\begin{theorem}
\label{thm:logaffine_holomorphy}
    For every $x\in D$, let $\boldsymbol{a}\pars{x}$ denote the sequence of coefficients $\pars{\boldsymbol{a}_j\pars{x}}_{j\in\mbb{N}}$.
    Assume that there exists a sequence $\boldsymbol\rho\in\pars{0,\infty}^{\mbb{N}}$ such that
    \begin{equation}
        \sup_{x\in D} \norm{\boldsymbol{a}\pars{x}}_{\ell^\infty_{\boldsymbol\rho}} = C_1 < \infty
        \qquad\text{and}\qquad
        \norm{\boldsymbol\rho^{-1}}_{\ell^2} = C_2 < \infty .
    \end{equation}
    Then the map $y\mapsto u\pars{y}$ is entire and belongs to $L^p\pars{\mbb{R},\gamma ; H^1_0\pars{D}}$ for all $p\in\mbb{N}$, where $\gamma$ denotes the Gaussian measure.
    % Moreover, for $C_3 := C_{\mathrm{P}} \norm{f}_{L^2\pars{D}}$ the power series coefficients satisfy the bound
    Moreover, the power series coefficients satisfy the bound
    \begin{equation}
        \norm{\bfu_{\boldsymbol{\nu}}}_{H^1_0(D)} \le \norm{f}_{H^{-1}\pars{D}} \exp\pars{\norm{{\boldsymbol{\nu}}}_{1}} \pars*{\tfrac{\boldsymbol\nu\boldsymbol\rho}{C_1}}^{-{\boldsymbol{\nu}}} .
    \end{equation}
\end{theorem}
\end{figure}
\begin{proof}
    We start by providing a lower bound for $\check{a}\pars{y} > 0$.
    Since
    \begin{align}
        \inf_x\exp\pars{\pars{\boldsymbol a\pars{x}, y}_{\ell^2}}
        &= \exp\pars{\inf_x \pars{\boldsymbol a\pars{x}\boldsymbol\rho, \boldsymbol\rho^{-1}y}_{\ell^2}} \\
        % &\ge \exp\pars{-\sup_x \norm{a\pars{x}}_{\ell^\infty_\rho} \norm{\rho^{-1}y}_{\ell^p}} \\
        &\ge \exp\pars{-\sup_x \norm{\boldsymbol a\pars{x}}_{\ell^\infty_{\boldsymbol\rho}} \norm{\boldsymbol\rho^{-1}y}_{\ell^1}} \\
        &= \exp\pars{-C_1 \norm{\boldsymbol\rho^{-1}y}_{\ell^1}} ,
    \end{align}
    it holds that $\check{a}\pars{y} \ge \exp\pars{-C_1 \norm{\boldsymbol\rho^{-1}y}_{\ell^1}}$.
    The integrability of $u$ now follows by the simple calculation
    \begin{equation}
        \mbb{E}_y\bracs{\check{a}\pars{y}^{-k}}
        \le \mbb{E}_y\bracs{\exp\pars{kC_1 \norm{\boldsymbol\rho^{-1}y}_{\ell^1}}}
        = \prod_{j\in\mbb{N}} \mbb{E}_{y_j}\bracs{\exp\pars{kC_1 \boldsymbol\rho_j^{-1}\abs{y_j}}}
        \lesssim \prod_{j\in\mbb{N}} \exp\pars{\tfrac{1}{2}k^2C_1^2 \boldsymbol\rho_j^{-2}}
        = \exp\pars{\tfrac{1}{2} k^2 C_1^2 C_2^{2}}
        < \infty.
    \end{equation}
    We now show that the extension of the map $y \mapsto u\pars{y}$ to the complex domain is analytic.
    Following~\cite{cohen_devore_2015}, we start by defining for $\boldsymbol{r}\in\pars{0, \infty}^{\mbb{N}}$ the open polydiscs $U_{\boldsymbol{r}} := \prod_{k\in\mbb{N}} B\pars{0, {\boldsymbol{r}}_k}$ on which $u$ is uniformly bounded by
    \begin{equation}
    \label{eq:uniform_bound}
        \norm{u}_{L^\infty\pars{U_{\boldsymbol{r}} ; H^1_0\pars{D}}} \le \exp\pars{-C_1\norm{\boldsymbol\rho^{-1}{\boldsymbol{r}}}_{\ell^1}} \norm{f}_{H^{-1}\pars{D}}.
    \end{equation}
    % If we assume that $\boldsymbol{r}$ satisfies $\norm{\boldsymbol\rho^{-1}{\boldsymbol{r}}}_{\ell^1} =: C_2 < \infty$, the extension is defined on the open set $\bigcup_{\boldsymbol{\rho}^{-1}\boldsymbol r\in\ell^1} U_{\boldsymbol r} = \mbb{C}^{\mbb{R}}$.
    Now, we introduce for any coefficient field $a$ the operator $B\pars{a} : v\mapsto -\operatorname{div}_{\!x}\pars{a\nabla_{\!x}v}$ mapping from $H^{1}_0\pars{D}$ to $H^{-1}\pars{D}$
    and decompose the map $y\mapsto u\pars{y}$ into the chain of holomorphic maps
    \begin{equation}
        y \mapsto a\pars{y} \mapsto B\pars{a\pars{y}} \mapsto B\pars{a\pars{y}}^{-1} \mapsto B\pars{a\pars{y}}^{-1} f = u\pars{y} .
    \end{equation}
    The first map is holomorphic by definition and the second and last map are continuous linear maps and thereby also holomorphic.
    The third map is the operator inversion which is holomorphic at any invertible $B$.
    % Since $B\pars{a\pars{y}}$ is invertible for every $y\in U_{\boldsymbol{r}}$ and every polydisc $U_{\boldsymbol r}$, the map $y \mapsto u\pars{y}$ is holomorphic on $\bigcup_{{\boldsymbol \rho}^{-1}{\boldsymbol r}\in\ell^1} U_{\boldsymbol{r}} = \mbb{C}^{\mbb{N}}$ and thereby entire.
    Since $B\pars{a\pars{y}}$ is invertible for every $y\in \mbb{C}^{\mbb{N}}$, the map $y \mapsto u\pars{y}$ is entire.
    Applying Cauchy's inequality in Theorem~\ref{thm:cauchy_banach} to~\eqref{eq:uniform_bound}, we hence obtain
    \begin{equation}
        \norm{\bfu_{\boldsymbol{\nu}}}_{H^1_0(D)} \le \norm{f}_{H^{-1}(D)} \exp\pars{-C_1\norm{\boldsymbol\rho^{-1}\boldsymbol{r}}_{\ell^1}} \boldsymbol{r}^{-{\boldsymbol{\nu}}} .
    \end{equation}
    % with $C_3 := C_{\mathrm{P}} \norm{f}_{L^2\pars{D}}$.
    Choosing for every fixed multi-index ${\boldsymbol{\nu}}$ the sequence $\boldsymbol{r}_k := \frac{{\boldsymbol{\nu}}_k\boldsymbol\rho_k}{C_1}$ yields $\norm{\boldsymbol{\rho}^{-1}\boldsymbol{r}}_{\ell^1} = \frac1{C_1}\norm{\boldsymbol{\nu}}_{\ell^1}$ and proves the result.
\end{proof}

In the setting of log-affine coefficients~\eqref{eq:a_lognormal}, a suitable basis is given by the Hermite polynomials.
Similar to the Lemma~\ref{lem:legendre_coefficients} and Theorem~\ref{thm:legendre_coefficient_bounds},
the subsequent two results show how the decay of the power series coefficients translates into a decay of the Hermite coefficients.

\begin{lemma}
\label{lem:hermite_coefficients}
    Let $v$ satisfy the conditions of Theorem~\ref{thm:cauchy_banach} and let $\bfv$ be the power series coefficents of $v$.
    Then 
    % Given that $f\pars{z} = \sum_{n\in\mbb{N}} f_nz^n$ is analytic in an open domain $\Omega\subseteq\mbb{C}$ with $\bracs{-1, 1}\subset\Omega$, it is true that
    \begin{equation}
        v\pars{z} = \sum_{m\in\mbb{N}} \hat{\bfv}_m H_m\pars{z},
    \end{equation}
    where $H_m$ is the $m$\textsuperscript{th} normalised Hermite polynomial and $\hat{\bfv}_m$ satisfies
    \begin{equation}
        \hat{\bfv}_m = \sum_{n\in\mbb{N}} \frac{\pars{m+2n}!}{2^{n}n!\sqrt{m!}} \bfv_{m+2n} .
    \end{equation}
\end{lemma}
\begin{proof}
    For every $m\in\mbb{N}$, let $\mathit{He}_m$ denote the $k$\textsuperscript{th} monic probabilist's Hermite polynomial.
    Then, by~\cite[Chapter~11, Section 110]{rainville1960special},
    \begin{equation}
        z^k = \sum_{\substack{n\in\mbb{N}\\2n\le k}} \frac{k!}{2^nn!\pars{k-2n}!} \mathit{He}_{k-2n}\pars{z}.
    \end{equation}
    Plugging this into the power series expansion for $v$ yields
    \begin{equation}
        v\pars{z}
        = \sum_{\substack{k,n\in\mbb{N}\\2n\le k}} \frac{k!}{2^nn!\pars{k-2n}!} \bfv_k \mathit{He}_{k-2n}\pars{z}
        = \sum_{\substack{k,n,m\in\mbb{N}\\2n\le k\\m=k-2n}} \frac{k!}{2^nn!m!}\bfv_k \mathit{He}_m\pars{z}
        = \sum_{n,m\in\mbb{N}} \frac{\pars{m+2n}!}{2^nn!m!} \bfv_{m+2n} \mathit{He}_m\pars{z} .
    \end{equation}
    Substituting $\mathit{He}_m = \sqrt{m!}H_m$ yields the desired relation.
\end{proof}

\begin{theorem}
\label{thm:hermite_coefficient_bounds}
    Let $u$ be the solution of the diffusion equation~\eqref{eq:darcy} with log-affine coefficients~\eqref{eq:a_lognormal}.
    Moreover, let the parameters $\boldsymbol\rho$ and $C_1$ be defined as in Theorem~\ref{thm:logaffine_holomorphy}.
    Finally, let $L=1$ and assume that $\boldsymbol\rho > C_1$.
    Then, the Hermite basis coefficients $\hat\bfu$ of $u$ satisfy $\norm{\hat{\bfu}_m}_{H^1_0(D)} \lesssim \frac{\norm{f}_{H^{-1}(D)}}{\sqrt{\pars{m-1}!}}$.
    This implies that $\hat{\bfu}\in\ell^q_{\bfomega}$ for all $q\in(0,\infty]$ and $\bfomega$ that grow subfactorially.
    % Let $v$ be the sequences of power series coefficients from Theorem~\ref{thm:logaffine_holomorphy} with $L=1$ and assume that $\rho > C_1$.
    % Moreover, let $\hat{v}$ be defined as in Lemma~\ref{lem:hermite_coefficients}.
    % Then $\abs{\hat{v}_m} \lesssim \frac{C_3}{\sqrt{\pars{m-1}!}}$.
% \color{MidnightBlue}
%     and assume that $v\in\ell^q_\omega$ for some $q < 1$ and an increasing squence $\omega$.
%     Moreover, define the sequence $\alpha_n := 2^{4n}$.
%     Then, if the sequence $\alpha\omega^{-q}$ is decreasing, $\abs{\hat{v}_m} \le 2^{4\pars{m+1}}\omega_{m+1}^{-1}\norm{v}_{\ell^q_\omega}$.
\end{theorem}
\begin{proof}
    Define the double sequence
    \begin{equation}
        \boldsymbol\alpha_{m,k}
        := \frac{k!}{2^{\pars{k-m}/2} \pars{\pars{k-m}/2}! \sqrt{m!}} .
    \end{equation}
    By Lemma~\ref{lem:hermite_coefficients} it holds that $\hat{\bfu}_m = \sum_{k\in m+2\mbb{N}} \boldsymbol\alpha_{m,k}\bfu_k$.
    Hence
    \begin{equation}
    \label{eq:hermite_coefficient_bounds}
        \abs{\hat{\bfu}_m}
        \le \norm{P_{m+2\mbb{N}} \bfu}_{\ell^1_{\boldsymbol\alpha_m}}
        % \le \norm{P_{2\pars{\tilde{m}+\mbb{N}}} v}_{\ell^1_{\alpha_m}}
        % \le \norm{\pars{1-P_{\bracs{\tilde{m}-1}}} v}_{\ell^1_{\alpha_m}} .
        \le \norm{\pars{1-P_{\bracs{m-1}}} \bfu}_{\ell^1_{\boldsymbol\alpha_m}} .
    \end{equation}
    This bound is tight, since equality holds for any sequence $\bfu \ge 0$ with $\bfu_{k} = 0$ for all $k\in m+2\mbb{N}$.
    By Theorem~\ref{thm:logaffine_holomorphy} it holds that $\norm{\bfu_k}_{H^1_0(D)} \le \norm{f}_{H^{-1}(D)} \exp\pars{k} \pars{\frac{k\rho}{C_1}}^{-k}$ and by Stirling's approximation
    \begin{equation}
        \sqrt{2\pi k} \pars{\tfrac{k}{e}}^k \exp\pars{\tfrac{1}{12k+1}} < k! < \sqrt{2\pi k} \pars{\tfrac{k}{e}}^k \exp\pars{\tfrac{1}{12k}} .
    \end{equation}
    Substituting all three estimates into~\eqref{eq:hermite_coefficient_bounds} and assuming $m\ge 1$ yields
    % \rephrase{Note that Stirling's approximation is only valid for $k\ge 1$. But this is not a problem, since $m\ge 1$ and $k\in m+2\mbb{N}$ imply $k\ge m\ge 1$.}
    \begin{align}
        \norm{\hat{\bfu}_m}_{H^1_0(D)}
        % \le \frac{C_3}{\sqrt{m!}} \sum_{k\ge m} \frac{\sqrt{2\pi k} \exp\pars{k+\tfrac{1}{12k}}}{2^{\pars{k-m}/2} \pars{\pars{k-m}/2}!} \pars*{\frac{k}{e}}^k \pars*{\frac{k\rho}{C_1}}^{-k}
        % \le \frac{C_3}{\sqrt{m!}} \sum_{k\ge m} \frac{\sqrt{2\pi k} \exp\pars{k+\tfrac{1}{12k}}}{2^{\pars{k-m}/2} \pars{\pars{k-m}/2}!} \pars*{\frac{e\rho}{C_1}}^{-k}
        % \le \frac{C_3}{\sqrt{m!}} \sum_{k\ge m} \frac{\sqrt{2\pi k} \exp\pars{\tfrac{1}{12k}}}{2^{\pars{k-m}/2} \pars{\pars{k-m}/2}!} \pars*{\frac{\rho}{C_1}}^{-k}
        % \le \frac{C_3}{\sqrt{m!}} \sum_{k\ge m} \frac{\sqrt{2\pi k} \exp\pars{\tfrac{1}{12k}}}{2^{\pars{k-m}/2} \sqrt{\pi \pars{k-m}} \exp\pars{\tfrac{1}{6\pars{k-m}+1}}} \pars*{\frac{k-m}{2e}}^{\pars{k-m}/2} \pars*{\frac{\rho}{C_1}}^{-k}
        % \le \frac{\sqrt{2}C_3}{\sqrt{m!}} \sum_{k\ge m} \frac{\sqrt{k} \exp\pars{\tfrac{1}{12k} - \tfrac{1}{6\pars{k-m}+1}}}{2^{\pars{k-m}/2} \sqrt{\pars{k-m}}} \pars*{\frac{k-m}{2e}}^{\pars{k-m}/2} \pars*{\frac{\rho}{C_1}}^{-k}
        % \le \frac{\sqrt{2}C_3}{\sqrt{m!}} \sum_{k\ge m} \sqrt{\frac{k}{k-m}} \pars*{\frac{k-m}{e}}^{-\pars{k-m}/2} \pars*{\frac{\rho}{C_1}}^{-k}
        &\le \norm{f}_{H^{-1}(D)}\sqrt{\frac{2\pars{m+1}}{m!}} \sum_{k\ge m} \pars*{\frac{k-m}{e}}^{-\pars{k-m}/2} \pars*{\frac{\boldsymbol\rho}{C_1}}^{-k}
        % \le C_3\sqrt{\frac{2\pars{m+1}}{m!}} \pars*{\frac{1}{e}}^{-1/2} + \pars*{\frac{2}{e}}^{-2/2} + \sum_{k\ge m+3} \pars*{\frac{k-m}{e}}^{-\pars{k-m}/2} \pars*{\frac{\rho}{C_1}}^{-k}
        \le 2 \norm{f}_{H^{-1}(D)} \pars*{\sqrt{e}+\tfrac{e}{2}+\frac{\pars{C_1/\boldsymbol\rho}^{m+3}}{1 - C_1/\boldsymbol\rho}} \frac{1}{\sqrt{\pars{m-1}!}} .
        % \lesssim \frac{C_3}{\sqrt{\pars{m-1}!}} .
        \qedhere
    \end{align}
\end{proof}

\begin{corollary}
    Let $\boldsymbol{\hat{u}}$ be the sequences of Hermite basis coefficients from Theorem~\ref{thm:hermite_coefficient_bounds} with $L=1$ and assume that $\rho \ge C_1$.
    Then, if $\bfomega_j = \tilde{r}^{j/2}$ with $\tilde{r}\in[1, 2)$ and $J_n$ is the set of indices corresponding to the $n$ largest elements of the sequence $\bfomega_k^{-1}\norm{\boldsymbol{\hat u_k}}_{H^1_0(D)}$, it holds that
    \begin{equation}
        \norm{\pars{1-P_{J_n}} \hat{u}}_{L^2}
        \lesssim \norm{f}_{H^{-1}\pars{D}} \frac{\sqrt{2}}{\sqrt{2} - \sqrt{\tilde{r}}} \tilde{r}^{-n/2} .
    \end{equation}
\end{corollary}
\begin{proof}
    From Theorem~\ref{thm:hermite_coefficient_bounds}, we know that $\norm{\boldsymbol{\hat{u}}_m}_{H^1_0(D)} \lesssim \frac{\norm{f}_{H^{-1}\pars{D}}}{\sqrt{\pars{m-1}!}}$.
    Since $m! \ge 2^{m-1}$, it holds that $\frac{1}{\sqrt{\pars{m-1}!}} \lesssim 2^{-m/2} =: \boldsymbol\gamma_m$.
    Applying Lemma~\ref{lem:weighted_stechkin} with $p=2$, $q=1$ and $\boldsymbol{\alpha}\equiv 1$ yields
    $$
        \norm{\pars{1-P_{J_n}}\hat{\bfu}}_{L^2}
        = \norm{\pars{1-P_{J_n}}\boldsymbol{\hat{\bfu}}}_{\ell^2}
        \le \norm{P_{J_{n+1}} \boldsymbol{\omega}^2}^{-1/2}_{\ell^1} \norm{\boldsymbol{\hat{\bfu}}}_{\ell^1_{\boldsymbol{\omega}}}
        \lesssim \norm{f}_{H^{-1}\pars{D}} \norm{P_{J_{n+1}} \boldsymbol{\omega}^2}^{-1/2}_{\ell^1} \norm{\boldsymbol{\gamma}}_{\ell^1_{\boldsymbol{\omega}}} .
    $$
    The claim follows since $
        \norm{\boldsymbol{\gamma}}_{\ell^1_{\boldsymbol{\omega}}}
        = \frac{\sqrt{2}}{\sqrt{2} - \sqrt{\tilde{r}}} $
    and
    $
        \norm{P_{J_{n+1}}\boldsymbol{\omega}^2}_{\ell^1} \ge \norm{P_{\bracs{n+1}}\boldsymbol{\omega}^2}_{\ell^1} \ge \tilde{r}^n
    $.
    % Using the last bound from Example~\ref{ex:q_omega_exponential_decay} with $q=1$ yields the result.
    % \color{BurntOrange}
    % Using the arguments from Example~\ref{ex:q_omega_exponential_decay} with $q<2$ and $\tilde{r}\in\pars{1/2, 1}$ then yields
    % \begin{equation}
    %     \norm{\pars{1-P_{J_n}} f}_{L^2}
    %     = \norm{\pars{1-P_{J_n}} \hat{v}}_{\ell^2}
    %     \le \tilde{r}^{-sn} \pars{1-2\tilde{r}}^{-1/q} \frac{\rho}{\pars{\rho^{1/2} - 1}^2} \exp\pars{C}\norm{f}_{H^{-1}}
    % \end{equation}
    % with $s = 1/q - 1/2$.
\end{proof}

\begin{remark}
    Note that the proofs of Theorem~\ref{thm:legendre_coefficient_bounds} and~\ref{thm:hermite_coefficient_bounds} rely essentially on the the formulas in Lemma~\ref{lem:legendre_coefficients} and Lemma~\ref{lem:hermite_coefficients}.
    The similarity of these formulas indicates a deeper relation stemming from the explicit representations of $\pars{p_n, L_m}_{L^2}$ and $\pars{p_n, H_m}_{L^2}$ with $p_n\pars{z} := z^n$.
    We conjecture that similar representations can be derived for all families of orthonormal polynomials by means of the corresponding three-term recurrence relation.
    % This recurrence relates $\pars{p_n, B_m}_{L^2}$ to $\pars{p_{n-1}, B_{m+1}}$, $\pars{p_{n-1}, B_{m}}$ and $\pars{p_{n-1}, B_{m-1}}$.
\end{remark}

\section{Sparse approximation using tensor trains}
\label{sec:sparse_als}

In this section we consider sparse approximation problems in a high-dimensional setting where weighted sparse vectors can be identified with tensors.
We show that tensors in $\ell^q_\omega$ can be approximated efficiently in a model class of tensor trains with (weighted) sparse component tensors.
The derivation of this relies heavily on results in~\cite{li_2022_sparse_tt}, from which we recall some theorems.
For the sake of completeness and since the proofs foster some interesting insights, they are also provided.

Finally, we provide a practical algorithm to obtain these representations which provides an alternative for classical sparse approximation algorithms (such as weighted $\ell^1$-minimisation) that circumvents the CoD.

\subsection{Tensor train representation of sparse tensors}
This section recalls basic representation results for sparse tensors that are originally due to~\cite{li_2022_sparse_tt}.
We first introduce some basic operations on tensors.

\begin{definition}[Vectorisation]
    For any tensor $A\in\mbb{R}^{d_1\times\cdots\times d_M}$ , the vectorisation of $A$ is a vector $\operatorname{vec}\pars{A}\in \mbb{R}^{d_1\cdots d_M}$ defined by the equality
    \begin{equation}
        A_{i_1,i_2,\ldots,i_d} = \operatorname{vec}\pars{A}_{\sum_{k\in\bracs{M}} i_k D_k}
        \qquad\text{with}\qquad
        D_k := \prod_{\ell=k+1}^{M}d_\ell
        \qquad\text{and}\qquad
        D_M := 1 .
    \end{equation}
\end{definition}

\begin{definition}[Unfolding~\cite{oseledets_tensor-train_2011}]
    For any tensor $A\in\mbb{R}^{d_1\times\cdots\times d_M}$ and $k\in\bracs{M}$,
    the $k$-unfolding of $A$ is a matrix $\operatorname{unfold}_k\pars{A}\in \mbb{R}^{C_1\times C_2}$ with $C_1=\prod_{j=1}^{k}d_j$ and $C_2=\prod_{j=k+1}^{M}d_j$,
    defined by the equality $\operatorname{vec}\pars{A} = \operatorname{vec}\pars{\operatorname{unfold}_k\pars{A}}$.
\end{definition}

\begin{definition}[Orthogonality]
    A tensor $A\in\mbb{R}^{d_1\times \cdots\times d_M}$ is called left-orthogonal, if
    $$
        \operatorname{unfold}_{M-1} \pars{A}^\intercal
        \operatorname{unfold}_{M-1} \pars{A} = I .
    $$
    It is called right-orthogonal if
    $$
        \operatorname{unfold}_1 \pars{A}
        \operatorname{unfold}_1 \pars{A}^\intercal = I .
    $$
\end{definition}

\begin{definition}[Contraction]
    Given two tensors $A\in\mbb{R}^{d_1\times\cdots\times d_M}$ and $B\in\mbb{R}^{d_M\times\cdots\times d_N}$,
    we define the \emph{contraction} of $A$ and $B$ along the last dimension of $A$ and the first dimension of $B$ as
    \begin{equation}
        \pars{A \circ B}_{i_1,\ldots,i_{M-1},i_{M+1},\ldots,i_N}
        := \sum_{i_M\in\bracs{d_M}} A_{i_1,\ldots,i_M} B_{i_M,\ldots,i_N} .
    \end{equation}
\end{definition}

% \rephrase{To represent high-dimensional tensors $A\in\mbb{R}^{d_1\times\cdots\times d_M}$ efficiently As the }
The \emph{tensor train} (TT) decomposition~\cite{oseledets_tensor-train_2011} represents a tensor of order $M$ as the contraction of $M$ lower order tensors.
A tensor $A\in\mbb{R}^{d_1\times\cdots\times d_M}$ is said to have a TT representation of \emph{rank} $r\in\mbb{N}^{M-1}$ if
\begin{equation}
    A = A^{\pars{1}} \circ \cdots \circ A^{\pars{M}}
\end{equation}
with \emph{component tensors} $A^{\pars{k}} \in \mbb{R}^{r_{k-1}\times d_k\times r_k}$ and the convention that $r_{-1} = r_M = 1$.
% Note that, due to the convention $r_{-1} = r_M = 1$, the components $A^{\pars{1}}$ and $A^{\pars{M}}$ are actually matrices.
By fixing the second index of every $A^{(k)}$ to $i_k$, we obtain a matrix $A^{(k)}_{i_k}$.
The entries of $A$ can then be computed by
\begin{equation}
    A_{i_1, \ldots, i_M} = A^{(1)}_{i_1} \cdots A^{(M)}_{i_M} .
\end{equation}

Now suppose that the tensor $A$ is $R$-sparse, i.e.\ that there exists a set $J$ of size $R$ such that $A_i \ne 0$ if and only if $i = \pars{i_1, \ldots, i_M}\in J$.
Then $A$ can be represented as the sum of $R$ rank-$1$ tensors,
\begin{equation}
\label{eq:sparse_sum}
    A = \sum_{i\in J} e_{i_1} \otimes\cdots\otimes \pars{A_i e_{i_k}} \otimes\cdots\otimes e_{i_M},
\end{equation}
where $e_{i_l}\in \mbb{R}^{d_l}$ are the standard basis vectors and the choice of the index $k\in\bracs{M}$ is arbitrary.
Since every summand is a TT of rank $1$, the sum~\eqref{eq:sparse_sum} can be represented as a TT of rank $R$.

\begin{lemma}[Section~4.1~in~\cite{oseledets_tensor-train_2011}]
\label{lem:tt_sum}
    Let $A, B \in \mbb{R}^{d_1 \times\cdots\times d_M}$ be two tensors given in TT format
    \begin{equation}
        A_i = A^{(1)}_{i_1} \cdots A^{(M)}_{i_M},
        \qquad
        B_i = B^{(1)}_{i_1} \cdots B^{(M)}_{i_M} .
    \end{equation}
    The sum $C = A + B$ can be represented in TT format with components
    \begin{equation}
        C^{(1)}_{i_1} = \begin{bmatrix} A^{(1)}_{i_1} & B^{(1)}_{i_1} \end{bmatrix},
        \qquad
        C^{(k)}_{i_k} = \begin{bmatrix} A^{(k)}_{i_k} & \\ & B^{(k)}_{i_k} \end{bmatrix},
        \qquad
        C^{(M)}_{i_M} = \begin{bmatrix} A^{(M)}_{i_M}\\[2ex] B^{(M)}_{i_M} \end{bmatrix},
    \end{equation}
    where $k = 2,\ldots,M-1$ and empty spaces denote blocks of zeros of appropriate dimension.
\end{lemma}
The proof of Lemma~\ref{lem:tt_sum} follows directly from the definition of the TT decomposition.
Together with the decomposition~\eqref{eq:sparse_sum} it implies that any $R$-sparse tensor $A\in\mbb{R}^{d_1\times\cdots\times d_M}$ can be represented as a TT of rank $R$.

If $k\not\in\braces{1,d}$, this decomposition can be written as
\begin{equation}
\label{eq:sparse_tt}
    A = P^{(1)}\circ\cdots\circ P^{(k-1)}\circ C\circ P^{(k+1)}\circ\cdots\circ P^{(M)},
\end{equation}
with $P^{(1)}\in\braces{0,1}^{1\times d_1\times R}$,
$P^{(j)}\in\braces{0,1}^{R\times d_j\times R}$ for $1 < j < M$,
$P^{(M)}\in\braces{0,1}^{R\times d_M\times1}$, and $C\in\mbb{R}^{R\times d_k\times R}$.
If $J = \braces{i^1, \ldots, i^R}$, then by the definition of the component tensors in Lemma~\ref{lem:tt_sum} it holds that
\begin{equation}
    \operatorname{unfold}_2\pars{P^{(1)}} = \begin{bmatrix}
        e_{i^1_1} & \cdots & e_{i^R_1}
    \end{bmatrix},
    \qquad
    \operatorname{unfold}_2\pars{P^{(j)}} = \begin{bmatrix}
        e_{i^1_j} & & \\
        & \ddots & \\
        & & e_{i^R_j}
    \end{bmatrix},
    \qquad
    \operatorname{unfold}_1\pars{P^{(M)}} = \begin{bmatrix}
        e_{i^1_M}^\intercal \\ \vdots \\ e_{i^R_M}^\intercal
    \end{bmatrix},
\label{eq:qpm_cores}
\end{equation}
where $C$ exhibits the same sparsity pattern as the corresponding $P^{(j)}$.
Note that all components in this representation are $R$-sparse and that similar representations exist for $k=1$ or $k=d$.

Now consider the case that $i^1_1 = i^2_1 = k$.
Then the column $e_k$ appears at least twice in the matricisation $\operatorname{unfold}_2\pars{P^{(1)}}$ 
resulting in an ambiguous representation of the tensor.
% numerical instabilities during the optimisation.
% This problem occurs in every tensor network and
This is a principal effect of the representation in Lemma~\ref{lem:tt_sum} and is not specific to the sparse TT decomposition.
In classical tensor algorithms, uniqueness of the representation is restored (up to orthogonal transformations) by performing a sequence of rank-revealing QR decompositions on the factors $P^{(j)}$.
However, since the QR decomposition is not guaranteed to preserve sparsity, we introduce a sparse QC decomposition $X=QC$, where $Q$ is orthogonal and sparse and $C$ is sparse.
The idea behind this decomposition is that the image space of $X$ is spanned by those standard basis vectors $e_i$ for which the row vector $e_i^\intercal X$ is non-zero.
We can hence define $Q$ as the sparse orthogonal matrix containing these standard basis vectors as its columns.

To rigorously define this decomposition, recall that any $R$-sparse matrix $A\in\mbb{R}^{n\times m}$ can be represented by the three $R$-tuples
\begin{equation}
    \operatorname{row}\pars{A} \in \bracs{n}^R,
    \qquad
    \operatorname{col}\pars{A} \in \bracs{m}^R
    \qquad\text{and}\qquad
    \operatorname{data}\pars{A} \in \mbb{R}^R .
\end{equation}
Here $\operatorname{row}\pars{A}_i$ and $\operatorname{col}\pars{A}_i$ are the row and column indices of the $i$\textsuperscript{th} non-zero entry in $A$ and $\operatorname{data}\pars{A}_i$ is its value.
Conversely, given three $R$-tuples $r\in\bracs{n}^R$, $c\in\bracs{m}^R$ and $d\in\mbb{R}^R$
such that the pairs $\braces{\pars{r_i, c_i}}_{i\in\bracs{R}}$ are unique,
we can uniquely define an $R$-sparse matrix $\operatorname{coo}\pars{r, c, d}$ with
\begin{equation}
    \operatorname{row}\pars{\operatorname{coo}\pars{r, c, d}} = r,
    \qquad
    \operatorname{col}\pars{\operatorname{coo}\pars{r, c, d}} = c
    \qquad\text{and}\qquad
    \operatorname{data}\pars{\operatorname{coo}\pars{r, c, d}} = d .
\end{equation}
Finally, define for every $R\in\mbb{N}$ the $R$-tuples 
\begin{equation}
    \operatorname{range}\pars{R} := \pars{ 1, \ldots, R}  
    \qquad\text{and}\qquad
    \operatorname{ones}\pars{R} := \pars{1, \ldots, 1}  
\end{equation}
as well as the tuple $\operatorname{unique}\pars{r}$ for every tuple $r\in\mbb{N}^R$, containing only the unique elements of $r$.
As usual, we define for any vector $x\in\mbb{R}^d$ the dimension $\operatorname{dim}\pars{x} := d$.

\begin{definition}
\label{def:sparse_qc}
    Let $X\in\mbb{R}^{n\times m}$ be an $R$-sparse matrix.
    Then the \emph{sparse QC decomposition} $X=QC$ is given by
    \begin{equation}
        Q := \operatorname{coo}\pars{s, \operatorname{range}\pars{r}, \operatorname{ones}\pars{r}} \in \mbb{R}^{n\times r}
        \qquad\text{and}\qquad
        C := Q^\intercal X ,
    \end{equation}
    where $s := \operatorname{unique}\pars{\operatorname{row}\pars{X}}$ and $r := \dim\pars{s} \le R$.
\end{definition}
\begin{lemma}
    Let $X\in\mbb{R}^{n\times m}$ be an $R$-sparse matrix and $X=QC$ be its sparse QC decomposition.
    Then $Q\in\mbb{R}^{n\times r}$ is orthogonal and $r$-sparse with $r\le R$ and $C\in\mbb{R}^{r\times m}$ is $R$-sparse.
\end{lemma}
\begin{proof}
    Recall that $Q := \operatorname{coo}\pars{s, \operatorname{range}\pars{r}, \operatorname{ones}\pars{r}}$ with $s := \operatorname{unique}\pars{\operatorname{row}\pars{X}}$ and $r := \dim\pars{s}$.
    This means that $Q$ is $r$-sparse with $r = \dim\pars{s} \le \dim\pars{\operatorname{row}\pars{X}} = R$.
    Moreover, since the $k$\textsuperscript{th} column of $Q$ is the standard basis vector $e_{s_k}$ and since the indices in $s$ are unique, it follows that $Q$ is orthogonal.
    For the same reason, $C = Q^\intercal X$ is just a version of $X$ with the non-zero rows removed.
    Therefore, $C$ is $R$-sparse.
    % \begin{align}
    %     P &= \operatorname{coo}\pars{\operatorname{row}\pars{X}, \operatorname{range}\pars{R}, \operatorname{ones}\pars{R}} \\
    %     S &= \operatorname{coo}\pars{\operatorname{range}\pars{R}, \operatorname{range}\pars{R}, \operatorname{data}\pars{X}} \\
    %     \tilde{P} &= \operatorname{coo}\pars{\operatorname{range}\pars{R}, \operatorname{col}\pars{X}, \operatorname{ones}\pars{R}}
    % \end{align}
    % \color{gray}
    % Observe that $X = PS\tilde{P}$, $S$ is a diagonal matrix and $P$ as well as $\tilde{P}^\intercal$ are QPMs.
    % Let $P = QT$ be the deparallelisation of $P$ and define $C = TS\tilde{P}$.
    % Since
    % $
    %     T = \begin{bmatrix}
    %         e_{i_1} & \cdots & e_{i_R}
    %     \end{bmatrix}
    % $ and $
    %     \tilde{P}^\intercal = \begin{bmatrix}
    %         e_{j_1} & \cdots & e_{j_R}
    %     \end{bmatrix}
    % $
    % are QPMs, we have that
    % \begin{equation}
    %     C = TS\tilde{P} = \sum_{k\in\bracs{R}} S_{k,k} e_{i_k} e_{j_k}^\intercal
    % \end{equation}
    % is an $R$-sparse matrix.
\end{proof}

Applying the sparse QC decomposition sequentially to the unfoldings of all component tensors results in a TT representation
\begin{equation}
\label{eq:A_sparse_depared}
    A = U^{(1)}\circ\cdots\circ U^{(k-1)}\circ C\circ V^{(k+1)}\circ\cdots\circ V^{(M)} .
\end{equation}
An implementation of this procedure is listed in Algorithm~\ref{alg:sparse_rounding}.
The resulting component tensors $U^{(j)}\in\braces{0,1}^{r_{j-1}\times d\times r_j}$ are $r_j$-sparse and left-orthogonal and the component tensors $V^{(j)}\in\braces{0,1}^{r_{j-1}\times d\times r_j}$ are $r_{j-1}$-sparse and right-orthogonal.
The ranks $r_j$ are uniformly bounded by $R$ and the \emph{core tensor} $C$ remains $R$-sparse.
These properties are summarised in the following definition.

\begin{definition}
\label{def:sparse_rounding}
    A tensor train representation
    \begin{equation}
        A = U^{(1)}\circ\cdots\circ U^{(k-1)}\circ C\circ V^{(k+1)}\circ\cdots\circ V^{(M)} 
    \end{equation}
    is called \emph{sparsely canonicalised with core position $k$} if
    \begin{enumerate}
        \item $U^{(j)}\in\braces{0,1}^{r_{j-1}\times d\times r_{j}}$ are left-orthogonal and $r_j$-sparse for all $1\le j < k$,
        \item $V^{(j)}\in\braces{0,1}^{r_{j-1}\times d\times r_{j}}$ are right-orthogonal and $r_{j-1}$-sparse for all $k < j \le M$ and
        \item $C\in\mbb{R}^{r_{k-1}\times d\times r_k}$ is $\min\braces{r_{k-1},r_k}$-sparse.
        % \item the $\operatorname{unfold}_2\pars{U^{(j)}}$ are left-orthogonal quasi-permutation matrices (QPMs) and
        % \item the $\operatorname{unfold}_1\pars{V^{(j)}}^\intercal$ are left-orthogonal QPMs.
    \end{enumerate}
\end{definition}

\begin{algorithm}[htb]
    \SetKwInput{KwInput}{input}
    \SetKwInput{KwOutput}{output}
    \KwInput{Tensor train representation $A = A^{(1)}\circ\cdots\circ A^{(M)}$, desired core position $k$.}
    \KwOutput{Sparsely canonicalised representation of $A$ with core position $k$.}
    \BlankLine
    
    Initialise $C^{(0)} := I$.\\
    \For{$j = 1$ \KwTo $k-1$}{
        Define $X^{(j)} := C^{(j-1)} \operatorname{unfold}_2\pars{A^{(j)}}$. \\
        Compute the sparse QC decomposition (cf.~Definition~\ref{def:sparse_qc}) $X^{(j)} = Q^{(j)} C^{(j)}$. \\
        Define $\operatorname{unfold}_2\pars{U^{(j)}} := Q^{(j)}$ .
    }
    Initialise $C^{(M+1)} := I$.\\
    \For{$j = M$ \KwTo $k+1$}{
        Define $X^{(j)} := \operatorname{unfold}_1\pars{A^{(j)}} C^{(j+1)}$. \\
        Compute the sparse QC decomposition (cf.~Definition~\ref{def:sparse_qc}) $\pars{X^{(j)}}^\intercal = \pars{Q^{(j)}}^\intercal \pars{C^{(j)}}^\intercal$. \\
        Define $\operatorname{unfold}_1\pars{V^{(j)}} := Q^{(j)}$ .
    }
    Define $C := C^{(k-1)} A^{(k)} C^{(k+1)}$. \\
    \Return{$A = U^{(1)}\circ\cdots\circ U^{(k-1)} \circ C\circ V^{(k+1)}\circ\cdots\circ V^{(M)}$.}
\caption{Sparse canonicalisation}
\label{alg:sparse_rounding}
\end{algorithm}

\subsection{Approximation results}

The deliberations of the preceding section give rise to a model class of tensor trains with sparse component tensors.
Moreover, due to the special structure of the component tensors any weighted summability condition on the full tensor translates into a weighted summability condition on the core tensor.
This is made precise in the subsequent theorem.

\begin{theorem}
\label{thm:sparse_tt}
    Every $R$-sparse tensor $A\in\mbb{R}^{d_1\times\cdots\times d_M}$ can be represented in a sparsely canonicalised TT format 
    \begin{equation}
        A = U^{(1)}\circ\cdots\circ U^{(k-1)}\circ C\circ V^{(k+1)}\circ\cdots\circ V^{(M)} 
    \end{equation}
    with ranks that are uniformly bounded by $R$, independent of the chosen core position $k\in\bracs{M}$.
    Moreover, we can define the operator % $Q : \mbb{R}^{r_{k}\times d_k\times r_{k+1}} \to \mbb{R}^{d_1\times\cdots\times d_M}$ 
    $Q\in \mcal{L}\pars{\mbb{R}^{r_{k}\times d_k\times r_{k+1}}, \mbb{R}^{d_1\times\cdots \times d_M}} \simeq \mbb{R}^{d_1\cdots d_M \times r_{k}d_kr_{k+1}}$ via
    \begin{equation}
    \label{eq:core_basis}
        Q
        = \operatorname{unfold}_{k}\pars{U^{(1)}\circ\cdots\circ U^{(k-1)}}
        \otimes I_{d_k} \otimes
        \operatorname{unfold}_{1}\pars{V^{(k+1)}\circ\cdots\circ V^{(M)}}^\intercal ,
    \end{equation}
    where $\otimes $ denotes the matrix Kronecker product. 
    This means that $A = QC$.
    Interpreted as a matrix, $Q$ is left-orthogonal and its columns are standard basis vectors and for every $q\in\bracs{0,\infty}$ and $\bfomega\in\bracs{0,\infty}^{d_1\times\cdots\times d_M}$ it holds that
    $$
        \norm{A}_{\ell^q_{\bfomega}} = \norm{C}_{\ell^q_{\bfbeta}},
    $$
    where $\bfbeta := Q^\intercal \bfomega$ is a (reshaped) subsequence of $\bfomega$.
    % Then $A = QC$ and for every $q\in\bracs{0,\infty}$ and $\omega\in\bracs{0,\infty}^{d_1\times\cdots\times d_M}$ it holds that $\norm{A}_{\ell^q_\omega} = \norm{C}_{\ell^q_{\beta}}$ with $\beta := Q^\intercal \omega$.
    % Since $A$ is sparsely canonicalised, $\beta$ is a (reshaped) subsequence of $\omega$.
\end{theorem}
\begin{proof}
    $Q$ is a linear operator mapping a component tensor from $\mbb{R}^{r_{k}\times d_k\times r_{k+1}}$ to the space of full tensors $\mbb{R}^{d_1\times\cdots\times d_M}$.
    After vectorising these tensor spaces, we can interpret $Q$ as a matrix $Q\in\braces{0,1}^{d_1\cdots d_M\times r_kd_kr_{k+1}}$.
    We now show that $Q$ is an orthogonal matrix where every column is a standard product basis vector.
    We begin by showing that the matrices
    \begin{equation}
        B_{k}
        := \operatorname{unfold}_{k+1}\pars{U^{(1)}\circ\cdots\circ U^{(k)}}
    \end{equation}
    are left-orthogonal with columns that are standard basis vectors.
    Following the lines of~\cite[Appendix B]{wang_2018_tt_embedding}, this can be proved by induction.
    For $k=1$ the assertion is true by construction of $U^{(1)}$.
    For $k>1$ it holds that $B_k = \pars{I_{r_{k-1}} \boxtimes B_{k-1}} \operatorname{unfold}_{2}\pars{U^{(k)}}$, where $I_{r_{k-1}} \boxtimes B_{k-1}$ denotes the Kronecker product.
    The two matrices $\pars{I_{r_{k-1}} \boxtimes B_{k-1}}$ and $\operatorname{unfold}_{2}\pars{U^{(k)}}$ are left-orthogonal and their columns are standard basis vectors.
    This implies the assertion, since the matrix product preserves these properties.
    A similar argument shows that the matrices
    \begin{equation}
        D_k := \operatorname{unfold}_1\pars{U^{(k)}\circ\cdots\circ U^{(M)}}^\intercal
    \end{equation}
    are left-orthogonal with columns that are standard basis vectors.
    This proves the claim, since $Q = B_{k-1}\otimes I_{d_k}\otimes D_{k+1}$.
    % Let $U^{(j)}$ for $j\ne k$ be component tensors of a tensor train.
    % Then the two following implications hold true.
    % \begin{itemize}
    %     \item If $\operatorname{unfold}_2(U^{(j)})$ are left-orthogonal QPMs for all $j<k$,
    %     then the matrix $\operatorname{unfold}_{k}\pars{U^{(1)}\circ\cdots\circ U^{(k-1)}}$ is a left-orthogonal QPM.
    %     \item If $\operatorname{unfold}_1(U^{(j)})^\intercal$ are left-orthogonal QPMs for all $j>k$,
    %     then the matrix $\operatorname{unfold}_{1}\pars{U^{(k+1)}\circ\cdots\circ U^{(M)}}^\intercal$ is a left-orthogonal QPM.
    % \end{itemize}
    % Following the lines of~\cite[Appendix B]{wang_2018_tt_embedding}, we prove the first assertion by induction.
    % For $k=1$, the assertion is true by assumption.
    % For $k>1$, we define $B_k := \operatorname{unfold}_{k+1}\pars{U^{(1)}\circ\cdots\circ U^{(k)}}$
    % and observe that $B_k = \pars{I_{r_{k-1}} \otimes B_{k-1}} \operatorname{unfold}_{2}\pars{U^{(k)}}$.
    % Since $B_k$ is the product of two left-orthogonal QPMs, it is itself a left-orthogonal QPM.
    % The second assertion follows by, mutatis mutandis, by the same arguments.
\end{proof}

Let $A$ be an $R$-sparse coefficient tensor with $\norm{A}_{\ell^0_\bfomega} \le r$.
Then Theorem~\ref{thm:sparse_tt} ensures that $A = QC$ with
\begin{alignat}{2}
    Q\in\mcal{Q}_{R,k} :=\ 
    &\{\,
    \operatorname{unfold}_{k}\pars{U^{(1)}\circ\cdots\circ U^{(k-1)}}
    \otimes I_{d_k} \otimes
    \operatorname{unfold}_{1}\pars{V^{(k+1)}\circ\cdots\circ V^{(M)}}^\intercal\\
    &\!: U^{(j)}, V^{(j)}\in\braces{0,1}^{r_{j-1}\times d_j\times r_j}\text{ with }r_{j-1},r_j \le R \\
    &\!: \text{all } U^{(j)} \text{ are left-orthogonal and  $r_j$-sparse} \\
    &\!: \text{all } V^{(j)} \text{ are right-orthogonal and $r_{j-1}$-sparse}
    \,\}
\end{alignat}
and
\begin{equation}
    C \in \mathcal{C}_{Q,r,\bfomega} := B_{\ell^0_{Q^\intercal\bfomega}}\pars{0, r} = \braces{C \in\mbb{R}^{r_{k-1}\times d_k\times r_k} \,:\, \norm{C}_{\ell^0_{Q^\intercal \bfomega}}\le r} .
\end{equation}
Since such a representation exists for all $k=1,\ldots,M$, this motivates the definition of the model class
\begin{equation}
    \mcal{M}_{R,r,\bfomega}
    := \bigcap_{k\in\bracs{M}} \bigcup_{Q\in\mcal{Q}_{R,k}} Q\mcal{C}_{Q,r,\bfomega} .
\end{equation}
As before, we identify the set of coefficient tensors $\mcal{M}_{R,r,\bfomega}$ with the corresponding set of functions.
% Theorem~\ref{thm:sparse_tt} allows us to translate the error bounds from sets of sparse vectors to the model class $\mcal{M}_{R,r,\omega}$.
The following corollary then translates the result of Corollary~\ref{cor:ward_weighted_stechkin} to the model class $\mcal{M}_{R,r,\bfomega}$
and shows that it exhibits similar approximation rates as the more classical sets of weighted sparsity.
\begin{corollary}
\label{cor:weighted_sparse_tt}
    Let $\boldsymbol\tau\in\pars{\bracs{0,\infty}^{\mbb{N}}}^{M}$ 
    and $T(R) := \min\braces{\norm{P_{J} \boldsymbol\tau}_{\ell^2}^2 \,:\, \abs{J} = R+1}$ be the sum of the $R+1$ smallest elements in $\boldsymbol\tau^2$.
    Moreover, let $0 < q < p \le 2$ and define $\bfalpha := \boldsymbol\tau^{\pars{2-p}/p}$ and $\bfomega := \boldsymbol\tau^{\pars{2-q}/q}$.
    Then every $\bfv\in\ell^q_\bfomega\pars{\mbb{N}^M}$ can be approximated by a tensor $\tilde{\bfv}\in\mcal{M}_{R,r,\bfomega}$ with accuracy
    \begin{equation}
        \norm{\bfv - \tilde{\bfv}}_{\ell^p_\bfalpha} \le \min\braces{T(R), r}^{-s} \norm{\bfv}_{\ell^q_\bfomega},
        \qquad
        s := \tfrac{1}{q} - \tfrac{1}{p} .
    \end{equation}
\end{corollary}
\begin{proof}
    Let $J_n$ be defined as in Corollary~\ref{cor:ward_weighted_stechkin} and recall that $\bfv$ can be approximated by the $n$-sparse tensor $\tilde{\bfv} := P_{J_n} \bfv$ with an error of at most 
    \begin{equation}
        \norm{\pars{1-P_{J_n}} \bfv}_{\ell^p_\bfalpha}
        \le \pars*{\norm{P_{J_{n+1}}\bfsigma^{-1}}_{\ell^q_\bfomega}^q}^{-s} \norm{\bfv}_{\ell^q_\bfomega}
        = \pars*{\norm{P_{J_{n+1}}\boldsymbol\tau}_{\ell^2}^2}^{-s} \norm{\bfv}_{\ell^q_\bfomega} .
    \end{equation}
    Define $n\pars{r} := \max\braces{n\in\mbb{N} \,:\, \norm{P_{J_{n}}\boldsymbol\tau}_{\ell^2}^2 \le r}$ and choose $n = \min\braces{R, n\pars{r}}$.
    Then $\tilde{\bfv}$ is $R$-sparse and $\boldsymbol\tau$-weighted $r$-sparse and Theorem~\ref{thm:sparse_tt} guarantees that it can be represented in $\mcal{M}_{R,r,\bfomega}$.
    The desired error bounds follows by case distinction.
    If $n=R$ then $\norm{P_{J_{n+1}} \boldsymbol\tau}_{\ell^2}^2 \ge T(R)$ by definition of $T(R)$.
    Hence $\pars{\norm{P_{J_{n+1}}\boldsymbol\tau}_{\ell^2}^{2}}^{-s} \le T(R)^{-s}$.
    If $n=n\pars{r}$ then $\norm{P_{J_{n+1}} \boldsymbol\tau}_{\ell^2}^2 \ge r$ by maximality of $n\pars{r}$.
    Hence $\pars{\norm{P_{J_{n+1}}\boldsymbol\tau}_{\ell^2}^{2}}^{-s} \le r^{-s}$.
\end{proof}

Using the model class $\mcal{M}_{R,r,\bfomega}$, the optimisation~\eqref{eq:min_emp} becomes feasible on product basis.
We use the remainder of this section to provide theoretical guarantees for this optimisation.
To apply Proposition~\ref{prop:empirical_projection_error},
we first show that the model class satisfies the required nestedness property $\mcal{M}_{R,r,\boldsymbol\omega} - \mcal{M}_{R,r,\boldsymbol\omega} \subseteq \mcal{M}_{2R,2r,\boldsymbol\omega}$
and then show $\operatorname{RIP}_{\mcal{M}_{R,r,\bfomega}}\pars{\delta}$ holds with high probability.
% {\color{BurntOrange}
% Let the space $\mcal{V}_{d}$ be spanned by the $L^2$-orthonormal basis functions $\braces{b_k}_{k\in\bracs{d}}$ and recall that every function $v\in\mcal{V}_{d}^{\otimes M} \subseteq \mcal{V}$ can be represented by $\pars{V, B}$ where the $\boldsymbol{V}_k$ are the components of the tensor train representation of the coefficient tensor $\boldsymbol{V}\in\pars{\mbb{R}^d}^M$ and $\boldsymbol{b}\pars{x} = \bracs{b_1\pars{x}, \ldots, b_d\pars{x}}^\intercal$ denotes the vector of basis functions evaluated at $x$.
% }
For this to make sense, we let $b : Y\to\mbb{R}^{d}$ be a vector of $L^2\pars{Y,\rho}$-orthonormal basis functions,
define the tensor product basis $B\pars{y} := b\pars{y_1}\otimes \cdots\otimes b\pars{y_M}$ and suppose that the weight sequence $\bfomega$ satisfies $\bfomega_j \ge \norm{B_j}_{L^\infty}$.
We now identity the space of coefficents $v \in \mcal{M}_{R,r,\bfomega}$ with the corresponding space of functions $y\mapsto\pars{B\pars{y}, v}_{\ell^2}$.

\begin{proposition}
    It holds that $\mcal{M}_{R,r,\boldsymbol\omega} - \mcal{M}_{R,r,\boldsymbol\omega} \subseteq \mcal{M}_{2R,2r,\boldsymbol\omega}$.
\end{proposition}
\begin{proof}
    Let $A,B\in\mcal{M}_{R,r,\boldsymbol\omega}$ with core position $k$.
    By Lemma~\ref{lem:tt_sum}, the difference $C = A - B$ can be represented in TT format with components
    \begin{equation}
        C^{(1)}_{i_1} = \begin{bmatrix} A^{(1)}_{i_1} & -B^{(1)}_{i_1} \end{bmatrix},
        \qquad
        C^{(j)}_{i_j} = \begin{bmatrix} A^{(j)}_{i_j} & \\ & B^{(j)}_{i_j} \end{bmatrix},
        \qquad
        C^{(M)}_{i_M} = \begin{bmatrix} A^{(M)}_{i_M}\\[2ex] B^{(M)}_{i_M} \end{bmatrix} ,
    \end{equation}
    for $j = 2,\ldots,M-1$.
    After removing duplicate columns in the matricisations of $C^{(j)}$ for $j\ne k$, the resulting tensor satisfies the unweighted sparsity, orthogonality and rank constraints of $\mcal{M}_{2R,2r,\boldsymbol\omega}$.
    And since
    $$
        \norm{C^{(k)}}_{\ell^0_{Q^\intercal\boldsymbol\omega}}
        = \norm{A-B}_{\ell^0_{\boldsymbol\omega}}
        \le \norm{A}_{\ell^0_{\boldsymbol\omega}}
        + \norm{B}_{\ell^0_{\boldsymbol\omega}}
        \le 2r,
    $$
    the weighted sparsity constraints are satisfied as well.
\end{proof}

The following immediate consequence of Theorem~\ref{thm:sparse_RIP_weighted} provides a bound for the probability of $\operatorname{RIP}_{\mcal{M}_{R,r,\bfomega}}\pars{\delta}$.

\begin{corollary}
\label{cor:sparse_sample_complexity}
	Fix parameters $\delta,\gamma\in\pars{0,1}$.
	Let $\braces{B_j}_{j\in\bracs{D}}$ be orthonormal with respect to the measure $\rho$ and let $w\ge 0$ be any weight function satisfying $\norm{w^{-1}}_{L^1} = 1$.
	Assume the weight sequence satisfies $\bfomega_j \ge \norm{w^{1/2} B_j}_{L^\infty}$ and fix
	\begin{equation}
		n \ge C \delta^{-2} r\max\braces{\log^3\pars{r}\log\pars{d^M}, -\log\pars{\gamma}} .
	\end{equation}
	Let $y_1,\ldots, y_n$ be drawn independently from $w^{-1}\rho$.
	% and define $\norm{f}_{w,m}^2 := \sum_{i=1}^m w\pars{x_i}f\pars{x_i}^2$ for any function $f\in L^2\pars{\rho}$.
	Then the probability of $\operatorname{RIP}_{\mcal{M}_{R,r,\bfomega}}\pars{\delta}$ exceeds $1-\gamma$.
\end{corollary}
\begin{proof}
    Theorem~\ref{thm:sparse_tt} guarantees that every $A\in\mcal{M}_{R,r,\bfomega}$ can be written as $A = QC$ with $\norm{A}_{\ell^0_\bfomega} = \norm{C}_{\ell^0_\bfbeta} \le r$ and where $\bfbeta = Q^\intercal \bfomega$.
    This implies that $\mcal{M}_{R,r,\bfomega}\subseteq B_{\ell^0_\bfomega}\pars{0, r}$.
    The assertion follows, since Theorem~\ref{thm:sparse_RIP_weighted} implies $\operatorname{RIP}_{B_{\ell^0_\bfomega}\pars{0,r}}\pars{\delta}$ and, consequently, $\operatorname{RIP}_{\mcal{M}_{R,r,\bfomega}}\pars{\delta}$.
    % Recall that every $A\in\mcal{M}_{R,r,\omega}$ can be written as $A = QC$ with $\norm{C}_{\ell^0_\beta} \le r$ and where $Q$ and $\beta$ are defined as in Theorem~\ref{thm:sparse_tt}.
    % This theorem guarantees $\norm{A}_{\ell^0_\omega} = \norm{C}_{\ell^0_\beta} \le r$ and hence $\mcal{M}_{R,r,\omega}\subseteq B_r\pars{\ell^0_\omega}\pars{\bracs{d}^M}$.
    % This bound is tight since $\mcal{M}_{R,r,\omega} = B_r\pars{\ell^0_\omega}\pars{\bracs{d}^M}$ for $R\ge r$.
\end{proof}

\subsection{Numerical method}

In order to present an efficient numerical realisation of the optimisation problem~\eqref{eq:min_emp}, define the vector $F\in\mbb{R}^n$ and the bounded linear operator $M : \pars{\mbb{R}^{d}}^{\otimes M} \to \mbb{R}^n$ by
\begin{equation}
\label{eq:F_and_M}
    F_i = \sqrt{w\pars{y^i}} u\pars{y^i}
    \qquad\text{and}\qquad
    \pars{Mv}_i = \sqrt{w\pars{y^i}}\pars{v, B\pars{y^i}}_{\mathrm{Fro}} .
\end{equation}
Then equation~\eqref{eq:min_emp} is equivalent to the optimisation problem
\begin{equation} \label{eq:min_emp_coef}
    \operatorname*{minimise}_{v\in\mcal{M}_{R,r,\bfomega}}\ \norm{F - Mv}_{\ell^2}^2 .
\end{equation}
We propose to solve this problem by a sparse variant of the \emph{alternating least squares} (ALS) algorithm introduced in~\cite{oseledets_2011_tensor_trains,holtz_alternating_2012}.
The ALS method solves~\eqref{eq:min_emp_coef} by refining an initial guess in a sequence of \emph{microsteps}, each optimising a single component tensor while keeping the others fixed.
Since every $v\in\mcal{M}_{R,r,\bfomega}$ can be written as a sparsely canonicalised tensor train $v = QC$ with core position $k$, the microstep optimising the $k$\textsuperscript{th} component tensor can be written as 
\begin{equation}
    \operatorname*{minimise}_{C\in\mcal{C}_{Q,r,\bfomega}}\ \norm{F - MQC}_{\ell^2}^2 .
\label{eq:min_emp_coef_core}
\end{equation}
The operator $Q$ can be efficiently computed by Algorithm~\ref{alg:sparse_rounding}\footnote{Indeed the operator $Q$ at core position $k$ can be efficiently updated from its value at position $k-1$ or $k+1$ by means of a single sparse QC decomposition.} and the resulting sparse tensor train representation allows for an efficient evaluation of $MQ$.
A classical approach to handle the weighted sparsity constraints in $\mcal{C}_{Q,r,\bfomega} = \braces{C \in\mbb{R}^{r_{k-1}\times d_k\times r_k} \,:\, \norm{C}_{\ell^0_{Q^\intercal \bfomega}}\le r}$ is to promote the weighted $\ell^0$-constraints via a weighted $\ell^1$-regularisation term.
The resulting problem then reads
\begin{equation} \label{eq:weighted_lasso}
    \operatorname*{minimise}_{C\in\mbb{R}^{r_{k-1}\times d_k\times r_k}}\ \norm{F - MQC}_{\ell^2}^2 + \lambda \norm{\bfbeta_Q\odot C}_1 ,
\end{equation}
where $\bfbeta_Q := Q^\intercal \bfomega$ can be efficiently computed due to the tensor train representation and sparsity structure of $Q$.
Substituting $D = \beta_{Q}\odot C$ into~\eqref{eq:weighted_lasso} we obtain the standard LASSO problem~\cite{Santosa1986,Efron2004}
\begin{equation} \label{eq:std_lasso}
    \operatorname*{minimise}_{D\in\mbb{R}^{r_{k-1}\times d_k\times r_k}}\ \norm{F - MQ\pars{\bfbeta_Q^{-1}\odot D}}_{\ell^2}^2 + \lambda \norm{D}_1 .
\end{equation}

The regularisation parameter $\lambda$ controls the sparsity of $C$ and must be chosen appropriately to remain in the model class $\mcal{M}_{R,r,\boldsymbol{\omega}}$.
To do this recall the following two facts.
\begin{enumerate}
    \item Theorem~\ref{thm:sparse_tt} implies $\norm{C}_{\ell^0_{\bfbeta_Q}} = \norm{\bfv}_{\ell^0_\bfomega}$, which ensures that the weighted sparsity constraint is satisfied for all components as soon as it is satisfied for the optimised component.
    \item 
    Theorem~\ref{thm:sparse_tt} implies $\norm{C}_{\ell^0} = \norm{\bfv}_{\ell^0}$, which ensures that the rank $R$ is bounded by the number of nonzero entries of the core tensor $\norm{C}_{\ell^0}$.
\end{enumerate}
It is thus sufficient to choose $\lambda$ such that $\norm{C}_{\ell^0_{\bfbeta_Q}} \le r$ and $\norm{C}_{\ell^0} \le R$ to remain in $\mcal{M}_{R,r,\bfomega}$ during optimisation.
Although this would be easy to implement, we propose to choose $\lambda$ by $10$-fold cross-validation instead.
This allows the algorithm to choose a different regularisation parameter $\lambda$, i.e.\ a different sparsity level, for every core position $k$.
Moreover, since the rank $R$ in the sparsely canonicalised representation depends on the sparsity of the solution of the microstep, the resulting algorithm is inherently rank-adaptive.
We call this algorithm \emph{sparse alternating least-squares} (SALS) since it modifies a standard ALS method to work on sparse tensors.
A listing of the complete algorithm, in pseudo-code, is provided in Algorithm~\ref{alg:sparse_als}.
There it can be seen that the algorithm differs from a standard ALS only in two points.
\begin{enumerate}
    \item The standard regression in the microstep is replaced by a (weighted) LASSO.
    \item The rank revealing QR decomposition, commonly used to compute the operators $Q$, is replaced by a sparse QC decomposition.
\end{enumerate}
It is therefore straight-forward to implement.

\begin{algorithm}[t!]
    % \color{MidnightBlue}
    \SetKwInput{KwInput}{input}
    \SetKwInput{KwOutput}{output}
    \KwInput{Data pairs $(x^{i},y^{i})\in\mathbb{R}^M\times \mathbb{R}$ for $i = 1,\ldots,n$, univariate basis functions $\braces{b_1,\ldots, b_d}$, and weight sequences $\omega_m\in\mbb{R}^d$ for $m=1,\ldots,M$.}
    \KwOutput{A coefficient tensor $\bfv\in\mcal{M}$ such that $y\mapsto\pars{\bfv, B\pars{y}}_{\mathrm{Fro}}$ approximates the data.}
    \BlankLine

    Initialize the coefficient tensor $v$.\\
    \While{not converged}{
        \For{$k=1$ \KwTo $M$}{
            Compute the sparse canonicalisation~\eqref{eq:A_sparse_depared} with core position $k$. \\
            Compute $Q$ as in~\eqref{eq:core_basis} and $\bfbeta_Q := Q^\intercal\bfomega$. \\
            Update $C$ by solving equation~\eqref{eq:std_lasso} and use cross-validation to select $\lambda$.
        }
    }
    \Return{$\boldsymbol{v}$.}
\caption{Sparse Alternating Least Squares (SALS)}
\label{alg:sparse_als}
\end{algorithm}

% \begin{algorithm}[t!]
%     \color{MidnightBlue}
%     \SetKwInput{KwInput}{input}
%     \SetKwInput{KwOutput}{output}
%     \KwInput{Data pairs $(x^{i},y^{i})\in\mathbb{R}^M\times \mathbb{R}$ for $i = 1,\ldots,n$, univariate basis functions $\braces{b_1,\ldots, b_d}$, and univariate Gramians $G_m$ for $m=1,\ldots,M$.}
%     \KwOutput{Coefficient tensor $\boldsymbol{V}$ of a function $v\in\mcal{M}$ that approximates the data.}
%     \BlankLine
    
%     Initialize the coefficient tensor $\boldsymbol{V}$.\\
%     \While{not converged}{
%         Right orthogonalise $\boldsymbol{V}$.\\
%         \For{$m=1,\ldots,M$}{
%             Compute the Gramian $G$ according to equation~\eqref{eq:gramian}.\\
%             Compute the spectral decomposition of $G=U\operatorname{diag}\pars{\beta_{\xhat{V}_m}^2}U^\intercal$.\\
%             Update $\boldsymbol{V}_m = U\beta_{\xhat{V}_m}^{-1}W$ by solving equation~\eqref{eq:std_lasso} using cross-validation.\\
%             Left orthogonalise $\boldsymbol{V}_m$ and adapt the $m$\textsuperscript{th} rank.\\
%         }
%     }
%     \Return{$\boldsymbol{V}$.}
% \caption{Restricted Alternating Least-Squares (RALS)}
% \label{alg:rals}
% \end{algorithm}
\section{Low-rank and sparse Tensor Train approximation}
\label{sec:semisparse_als}

Despite its straightforward sample bound and built-in rank adaptivity, the sparse tensor train model class from the previous section and the associated  Algorithm~\ref{alg:sparse_als} are not optimal, since the resulting tensor representation does not have minimal rank in general.
% The sparse tensor train model class from the preceding section does not allow to exploit the potential low rank of weighted $\ell^q$-summable tensors.
Motivated by promising practical results with low-rank tensor reconstructions for holomorphic functions as considered in~\cite{trunschke21,eigel2019variational}, this section introduces a new tensor train format which incorporates sparsity and low-rank.

To illustrate the advantage of this new format, we consider the approximation of the rank-$1$ function $x\mapsto \exp\pars{x_1 + \ldots + x_M}$ by Legendre polynomials in Appendix~\ref{app:low-rank_advantage}.
The remainder of this section is devoted to investigating this idea in the general setting.

\subsection{Approximation results}

To obtain an operator $Q$ which still allows for a meaningful concept of sparsity in the component tensor $C$, we replace the sparse QC decomposition from the preceding section with an $\boldsymbol{\omega}$-weighted QC decomposition.
\begin{definition}
    We say that a matrix $Q$ is \emph{$\boldsymbol{\omega}$-orthogonal} if $Q^\intercal\operatorname{diag}\pars{\boldsymbol{\omega}}Q$ is diagonal.
\end{definition}
\begin{definition}
% \label{def:sparse_qc}
    Let $A\in\mbb{R}^{n\times m}$ be a rank-$r$ matrix.
    A \emph{$\boldsymbol{\omega}$-orthogonal QC decomposition of $A$} is a decomposition $A=QC$ with $Q\in\mbb{R}^{n\times r}$ and $C\in\mbb{R}^{r\times m}$
    for which $Q$ is orthogonal and $\boldsymbol{\omega}$-orthogonal, i.e.\ $Q^\intercal Q = I$, and $Q^\intercal\operatorname{diag}\pars{\boldsymbol{\omega}}Q$ is diagonal.
\end{definition}

Even though this new decomposition may not retain the sparsity as well as the sparse QC decomposition did,
the resulting factors still exhibit a considerable amount of sparsity.

\begin{lemma}
\label{lem:not_so_sparse}
    Let $A\in\mbb{R}^{n\times m}$ be a rank-$r$ matrix.
    Then there exists an $\boldsymbol{\omega}$-orthogonal QC-decomposition $A=QC$.
    This decomposition is unique up to reordering of the columns of $Q$.
    Moreover, if $A$ is $R$-sparse then $r\le R$ and $Q$ and $C$ are $Rr$-sparse.
    (Note that the complexity is independent of $n$ and $m$.)
    % \rephrase{Let $X\in\mbb{R}^{n\times m}$ be an $R$-sparse matrix and $X=QC$ be its $\boldsymbol{\omega}$-orthogonal QC decomposition.
    % Then $Q\in\mbb{R}^{n\times r}$ and $C\in\mbb{R}^{r\times m}$ are $rR$-sparse with $r\le\min\braces{m,n,R}$.
    % Moreover, $Q$ is orthogonal and $Q^\intercal\operatorname{diag}\pars{\boldsymbol{\omega}}Q$ is diagonal.}
\end{lemma}
\begin{proof}
    Let $A = Q_1C_1$ be the sparse QC decomposition of $A$ and let $C_1 = Q_2 C_2$ be the QR decomposition of $C_1$.
    Moreover, let $Q_{12} := Q_1Q_2$ and $U\Lambda U^\intercal$ be the spectral decomposition of $Q_{12}^\intercal \operatorname{diag}\pars{\boldsymbol{\omega}} Q_{12}$,
    define $Q := Q_{12} U$ and $C := U^\intercal C_2$.
    Then $A = QC$ by construction and it holds that
    $$
        Q^\intercal \operatorname{diag}\pars{\boldsymbol{\omega}}Q
        = U^\intercal \pars{Q_{12}^\intercal \operatorname{diag}\pars{\boldsymbol{\omega}} Q_{12}} U
        = U^\intercal \pars{U\Lambda U^\intercal} U
        = \Lambda .
    $$
    Note that $Q$ is a product of three orthogonal matrices and hence orthogonal.
    Since the QR decompostion $A = Q_{12} C_2$ is unique and since the spectral decomposition of $Q_{12}^\intercal \operatorname{diag}\pars{\boldsymbol{\omega}} Q_{12}$ is unique up to reordering of the columns $U$,
    the matrix $Q = Q_{12} U$ is unique up to reordering of its columns.
    Now suppose that $A$ is $R$-sparse.
    Then $Q_1\in\braces{0,1}^{n\times \tilde{R}}$ with $\tilde{R}\le R$.
    This means that $C_1 \in\mbb{R}^{\tilde{R}\times m}$, which yields the standard bound $r\le \min\braces{\tilde{R}, m} \le R$.
    Moreover, since the columns of $Q_1$ are standard basis vectors, only the rows in $\operatorname{row}\pars{Q_1}\in\bracs{n}^{\tilde{R}}$ are nonzero.
    Consequently, only the same rows can be nonzero in the product $Q = Q_1 \pars{Q_2 U}$.
    In the worst case, all of the $r$ columns of $Q$ become nonzero for every of these $\tilde{R}$ rows.
    This yields a total of $\tilde{R}r \le Rr$ nonzero entries.
    To obtain a sparsity bound for $C$, let $A^\intercal = \tilde{Q}\tilde{C}$ be the sparse QC decomposition and $\tilde{C}^\intercal = Q\bar{C}$ be the $\boldsymbol{\omega}$-weighted QC decomposition.
    Now define $C = \bar{C}\tilde{Q}^\intercal$ and observe that $A = QC$ is a valid $\boldsymbol{\omega}$-weighted QC decomposition.
    Since the rows of $\tilde{Q}^\intercal$ are standard basis vectors, only the columns in $\operatorname{col}\pars{\tilde{Q}}\in\bracs{m}^{\tilde{R}}$ are nonzero.
    Consequently, only the same columns can be nonzero in the product $C = \bar{C}\tilde{Q}^\intercal$.
    In the worst case, all of the $r$ rows of $C$ will be nonzero for every of these $\tilde{R}$ columns.
    This yields a total of $\tilde{R}r \le Rr$ nonzero entries.
\end{proof}

Applying the $\boldsymbol{\omega}$-weighted QC decomposition sequentially to the unfoldings of all component tensors results in a TT representation
\begin{equation}
\label{eq:A_omega_orth}
    A = QC = U^{(1)}\circ\cdots\circ U^{(k-1)}\circ C\circ V^{(k+1)}\circ\cdots\circ V^{(M)} ,
\end{equation}
where the component tensors $U^{(j)}\in\mbb{R}^{r_{j-1}\times d\times r_j}$ are $r_jR$-sparse and left-orthogonal and the component tensors $V^{(j)}\in\mbb{R}^{r_{j-1}\times d\times r_j}$ are $r_{j-1}R$-sparse and right-orthogonal, the ranks $r_j$ are uniformly bounded by $R$ and the \emph{core tensor} $C$ remains $R$-sparse.
An implementation of this procedure can be obtained from Algorithm~\ref{alg:sparse_rounding} by replacing all sparse QC decompositions with $\boldsymbol{\omega}$-weighted QC decompositions.
%
% \begin{definition}
% % \label{def:sparse_rounding}
%     A tensor train representation of the form~\eqref{eq:A_omega_orth}
%     % \begin{equation}
%     %     A = U^{(1)}\circ\cdots\circ U^{(k-1)}\circ C\circ V^{(k+1)}\circ\cdots\circ V^{(M)} .
%     % \end{equation}
%     is called \emph{$\boldsymbol{\omega}$-canonicalised with core position $k$}, if the operator
%     $$
%         Q\in\mcal{L}\pars{\mbb{R}^{r_{k}\times d_k\times r_{k+1}}, \mbb{R}^{d_1\cdots d_{k-1}\times d_k\times d_{k+1}\cdots d_M}},
%     $$
%     interpreted as a matrix in $\mbb{R}^{d_1\cdots d_M \times r_{k}d_kr_{k+1}}$,
%     is orthogonal and $\boldsymbol{\omega}$-orthogonal.
%     % if
%     % \begin{enumerate}
%     %     \item $U^{(j)}\in\mbb{R}^{r_{j-1}\times d\times r_{j}}$ are left-orthogonal and $\boldsymbol{\omega}$-orthogonal for all $1\le j < k$,
%     %     \item $V^{(j)}\in\mbb{R}^{r_{j-1}\times d\times r_{j}}$ are right-orthogonal and $\boldsymbol{\omega}$-orthogonal for all $k < j \le M$ and
%     %     \item $C\in\mbb{R}^{r_{k-1}\times d\times r_k}$ is $\min\braces{r_{k-1},r_k}$-sparse.
%     % \end{enumerate}
% \end{definition}
%
Analogously to the model class $\mcal{M}_{R,r,\boldsymbol{\omega}}$,
which is based on the sparse QC decomposition, we define the model class 
\begin{equation}
\label{eq:MRromega}
    \tilde{\mcal{M}}_{R,r,\boldsymbol{\omega}}
    := \bigcap_{k\in\bracs{M}} \bigcup_{Q\in\tilde{\mcal{Q}}_{k,R,\boldsymbol{\omega}^3}} Q\mcal{C}_{Q,r,\boldsymbol{\omega}^3},
\end{equation}
which is based on the $\boldsymbol{\omega}$-weighted QC decomposition.
The elements of this new model class are tensors $A=QC$ with $Q\in\tilde{\mcal{Q}}_{k, R, \boldsymbol{\omega}}$ and $C\in \mathcal{C}_{Q,r,\boldsymbol{\omega}}$, where
\begin{equation}
    \tilde{\mcal{Q}}_{k,R,\boldsymbol{\omega}} := \Set*{
    Q\in\mcal{L}\pars{\mbb{R}^{r_{k}\times d_k\times r_{k+1}}, \mbb{R}^{d_1\times \cdots\times d_M}}
    \given
    \text{$Q$ is orthogonal and $\boldsymbol{\omega}^2$-orthogonal and $r_k, r_{k+1} \le R$}
    }
\end{equation}
and
\begin{equation}
    \mathcal{C}_{Q,r,\boldsymbol{\omega}}
    := \braces{C \in\mbb{R}^{r_{k-1}\times d_k\times r_k} \,:\, \norm{C}_{\ell^0_{\boldsymbol{\beta}_Q}}\le r \text{ with } \boldsymbol{\beta}_Q^2 := \operatorname{diag}\pars{Q^\intercal\operatorname{diag}\pars{\boldsymbol{\boldsymbol{\omega}}^2}Q}} .
\end{equation}
Note that the new definition of $\mcal{C}_{Q,r,\bfomega}$ is a generalisation of the old definition to cases where the columns of $Q$ are not standard basis vectors.
%\todo{Define a new model class $\tilde{\mcal{C}}_{\omega}$ and define $\tilde{\mcal{M}}$ in terms of $Q_\omega$ and $C_\omega$ instead of $Q_{\omega^3}$...?
%Then $\tilde{\mcal{C}}$ would not be a generatlisation of $\mcal{C}$ but the proof of Theorem~\ref{thm:RIP_MRromega} would be more straightforward.}
Note that the definition of $\tilde{\mcal{M}}_{R,r,\bfomega}$ corresponds to the choice of
\begin{itemize}
    \item a basis for the core space $Q$ as well as
    \item a weight sequence $\bfbeta_Q$.
\end{itemize}
In Theorem~\ref{thm:RIP_MRromega} we show that this choice ensures that a tensor with an $\ell^0_{\bfbeta_Q}$-sparse core is close to a sparse vector in $\ell^0_\bfomega$.
This is quite surprising since for any sparse core $C$ there exists an easy to construct $C$-dependent orthogonal basis $U$ such that all coefficients of the full tensor $UC$ are equal to $\norm{C}_{2}$.
This means that $UC$ is the least sparse tensor possible.
However, using information about the weight sequence $\bfomega$, we can construct a basis $Q$ and a weight sequence $\bfbeta_Q$ such that the full tensor $QC$ retains some of the sparsity of the core $C$.

\begin{remark}
    Since $\mcal{M}_{R,r,\bfomega}\subseteq \tilde{\mcal{M}}_{R,r,\bfomega}$, the approximation error for $\tilde{\mcal{M}}_{R,r,\bfomega}$ can be bounded from above by Corollary~\ref{cor:weighted_sparse_tt}.
    To obtain a tighter bound, the total approximation error can be split into the low-rank approximation error and a subsequent weighted best $n$-term approximation of the core tensor
    $$
        \norm{v - v_{\mathrm{low\texthyphen rank\,\&\,sparse}}}_{L^2}
        \le \norm{v - v_{\mathrm{low\texthyphen rank}}}_{L^2}
        + \norm{v_{\mathrm{low\texthyphen rank}}\,-\,v_{\mathrm{low\texthyphen rank\,\&\,sparse}}}_{L^2} .
    $$
    % Suppose we are given a function $v$ and a dense rank-$r$ tensor train approximation $v_r$ with
    % $$
    %     \norm{v - v_r}_{H^{k,\mathrm{mix}}} \le c\pars{r} \norm{v}_{H^{\ell,\mathrm{mix}}} .
    % $$
    % Then we know that every component tensor represents a function in $H^{k,\mathrm{mix}}$ and we can use the arguments from Example~\ref{ex:Hk} to provide sparse approximation errors.
    The first term is a classical low-rank approximation error, which is studied in~\cite{Schneider201456,bachmayr2021,bigoni_2016_spectral_tt} for $v\in H^k([0, 1]^m)$ and in~\cite{bachmayr2021,griebel_2022_low_rank} for $v\in H^{k_1}([0, 1]^{d_1}) \otimes \cdots \otimes H^{k_m}([0, 1]^{d_m})$.
    % These references, however, only provide bound the low-rank approximation error in $L^2$ and we are not aware of any error bounds in $H^{k}$.
    % However, since the analysis in~\cite{griebel_2022_low_rank} is based on sparse tensor product space it would be interesting to investigate the relation with the theory from this paper.
    The second term is a sparse approximation error, which can in principle be bounded by applying the weighted Stechkin's lemma to the core tensor $C$ of $v_{\mathrm{low\texthyphen rank}} = QC$.
    This gives the bound
    $$
        \norm{v_{\mathrm{low\texthyphen rank}}\,-\,v_{\mathrm{low\texthyphen rank\,\&\,sparse}}}_{L^2}
        = \norm{\pars{I-P_{J_n}}C}_{\ell^2}
        \le c\pars{\bfbeta, q, n}^{-1} \norm{C}_{\ell^q_\bfbeta}
        \le \bar{c}\pars{\bfbeta, q, n}^{-1} \norm{C}_{\ell^2_{\bar\bfbeta}}
        % \le c\pars{\bfbeta, q, n}^{-1} \norm{Q^\intercal}_{\ell^q_{\bar\bfbeta} \to \ell^q_\bfbeta} \norm{A}_{\ell^q_{\bar{\bfbeta}}},
    $$
    for some $q<2$ and weight sequences $\bfbeta$ and $\bar\bfbeta$.
    However, bounding $\norm{C}_{\ell^2_{\bar\bfbeta}}$ in terms of some norm of the full tensor $\norm{QC}_{\ell^2_{\bar\bfomega}}$, for some arbitrary $\bar\bfomega$, requires knowledge of the operator norm $\norm{Q^\intercal}_{\ell^2_{\bar\bfbeta}\to\ell^2_{\bar\bfomega}}$, which is unknown a priori.
    However, if $Q$ is $\bar\bfomega$-orthogonal and $\bar\bfbeta = Q^\intercal \bar\bfomega$, then the low-rank approximation can be carried out with respect to the stronger $\ell^2_{\bar\bfomega}$-norm and we can bound
    $$
        \norm{C}_{\ell^2_{\bar\bfbeta}}^2
        = C^\intercal \operatorname{diag}\pars{\bar\bfbeta^2} C
        = (QC)^\intercal \operatorname{diag}\pars{\bar\bfomega^2} QC
        = \norm{v_{\mathrm{low\texthyphen rank}}}_{\ell^2_{\bar\bfomega}}^2
        \le \norm{v}_{\ell^2_{\bar\bfomega}}^2 .
    $$
    But this requires the low-rank approximation to be carried out with respect to a stronger norm than $L^2$.
    To the knowledge of the authors no rates for this are known.
    % For $q \le 1$, this can be theoretically computed via
    % \begin{equation}
    %     \norm{Q_V^\intercal}_{\ell^q_\bfomega\to\ell^q_\bfbeta} = \norm{\bfomega^{-1}\delta}_{\ell^\infty} ,
    % \end{equation}
    % where $\delta_k := \norm{Q^\intercal e_k}_{\ell^q_\bfbeta}$.
    % But it is clear that the operator norm depends on the choice of $Q$, which is unknown a priori.
    % \color{BurntOrange}
    % To obtain a rough bound, we can use Lemma~\ref{lem:trade_sparsity} to estimate
    % $$
    %     \norm{\pars{I-P}C}_{\ell^p_\alpha}
    %     \le c\pars{n}^{-s'} \norm{C}_{\ell^q_{\beta'}}
    %     \lesssim c\pars{n}^{-s'} \norm{C}_{\ell^2_\beta}
    %     = c\pars{n}^{-s'} \norm{A}_{\ell^2_\omega}
    %     \lesssim c\pars{n}^{-s'} \norm{A}_{\ell^{q'}_{\omega'}} ,
    % $$
    % where we have used the definition of $\beta_Q$ from Theorem~\ref{thm:RIP_MRromega}.
\end{remark}

\begin{remark}[Rank bounds for mixed Sobolev spaces]
    Consider a function $u$ of $M$ variables and the corresponding sequence of coefficients $\bfu\in \mbb{R}^{\mbb{N}}\otimes \cdots \otimes \mbb{R}^{\mbb{N}}$ with respect to a tensor product basis.
    Moreover, suppose that $\bfu$ is $\ell^2_{\bar\bfomega}$-summable with respect to the product weight sequence $\bar\bfomega := \bfomega^{\otimes M}$ with $\bfomega_j := \pars{j+1}^{k}$. 
    The space $\ell^2_{\bar\bfomega}$ captures the regularity of the mixed Sobolev spaces $H^{k,\mathrm{mix}} \simeq H^k\otimes \cdots \otimes H^k$.
    To bound the rank of $u$ by means of weighted sparsity we utilise the best $n$-term approximation rates for $\bfu$ from Remark~\ref{rmk:best-n-term_hd}.
    Recall that $r$-term approximation in a product basis can be represented with rank $r$ in the CP format and that the rank of any tree-based format is upper bounded by the CP rank (but may indeed be much smaller).
    This na\"ive bound yields (up to logarithmic factors) the best rank-$r$ approximation rates
    % $$
    %     \norm{u - u_r}_{L^2\pars{\gamma^{\otimes M}}} \le r^{-\pars{k+1/2}} \norm{u}_{H^k\pars{\gamma^{\otimes M}}}
    % $$
    % and 
    $$
        \norm{u - u_r}_{H^j} \le r^{-(k-j)} \norm{u}_{H^k}
    $$
    for all $0\le j < k$.
    For $j=0$, these bounds slightly extend the rates from~\cite{griebel_harbrecht_2013,Griebel2018} but are worse than the more recent $r^{-2k}$ rates that are derived in~\cite{griebel_2022_low_rank} for the rank in the tensor train format.
    We conjecture that sparsity implies simple rank bounds with sparse components and a subsequent rounding can reduce the rank from $r^2$ to $r$.
    
    Note, however, that both rank bounds imply roughly the same approximation rates.
    By Theorem~\ref{thm:sparse_tt}, the number of parameters that are needed to represent the best $r$-term approximation in the sparse tensor train format is bounded by $N = Mr$.
    This implies an approximation rate of $\pars{\frac{N}{M}}^{-k}$.
    If we consider the best rank-$r$ rate $r^{-2k}$ from~\cite{griebel_2022_low_rank}, and assume that the component tensors in the corresponding tensor train representation are dense, then the number of parameters scales like $N \in\mcal{O}\pars{Mr^2}$ and we obtain the same approximation rate of $\pars{\frac{N}{M}}^{-k}$.
    
    We also remark that the best $r$-term approximation rates crucially depend on the chosen basis, while the ranks do not.
    Therefore, the rank-$r$ approximation rates that are obtained by this method can only provide upper bounds.
    Moreover, the ordering of the modes matters for the ranks of a tensor train representation.
    This is not reflected in these simple bounds, where the ordering is only important for an anisotropic choice of weight sequences.
\end{remark}

\begin{remark}[Constructive rank bounds for Sobolev spaces]
    Similar to the hierarchical SVD (or PCA) that can be used to construct classical low-rank representation, we can perform the weighted LASSO hierarchically to construct simultaneously sparse and low-rank representations.
    In this remark we demonstrate this procedure for the Tucker decomposition.
    Consider a function $u$ of $M$ variables and the corresponding sequence of coefficients $\bfu\in \mbb{R}^{\mbb{N}}\otimes \cdots \otimes \mbb{R}^{\mbb{N}}$ with respect to a tensor product basis.
    Moreover, suppose that $\bfu$ is $\ell^2_{\bar\bfomega}$-summable with respect to the weight sequence $\bar\bfomega := \sum_{m=1}^M \boldsymbol{1}^{\otimes(m-1)}\otimes \bfomega \otimes \boldsymbol1^{\otimes(M-m)}$ with $\bfomega_j := \pars{j+1}^k$.
    The space $\ell^2_{\bar\bfomega}$ captures the regularity of the standard Sobolev spaces
    % Consider an order-$M$ tensor $A$ which is summable with respect to the weight sequence $\bar\bfomega := \sum_{m=1}^M 1^{\otimes(m-1)}\otimes \omega \otimes 1^{\otimes(M-m)}$.
    % This weight sequence captures the regularity of standard Sobolev spaces
    $$
        H^k \simeq \bigcap_{m=1}^{M} \pars{L^2}^{\otimes (m-1)} \otimes H^k \otimes \pars{L^2}^{\otimes(M-m)} .
    $$
    We can bound the Tucker-rank of $u$ by means of weighted sparsity.
    The method that we use to derive our rank bounds is constructive and proceeds analogously to the HOSVD algorithm.
    We define for every $m=1,\ldots, M$ the matricisation
    $$
        (\bfu^{(m)}_j)_{i_1, \ldots, i_{M-1}} = \bfu_{i_1, \ldots, i_{m-1}, j, i_m, \ldots, i_{M-1}} ,
    $$
    which we interpret as a sequence of tensors of order $M-1$.
    % \paragraph{Tucker rank} Any tensor $\bfu\in (\ell^2)^{\otimes M}$ can be interpreted as a sequence of tensors via
    % $$
    %     (\bfu^{(k)}_j)_{i_1, \ldots, i_{M-1}} = \bfu_{i_1, \ldots, i_{k-1}, j, i_k, \ldots, i_{M-1}} .
    % $$
    Applying the weighted Stechkin lemma to the sequence $\bfu^{(m)}$, we select $r$ many $(M-1)$-dimensional ``slices'' of $\bfu$ and set the remaining slices to zero.
    This results in a new tensor which we denote by $\tilde{\bfu}^{(m)}$.
    This construction can be performed sequentially for every $m=1,\ldots,M$, leading to the sequence of approximations
    $$
        \bfu =: \tilde{\bfu}^{(0)} \leadsto \tilde{\bfu}^{(1)} \leadsto \ldots \leadsto \tilde{\bfu}^{(M)} .
    $$
    The approximation error of this scheme is given by the telescoping sum
    \begin{align}
        \norm{\bfu - \tilde{\bfu}^{(M)}}_{\ell^2}^2
        &= \sum_{m=1}^{M} \norm{\tilde{\bfu}^{(m-1)} - \tilde{\bfu}^{(m)}}_{\ell^2}^2 \\
        &\lesssim \sum_{m=1}^{M} r^{-2k} \norm{(\boldsymbol1^{\otimes(m-1)} \otimes \bfomega\otimes \boldsymbol1^{\otimes (M-m)}) \tilde{\bfu}^{(m-1)}}_{\ell^2}^2 ,
    \end{align}
    where the inequality follows from the weighted Stechkin lemma applied to the tensor-valued sequence $\tilde{\bfu}^{(m-1)}$, which is weighted by $\bfomega$ and a subsequent application of Lemma~\ref{lem:trade_sparsity}.
    Bounding $\norm{(\boldsymbol1^{\otimes(k-1)} \otimes \bfomega \otimes \boldsymbol1^{\otimes M-k-1}) \tilde{\bfu}^{(k)}}_{\ell^2} \le \norm{\bfu}_{\ell^2_{\bar\bfomega}}$ yields the simplified expression
    $$
        \norm{u - u_r}_{L^2}
        = \norm{\bfu - \tilde{\bfu}^{(M)}}_{\ell^2}
        \lesssim \sqrt{M}r^{-k} \norm{\bfu}_{\ell^2_{\bar\bfomega}} .
        = \sqrt{M}r^{-k} \norm{u}_{H^k} .
    $$
\end{remark}

But in contrast to the model class in Section~\ref{sec:sparse_als} the matrices $Q\in\tilde{\mcal{Q}}_{k,R,\bfomega}$ are not spanned by a subbasis of the standard product basis.
As a consequence, $\tilde{\mcal{M}}_{R,r,\bfomega}$ is no longer a subset of the sparse vectors $\ell^0_\bfomega$ and Theorem~\ref{thm:sparse_RIP_weighted} can no longer be applied directly, as it was done in the proof of Corollary~\ref{cor:sparse_sample_complexity}.
% to every $\ell^0_{Q\bfomega}$ in a union bound argument as it was used in~\cite{trunschke_2022_convergence}.
Instead, we rely on the following strong property:
If the RIP is satisfied at a point, it is satisfied in a small neighbourhood of that point.
This is stated formally in the subsequent lemma.
\begin{lemma}
\label{lem:rip_is_strong}
    Let $\delta,\tau\in[0,1)$.
    Then there exists a constant $\varepsilon\ge 8\tau$ such that for every $a\in L^\infty_w$
    \begin{equation}
        \operatorname{RIP}_{\braces{a}}\pars{\delta}
        \Rightarrow
        \operatorname{RIP}_{B_{L^\infty_w}\pars{a, \norm{a}\tau}}\pars*{\delta + \varepsilon},
    \end{equation}
    with $\varepsilon\le 15\tau$, if $\delta \le \frac12$ and $\tau \le \frac{1}{4}$.
    For $\delta\le\frac12$ and $\tau \le\frac\delta{15}$, this implies
    \begin{equation}
        \operatorname{RIP}_{\braces{a}}\pars{\delta}
        \Rightarrow
        \operatorname{RIP}_{B_{L^\infty_w}\pars{a, \norm{a}\tau}}\pars{2\delta} .
    \end{equation}
\end{lemma}

To prove this lemma, we require the following result.
\begin{lemma}
\label{lem:unitised_norm}
    Let $a,b\in\mathcal V$ be bounded with respect to $\norm{\bullet}_{w,\infty}$ and define $\tilde{a} := \frac{a}{\norm{a}}$ and $\tilde{b} := \frac{b}{\norm{b}}$.
    Moreover, let $\norm{\bullet}_*$ denote either the norm $\norm{\bullet}$ or the empirical norm $\norm{\bullet}_n$.
    Then
    $$
        \norm{\tilde a - \tilde b}_* \le \frac{1 + \norm{\tilde{b}}_*}{\norm{a}} \norm{a - b}_{L^\infty_w} .
    $$
\end{lemma}
\begin{proof}
    By triangle and reverse triangle inequality, it holds that
    \begin{align}
        \norm*{\frac{a}{\norm{a}} - \frac{b}{\norm{b}}}_*
        &\le \frac{1}{\norm{a}} \pars*{\norm{a - b}_* + \norm{b - \tfrac{\norm{a}}{\norm{b}} b}_*}
        = \frac{\norm{a - b}_*}{\norm{a}} + \frac{\abs{\norm{b}-\norm{a}}}{\norm{a}} \frac{\norm{b}_*}{\norm{b}}
        \le \frac{\norm{a - b}_*}{\norm{a}} + \frac{\norm{a - b}}{\norm{a}} \frac{\norm{b}_*}{\norm{b}} .
    \end{align}
    The claim follows, since both $\norm{\bullet}$ and $\norm{\bullet}_n$ are dominated by $\norm{\bullet}_{L^\infty_w}$.
\end{proof}

\begin{proof}[Proof of Lemma~\ref{lem:rip_is_strong}]
    Let $B := B_{\norm{\bullet}_{w,\infty}}\pars{a, r}$ with $r := \norm{a}\tau$ and define $\tilde{b} := \frac{b}{\norm{b}}$  for any $b\in B$.
    We want to show that
    \begin{equation}
    \label{eq:rip_b_tilde}
            \operatorname{RIP}_{B}\pars{\delta + \varepsilon} \Leftrightarrow \abs{\norm{\tilde{b}}^2 - \norm{\tilde{b}}_n^2} \le \delta + \varepsilon
            \text{ for all }b\in B,
    \end{equation}
    given that $\operatorname{RIP}_{\braces{a}}\pars{\delta} \Leftrightarrow\abs{\norm{\tilde{a}}^2 - \norm{\tilde{a}}_n^2} \le \delta$ holds.
    For this, let $\norm{\bullet}_*$ denote either $\norm{\bullet}$ or $\norm{\bullet}_n$ and observe that for any $b\in B$ it holds that
    \begin{equation}
    \label{eq:star_norm_squared}
        \abs*{\norm{\tilde a}_*^2 - \norm{\tilde{b}}_*^2}
        \le \pars{\norm{\tilde a}_* + \norm{\tilde b}_*}\norm{\tilde a - \tilde b}_*
        \le \pars{\norm{\tilde a}_* + \norm{\tilde b}_*} \pars{1 + \norm{\tilde b}_*} \tau,
    \end{equation}
    where the last inequality follows from Lemma~\ref{lem:unitised_norm} and the assumption $\norm{a-b}_{L^\infty_w}\le r = \norm{a}\tau$.
    By assumption, it holds that $\norm{\tilde a}_n \le \sqrt{1 + \delta}$ and using the fact that $\norm{a - b}_* \le \norm{a - b}_{L^\infty_w} \le r$, we can bound
    \begin{equation}
        \norm{\tilde b}_n
        = \frac{\norm{b}_n}{\norm{b}}
        \le \frac{\norm{a}_n + r}{\norm{a} - r}
        \le \frac{\sqrt{1+\delta}\norm{a} + r}{\norm{a} - r}
        = \frac{\sqrt{1+\delta} + \tau}{1 - \tau} .
    \end{equation}
    Inserting these estimates into equation~\eqref{eq:star_norm_squared} gives the bounds
    \begin{equation}
        \abs*{\norm{\tilde a}^2 - \norm{\tilde b}^2}
        \le 4\tau
        \qquad\text{and}\qquad
        \abs*{\norm{\tilde a}_n^2 - \norm{\tilde b}_n^2}
        \le 4\tau c ,
    \end{equation}
    with $c := \frac{1+\delta}{\pars{1 - \tau}^2}$.
    % Proof:
    % $$
    %     \pars{\norm{\tilde a}_* + \norm{\tilde b}_*} \pars{1 + \norm{\tilde b}_*}
    %     \le \pars*{\sqrt{1+\delta} + \frac{\sqrt{1+\delta} + \tau}{1 - \tau}} \pars*{1 + \frac{\sqrt{1+\delta} + \tau}{1 - \tau}}
    %     \le \pars{1+\delta}\pars*{1 + \frac{1 + \tau}{1 - \tau}}^2
    %     = 4 \frac{1+\delta}{\pars{1 - \tau}^2}
    % $$
    We can now use the triangle inequality to prove~\eqref{eq:rip_b_tilde} via
    \begin{align}
        \abs{\norm{\tilde b}^2 - \norm{\tilde b}^2_n}
        &\le \abs{\norm{\tilde b}^2 - \norm{\tilde b}^2_n - \pars{\norm{\tilde a}^2 - \norm{\tilde a}^2_n}} + \abs{\norm{\tilde a}^2 - \norm{\tilde a}^2_n} \\
        &\le \abs{\norm{\tilde b}^2 - \norm{\tilde a}^2} + \abs{\norm{\tilde b}^2_n - \norm{\tilde a}^2_n} + \abs{\norm{\tilde a}^2 - \norm{\tilde a}^2_n} \\
        &\le \underbrace{4\tau + 4\tau c}_{=: \varepsilon} + \delta .
    \end{align}
    To obtain the lower bound for $\varepsilon$, observe that $c \ge 1$ with equality when $\delta = \tau = 0$.
    Therefore, $\varepsilon = 4\pars{1+c}\tau \ge 8\tau$.
    
    \medskip
    
    To obtain the other bound, observe that the function $(\delta, \tau) \mapsto c(\delta, \tau)$ is increasing in both arguments.
    Therefore $c(\delta, \tau) \le c(\frac12, \frac14) = \frac{8}{3}$ for any $\delta\le\frac12$ and $\tau\le\frac14$.
    This results in the loose upper bound
    $$
        \varepsilon \le 15\tau .
    $$
    The special case $\tau \le \frac\delta{15}$ follows immediately.
\end{proof}

For the sake of brevity, let $A := \mcal{M}_{R,r,\bfomega}$ and $B := B_{\ell^0_\bfomega}\pars{0, r}$.
In the preceding section, we used the fact that $A\subseteq B$ to trivially obtain the restricted isometry property of $A$ from that of $B$.
Lemma~\ref{lem:rip_is_strong} implies that this inclusion is not necessary for the RIP of $B$ to extend to $A$ if the set $A$ is close enough to $B$.
To make this intuition rigorous, we define the scale-invariant non-symmetric distance function
\begin{equation}
    \label{eq:dsiL}
    d_{\mathrm{si}L^\infty_w}\pars{A,B}
    := \sup_{a\in A} d_{L^\infty_w}\pars{\operatorname{Cone}\pars{a}, U\pars{B}} ,
\end{equation}
where the distance function $d_{L^\infty_w}\pars{A,B}$ is defined as
\begin{equation}
    d_{L^\infty_w}\pars{A,B}
    := \inf_{a\in A}\inf_{b\in B} \norm{a - b}_{L^\infty_w} .
\end{equation}
With this definition, we can formulate the following theorem.
\begin{theorem}
    \label{thm:rip_close_BA}
    Let $\delta,r\in[0,1)$ and assume that $d_{\mathrm{si}L^\infty_w}\pars{A, B} \le r$.
    Then there exists $\varepsilon \ge 8r$ such that
    \begin{equation}
        \operatorname{RIP}_{B}\pars{\delta}
        \Rightarrow
        \operatorname{RIP}_{A}\pars{\delta + \varepsilon} ,
    \end{equation}
    with $\varepsilon \le 15r$ if $\delta\le\frac12$ and $r\le\frac14$.
    For $\delta \le \frac12$ and $r \le \frac{\delta}{15}$, this implies
    \begin{equation}
        \operatorname{RIP}_{B}\pars{\delta}
        \Rightarrow
        \operatorname{RIP}_{A}\pars{2\delta} .
    \end{equation}
\end{theorem}
\begin{proof}
    Since $d_{\mathrm{si}L^\infty_w}\pars{{A},{B}} \le r$, it holds that for all $a\in{A}$ there exists $t\in\pars{0,\infty}$ and $b\in U\pars{{B}}$ such that $d_{L^\infty_w}\pars{ta, b}\le r$.
    Assuming $\operatorname{RIP}_{{B}}\pars{\delta}$, Lemma~\ref{lem:rip_is_strong} guarantees that there exists a constant $\varepsilon$ satisfying the given bounds such that
    \begin{equation}
        \operatorname{RIP}_{{B}}\pars{\delta}
        \Rightarrow
        \operatorname{RIP}_{U\pars{{B}}}\pars{\delta}
        \Rightarrow
        \operatorname{RIP}_{\braces{b}}\pars{\delta}
        \Rightarrow
        \operatorname{RIP}_{B_{L^\infty_w}\pars{b,r}}\pars{\delta + \varepsilon}
        \Rightarrow
        \operatorname{RIP}_{\braces{ta}}\pars{\delta + \varepsilon}
        \Rightarrow
        \operatorname{RIP}_{\braces{a}}\pars{\delta + \varepsilon} .
    \end{equation}
    This implies $\operatorname{RIP}_{\braces{a}}\pars{\delta+\varepsilon}$ for all $a\in{A}$ and consequently $\operatorname{RIP}_{{A}}\pars{\delta+\varepsilon}$.
\end{proof}

\begin{remark}
    Note that
    $$
        d_{\mathrm{si}L^\infty_w}(A, B)
        := \sup_{a\in A} d_{L^\infty_w}\pars{\operatorname{Cone}\pars{a}, U(B)}
        \le \sup_{\tilde a\in U(A)} d_{L^\infty_w}\pars{\tilde a, U(B)}
        =: d_{\mathrm{h}}(U(A), U(B)), 
    $$
    where $d_{\mathrm{h}}$ is the directed Hausdorff distance between sets.
    This bound can be used in conjunction with Theorem~\ref{thm:rip_close_BA} to provide a simple proof of one of the major corollaries in~\cite{trunschke_2022_convergence}.
    Consider the setting of Proposition~\ref{prop:empirical_projection_error}.
    Assume that $\mcal{M}$ is a manifold with strictly positive reach $R := \operatorname{rch}\pars{\mcal{M}, u_{\mcal{M}}}$ at point $u_{\mcal{M}}$,  and define $\mcal{M}_r := \mcal{M}\cap B\pars{u_{\mcal{M}}, r}$ for any $r\le R$.
    By Proposition~16 in~\cite{trunschke_2022_convergence}, it holds that
    $$
        d_{\mathrm{h}}\pars{U(\braces{u_{\mcal{M}}} - \mcal{M}_r), U(\mbb{T}_{u_{\mcal{M}}}\mcal{M})} \le \frac{r}{R} .
    $$
    From this follows that there exists $\varepsilon > 0$ such that
    $$
        \operatorname{RIP}_{\mbb{T}_{u_{\mcal{M}}} \mcal{M}}\pars{\delta}
        \Rightarrow
        \operatorname{RIP}_{\braces{u_{\mcal{M}}} - \mcal{M}_r}\pars{\delta + \varepsilon}.
    $$
    This means that if the RIP holds for the tangent space at $u_{\mcal{M}}$ then it also holds for a neighbourhood of $u_{\mcal{M}}$ in $\mcal{M}$.
    This is precisely the property that is required in Proposition~\ref{prop:empirical_projection_error}.
\end{remark}

\begin{remark}
    Note that Theorem~\ref{thm:rip_close_BA} may also be used to verify $\operatorname{RIP}_{A}\pars{2\delta}$ by checking $\operatorname{RIP}_{B}\pars{\delta}$ for a finite subset $B\subseteq A$ with $d_{\mathrm{si}L^\infty_w}\pars{A, B}\le \frac{\delta}{15}$.
\end{remark}

The preceding theorem can be utilised to prove the RIP for our model class of semi-sparse tensors $\tilde{\mcal{M}}_{R,r,\omega}$.
This is done in the subsequent theorem, which chooses $A := \tilde{\mcal{M}}_{R,r,\bfomega}$ and $B := B_{\tilde{r}}\pars{\ell^0_\bfomega}$ and shows that
\begin{equation}
    d_{\mathrm{si}L^\infty_w}\pars{A, B} \le \mcal{O}\pars*{\sqrt{\frac{r}{\tilde{r}}}} .
\end{equation}

\begin{theorem}
\label{thm:RIP_MRromega}
    % For a given $Q\in\tilde{\mcal{Q}}_{R,k}$ define $\Omega_Q := Q \operatorname{diag}\pars{\omega^6} Q^\intercal$ and the spectral decomposition $\Omega_Q = U_Q\operatorname{diag}\pars{\beta_Q^2} U_Q^\intercal$.
    % For $r>0$ and $\tilde{\mcal{M}}_{R,r,\omega}$ be the model class defined in~\eqref{eq:MRromega} with
    % \begin{equation}
    %     \tilde{\mcal{C}}_{Q,r,\omega} := \Set*{U_Q C \given C \in B_r\pars{\ell^0_{\beta_Q}}} .
    % \end{equation}
    % Then, 
	Let $\braces{B_j}_{j\in\bracs{D}}$ be orthonormal with respect to the measure $\rho$ and let $w\ge 0$ be any weight function satisfying $\norm{w^{-1}}_{L^1} = 1$.
	Assume the weight sequence satisfies $\bfomega_j \ge \norm{w^{1/2} B_j}_{L^\infty}$ and fix $c, r > 0$ and $\tilde{r} := \pars{1+c^2} \norm{\omega^{-1}}_{\ell^{2/3}}^2r$.
    Then it holds that
    \begin{equation}
        d_{\mathrm{si}L^\infty_w}\pars{\tilde{\mcal{M}}_{R,r,\bfomega}, B_{\ell^0_\bfomega}\pars{0,\tilde{r}}} \le \frac{1}{c} .
    \end{equation}
    % 1/c <= \delta/14 --> c >= 14/\delta
    Hence, if $\delta\le\frac{1}{2}$ and $c\ge\frac{15}{\delta}$, then $\operatorname{RIP}_{B_{\ell^0_\bfomega}\pars{0, \tilde{r}}}\pars{\delta}$ implies $\operatorname{RIP}_{\tilde{\mcal{M}}_{R,r,\bfomega}}\pars{2\delta} $.
    % For a given $Q\in\tilde{\mcal{Q}}_{R,k}$ define $\Omega_Q := Q \operatorname{diag}\pars{\omega^8} Q^\intercal$ and the spectral decomposition $\Omega_Q = U_Q\operatorname{diag}\pars{\beta_Q^2} U_Q^\intercal$.
    % For $r>0$ and $\tilde{\mcal{M}}_{R,r,\omega}$ be the model class defined in~\eqref{eq:MRromega} with
    % \begin{equation}
    %     \tilde{\mcal{C}}_{Q,r,\omega} := \Set*{U_Q C \given C \in B_r\pars{\ell^0_{\beta_Q}}} .
    % \end{equation}
    % Then, for any $c > 1$ and $\tilde{r} := c \norm{\omega^{-1}}_{\ell^{1/2}}\sqrt{r}$, it holds that
    % \begin{equation}
    %     d_{\mathrm{si}L^\infty}\pars{\tilde{\mcal{M}}_{R,r,\omega}, B_{\ell^0_\omega}\pars{0,\tilde{r}}} \le \frac{1}{\sqrt{c^2 - 1}} .
    % \end{equation}
    % Hence, if $\delta\le\frac{3}{4}$ and $c>\sqrt{\frac{2-\delta}{1-\delta}}$, it holds that
    % \begin{equation}
    %     \operatorname{RIP}_{B_{\ell^0_\omega}\pars{0, \tilde{r}}}\pars{\delta}
    %     \Rightarrow
    %     \operatorname{RIP}_{\tilde{\mcal{M}}_{R,r,\omega}}\pars{\delta + \varepsilon}
    %     \qquad\text{with}\qquad
    %     \varepsilon := \frac{18}{\sqrt{c^2-1}} .
    % \end{equation}
\end{theorem}
\todo{It could be a problem that we acutally need $\norm\bullet_{L^\infty_w}$ instead of $\norm\bullet_{L^\infty}$! --- It is not really. Equation~\eqref{eq:bound_L2L_inf_ell05omega} may no longer hold. But we can show $\norm{A-\tilde{A}}_{L^2} \le \norm{A-\tilde{A}}_{\ell^1_\bfomega}$ and $\norm{A-\tilde{A}}_{L^\infty_w} \le \norm{A-\tilde{A}}_{\ell^1_\bfomega}$ individually. This just means that $\abs{\tilde{J}}$ is twice as large...}
\begin{proof}
    Let $A\in \tilde{\mcal{M}}_{R,r,\bfomega}$.
    Since $\bfomega_\nu \ge \norm{B_\nu}_{L^\infty}$, Corollary~\ref{cor:ward_weighted_stechkin} states that for every $\tilde{r}>0$ there exists $\tilde{J}\subseteq\mbb{N}$ and $\tilde{A} := P_{\tilde{J}}A$ such that $\tilde{A} \in B_{\ell^0_\bfomega}\pars{0, \tilde{r}}$ and
    \begin{equation}
    \label{eq:bound_L2L_inf_ell05omega}
        \norm{A - \tilde{A}}_{L^2}
        \le \norm{A - \tilde{A}}_{L^\infty}
        \le \norm{A - \tilde{A}}_{\ell^1_\bfomega}
        \le \tilde{r}^{-1/2} \norm{A}_{\ell^{2/3}_{\bfomega^2}}.
    \end{equation}
    The final $\ell^{2/3}_{\bfomega^2}$-norm can be bounded by Lemma~\ref{lem:trade_sparsity} via
    \begin{equation}
    \label{eq:bound_ell05omega_ell2omega}
        \norm{A}_{\ell^{2/3}_{\bfomega^{2}}}
        \le \norm{\bfomega^{-3}}_{\ell^{2/3}_{\bfomega^2}} \norm{A}_{\ell^2_{\bfomega^3}}
        = \norm{\bfomega^{-1}}_{\ell^{2/3}} \norm{A}_{\ell^2_{\bfomega^3}} .
    \end{equation}
    Recall that $A\in\tilde{\mcal{M}}_{R,r,\bfomega}$ can be written as $A=QC$ with $Q\in\tilde{\mcal{Q}}_{k,R,\bfomega}$ for some $k\in\bracs{M}$ and $C\in\mcal{C}_{Q,r,\bfomega}$.
    Thus,
    \begin{equation}
    \label{eq:bound_ell2omega_ell2beta}
        \norm{A}_{\ell^2_{\bfomega^3}}^2
        = \pars{QC}^\intercal \operatorname{diag}\pars{\bfomega^6} \pars{QC}
        % = C^\intercal \Omega_Q C
        % = C^\intercal U_Q^\intercal \Omega_Q U_Q C
        = C^\intercal \operatorname{diag}\pars{\beta_Q^2} C
        = \norm{C}_{\ell^2_{\beta_Q}}^2 .
    \end{equation}
    Now let $J := \operatorname{supp}\pars{C}$ and bound
    \begin{equation}
    \label{eq:bound_ell2beta_ell2}
        \norm{C}_{\ell^2_{\beta_Q}}^2
        = \sum_{j\in J} \beta_{Q,j}^2 C_{j}^2
        \le \sum_{j\in J} \Big(\sum_{j\in J}\beta_{Q,j}^2\Big) C_{j}^2
        = \norm{C}_{\ell^0_{\beta_Q}} \norm{C}_{\ell^2}^2
        \le r \norm{C}_{\ell^2}^2
        = r \norm{A}_{\ell^2}^2 .
    \end{equation}
    Combining equations~\eqref{eq:bound_L2L_inf_ell05omega},~\eqref{eq:bound_ell05omega_ell2omega},~\eqref{eq:bound_ell2omega_ell2beta} and~\eqref{eq:bound_ell2beta_ell2} results in the bound
    \begin{equation}
        \norm{A - \tilde{A}}_{L^2}
        \le \norm{A - \tilde{A}}_{L^\infty}
        \le \sqrt{c_1\tfrac{r}{\tilde{r}}} \norm{A}_{L^2},
    \end{equation}
    with $c_1 := \norm{\bfomega^{-1}}_{\ell^{2/3}}^2$.
    Finally, recall that $\tilde{A} := P_{\tilde J} A$ and hence
    \begin{equation}
    \label{eq:bound_L2_A_Atilde}
        \norm{\tilde{A}}_{L^2}^2
        = \norm{A}_{L^2}^2 - \norm{A - \tilde{A}}_{L^2}^2
        \ge \tfrac{\tilde{r} - c_1r}{\tilde{r}} \norm{A}_{L^2}^2 .
    \end{equation}
    This means that for every $A\in\tilde{M}_{R,r,\bfomega}$ there exists $\tilde{A}\in B_{\tilde{r}}\pars{\ell^0_\bfomega}$ such that 
    \begin{equation}
        \norm{A - \tilde{A}}_{L^\infty}
        \le \sqrt{c_1\tfrac{r}{\tilde{r}}} \norm{A}_{L^2}
        \le \sqrt{c_1\tfrac{r}{\tilde{r}}} \cdot \sqrt{\tfrac{\tilde{r}}{\tilde{r} - c_1r}} \norm{\tilde{A}}_{L^2}
        = \sqrt{\tfrac{c_1r}{\tilde{r} - c_1r}} \norm{\tilde{A}}_{L^2}
        = \tfrac{1}{c} \norm{\tilde{A}}_{L^2},
    \end{equation}
    where the final equality follows from the choice $\tilde{r} := \pars{1+c^2} \norm{\bfomega^{-1}}_{\ell^{2/3}}^2 r = c^2c_1r + c_1r$.
    This shows that for all $A\in\tilde{\mcal{M}}_{R,r,\bfomega}$ there exists a constant $t := \norm{\tilde{A}}_{L^2}^{-1} > 0$ and an element $\bar{A} := t\tilde{A}\in U\pars{B_{\ell^0_\bfomega}\pars{0,\tilde{r}}}$ such that $d_{L^\infty}\pars{tA, \bar{A}} \le \frac{1}{c}$.
    In other words,
    \begin{equation}
        d_{\mathrm{si}L^\infty}\pars{\tilde{\mcal{M}}_{R,r,\bfomega}, B_{\ell^0_\bfomega}\pars{0, \tilde{r}}} \le \frac{1}{c} .
    \end{equation}
    The claim now follows from Theorem~\ref{thm:rip_close_BA}.
\end{proof}

Even though Theorem~\ref{thm:RIP_MRromega} is valid for any weight sequence $\bfomega$, an increasing sequence $\bfomega$ is necessary in practice.
If $\bfomega\propto 1$, then $\tilde{r} = \pars{1+c^2}\norm{\bfomega^{-1}}_{\ell^{2/3}}^2r \propto \pars{1+c^2} \operatorname{dim}\pars{\bfomega}^3r$, where the dimension of the weight vector $\bfomega$ is the dimension of the ambient tensor product space.
Hence, applying Theorem~\ref{thm:RIP_MRromega} with a constant weight sequence would require the RIP to hold for the entire ambient tensor product space.

The nestedness property of the model classes $\tilde{\mcal{M}}_{R,r,\boldsymbol\omega}$ can be proved in the same way as for the model class $\mcal{M}_{R,r,\boldsymbol\omega}$.
\begin{proposition}
    It holds that $\tilde{\mcal{M}}_{R,r,\boldsymbol\omega} - \tilde{\mcal{M}}_{R,r,\boldsymbol\omega} \subseteq \tilde{\mcal{M}}_{2R,2r,\boldsymbol\omega}$.
\end{proposition}
This allows the application of Proposition~\ref{prop:empirical_projection_error}.
As in the preceding section, we can use Theorem~\ref{thm:sparse_RIP_weighted} to provide a bound for the required number of samples when the model class $\tilde{\mcal{M}}_{R,r,\bfomega}$ is used in the optimisation problem~\eqref{eq:min_emp}.
As before, let $b : Y\to\mbb{R}^{d}$ be a vector of $L^2\pars{Y,\rho}$-orthonormal basis functions, define the tensor product basis $B\pars{y} := b\pars{y_1}\otimes \cdots\otimes b\pars{y_M}$ and suppose that the weight sequence $\bfomega$ satisfies $\bfomega_j \ge \norm{B_j}_{L^\infty}$.
Then the following proposition holds true.

\begin{corollary}
\label{cor:semi-sparse_sample_complexity}
	Fix parameters $\gamma\in\pars{0,1}$ and $\delta\in\pars{0, \frac{1}{2}}$.
	Let $\braces{B_j}_{j\in\bracs{D}}$ be orthonormal with respect to the measure $\rho$ and let $w\ge 0$ be any weight function satisfying $\norm{w^{-1}}_{L^1} = 1$.
    Assume the weight sequence satisfies $\bfomega_j \ge \norm{w^{1/2} B_j}_{L^\infty}$ and 
    fix $\tilde{r} := \pars{1+c^2}\norm{\bfomega^{-1}}_{\ell^{2/3}}^2r$ for some $c > \frac{15}{\delta}$.
    Then, if
    \begin{equation}
        n \ge C \delta^{-2} \tilde{r}\max\braces{\log^3\pars{\tilde{r}}\log\pars{d^M}, -\log\pars{\gamma}}
    \end{equation}
    and $y_1,\ldots, y_n$ are drawn independently from $w^{-1}\rho$,
    % and define $\norm{f}_{w,m}^2 := \sum_{i=1}^m w\pars{x_i}f\pars{x_i}^2$ for any function $f\in L^2\pars{\rho}$.
    the probability of $\operatorname{RIP}_{\tilde{\mcal{M}}_{R,r,\bfomega}}\pars{2\delta}$ exceeds $1-\gamma$.
\end{corollary}
\begin{proof}
    Under the given assumptions, Theorem~\ref{thm:sparse_RIP_weighted} guarantees $\operatorname{RIP}_{B_{\ell^0_\bfomega\pars{\bracs{d}^M}}\pars{0, \tilde{r}}}\pars{\delta}$.
    This implies $\operatorname{RIP}_{\tilde{\mcal{M}}_{R,r,\bfomega}}\pars{2\delta}$ by Theorem~\ref{thm:RIP_MRromega}.
\end{proof}

Although it is not clear how to write an algorithm that remains in this model class, this is not a significant drawback, since we can again choose $r$ by cross-validation and $R$ by standard rank adaptation strategies.

\subsection{Numerical method}

We call the resulting algorithm semisparse ALS (SSALS).
The only difference of this method to the sparse ALS (Algorithm~\ref{alg:sparse_als}) is the usage of the $\boldsymbol{\omega}$-orthogonal QC decomposition instead of a sparse QC decomposition.
Due to this change, the SSALS looses the intrinsic rank-adaptivity of the SALS.
But since SSALS is stable by design, the tensor train rank of the coefficient tensor can be chosen arbitrary.
Note that from an approximation error point of view, it would even be optimal to perform SSALS on a full rank tensor, which is infeasible due to the size of the resulting component tensors.
We hence propose to implement a rank-adaptive algorithm that is based on the rank-adaptation strategy proposed in~\cite{Grasedyck2019SALSA}.
This approach splits the sequence of singular values of a singular value decomposition into two groups.
The first group contains all singular values that exceed a certain significance threshold and the second group contains all remaining singular values.
By fixing the size of the second group, dropping the smallest singular values or adding small random singular values if necessary, adaptivity is achieved.
Moreover, since the second group is assumed insignificant, the corresponding singular vectors can be perturbed randomly without adversely affecting the approximation error.
This allows to randomly explore the space of singular vectors in order to find those that are necessary to represent the sought function.
If a singular vector in the second group is important to represent the sought function, the corresponding singular value increases during optimisation and is eventually assigned to the first group.

\section{Experiments} \label{sec:experiments}

This section is concerned with numerical experiments that illustrate the practical performance of the sparse ALS algorithms derived from the theoretical results in the previous sections.
We examine the reconstruction of a quantity of interest of the finite dimensional Darcy problem~\eqref{eq:darcy} with affine and log-affine coefficients.
From Theorem~\ref{thm:Hermite_decay_Bachmayr} it is known that the solution lies in an exponentially weighted $\ell^2$ space.
% Hence, an exponential weighting in the $\ell^q$ space for any $q\le 2$ can be obtained with Lemma~\ref{lem:trade_sparsity}.
As a consequence, a weighted LASSO as used in the SALS should (at least theoretically) provide good approximation rates.
Since the bases of the micro steps may become very large, as discussed in Section~\ref{sec:semisparse_als}, we modify the SALS to terminate after a fixed maximal time.

The source code of the implementation is available at \url{github.com/ptrunschke/sparse_als}.
Moreover, we compare our results to the highly optimised \texttt{tensap} library~\cite{nouy_anthony_2020_3894378}, which can be found at \url{github.com/anthony-nouy/tensap}.

\subsection{Affine Darcy equation}
\label{sec:affine_experiment}

Our first experiment is taken from~\cite{bouchot_2015_CSPG}, where a weighted $\ell^1$ minimisation was used.
We consider model problem~\eqref{eq:darcy} on the unit interval $D = \bracs{0,1}$ and parameter domain $Y = \bracs{-1,1}^L$ with $L=20$.
We consider the forcing term $f\equiv 10$ and the diffusion coefficient
\begin{equation}
    a(x,y) := \frac{1}{10} + \frac{\pi^2}{3} + \sum_{m=1}^L k^{-2} a_m\pars{x} y_m,
\end{equation}
where $a_{2m-1}\pars{x} = \cos\pars{m\pi x}$ and $a_{2m} = \sin\pars{m\pi x}$.
The PDE in its variational form is solved on a uniform grid with $50$ nodes using conforming $P1$ finite elements.
In this first experiment we consider the quantity of interest
    \begin{equation}
        U(y) := \int_{D} u(x,y) \dx .
    \end{equation}
We use the probability measure $\rho = \frac{1}{2^L}\dx[y]$ and weight function $w\equiv 1$ (cf.~\eqref{eq:norms}) and search for the best approximation with respect to $\norm{\bullet}_{L^2\pars{Y,\rho}}$, using a product basis with $d=20$ Legendre polynomials in each variable.
Concerning the weight sequence, we utilise the smallest possible choice $\bfomega_\alpha := \norm{B_\alpha}_{L^\infty}$.
Note that this is not the exponential weighting, that we could have used according to Lemma~\ref{lem:weighted_stechkin} and Theorem~\ref{thm:Legendre_decay_Bachmayr}.
Numerical results for the proposed algorithms for the empirical best-approximation of $U$ is provided in Table~\ref{tbl:darcy_uniform}.

\begin{table}[ht]
    \centering
    \includegraphics[width=\textwidth]{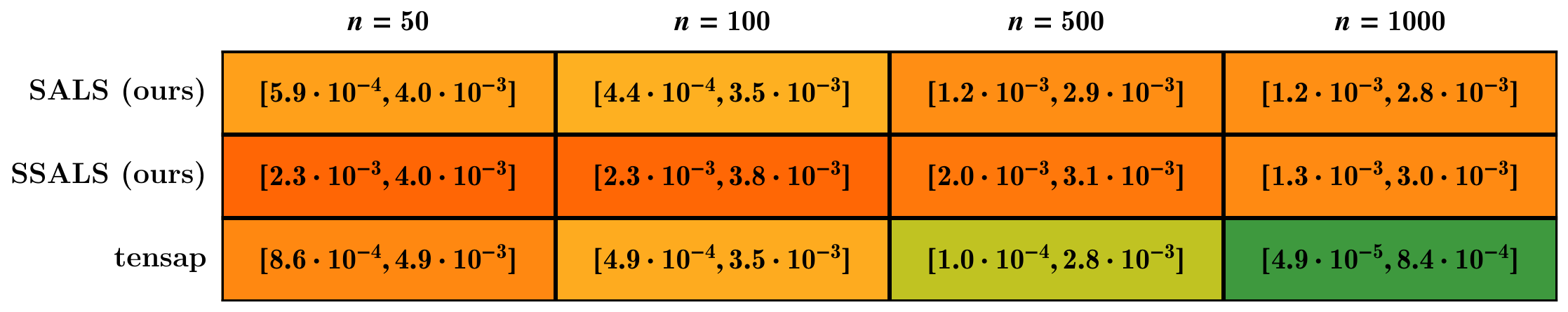}
    \caption{%
    Relative $L^2$-approximation error for the quantity of interest in Section~\ref{sec:affine_experiment}.
    The relative error in the $L^2$-norm is estimated on a test set of $1\,000$ independent samples.
    The experiments are performed $10$ times and the $5\%$ and $95\%$ quantiles are displayed.
    All algorithms use the same samples to compute the empirical approximation (in each column) and the errors are always computed on the same test set.
    SALS and SSALS are compared to the \texttt{TreeBasedTensorLearning} procedure with basis adaptation from \texttt{tensap}.%
    }
    \label{tbl:darcy_uniform}
\end{table}

\subsection{Log-affine Darcy equation}
\label{sec:logaffine_experiments}
The second example considers the Darcy equation with log-affine coefficient with $D=\bracs{0, 1}^2$ and $Y=\mbb{R}^L$ where $L=20$.
We define $f\equiv 1$ and
\begin{equation}
    a(x,y) := \exp\pars*{S_k \sum_{m=1}^L m^{-k} \sin(\pi\floor{\tfrac{m}{2}} x_1) \sin(\pi\ceil{\tfrac{m}{2}} x_2) y_m},
\end{equation}
where $k\in\braces{1, 2}$, $S_1 := 2.4$ and $S_2 := 1.9$.
As before, we solve the resulting PDE in its variational form on a uniform grid with $50\times 50$ nodes using conforming $P1$ elements.
The examined quantity of interest now is the coefficient of the most important POD mode corresponding to $20\,000$ sample points.
The POD is computed by performing a SVD on the matrix of solution snapshots.
The resulting singular vectors constitute an (almost) orthogonal basis and the coefficeient for the basis function that is associated with the largest singular value is used as the QoI.
We choose $\rho = \mcal{N}\pars{0, I}$ as a multivariate standard normal distribution and $w$ such that $w^{-1}\rho$ is a multivariate centred normal distribution with variance $2I$ or $4I$.

In this experiment we search for the best approximation with respect to $\norm{\bullet}_{L^2\pars{Y,\rho}}$, using a basis of $d=20$ Hermite polynomials in each mode.
Similar to Theorem~\ref{thm:Hermite_decay_Bachmayr}, the used weight sequence $\omega_\alpha := \norm{\sqrt{w} B_\alpha}_{L^\infty}$ exhibits an exponential scaling.
The results are depicted in Table~\ref{tbl:darcy_lognormal_1} for $k=1$ and in Tables~\ref{tbl:darcy_lognormal_2} and~\ref{tbl:darcy_lognormal_2_var4} for $k=2$.

\begin{table}[ht]
    \centering
    \includegraphics[width=\textwidth]{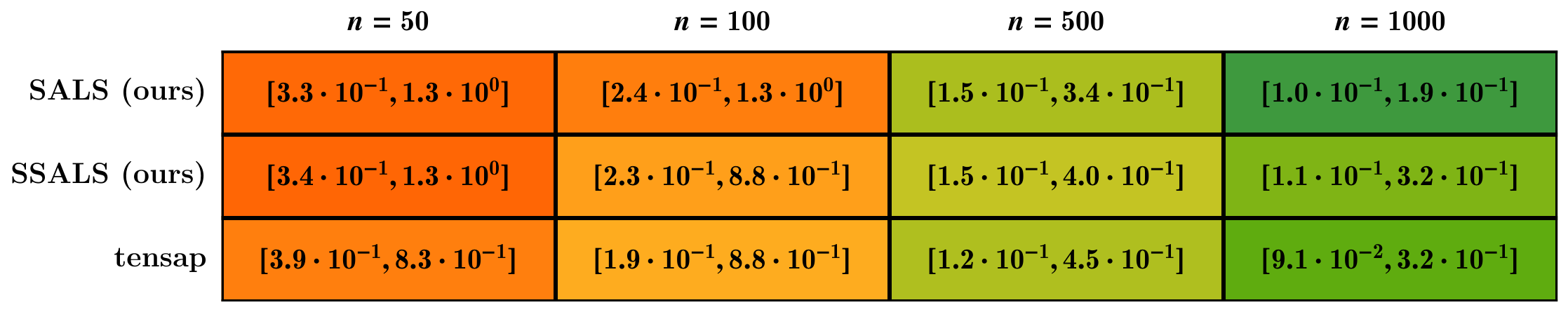}
    \caption{%
    Relative $L^2$-approximation error for the quantity of interest for $k=1$ in Section~\ref{sec:logaffine_experiments}.
    The relative error in the $L^2$-norm is estimated on a test set of $1\,000$ independent samples.
    The experiments are performed $10$ times and the $5\%$ and $95\%$ quantiles are displayed.
    All algorithms use the same samples to compute the empirical approximation (in each column) and the errors are always computed on the same test set.
    SALS and SSALS are compared to the \texttt{TreeBasedTensorLearning} procedure with basis adaptation from \texttt{tensap}.%
    }
    \label{tbl:darcy_lognormal_1}
\end{table}

\begin{table}[ht]
    \centering
    \includegraphics[width=\textwidth]{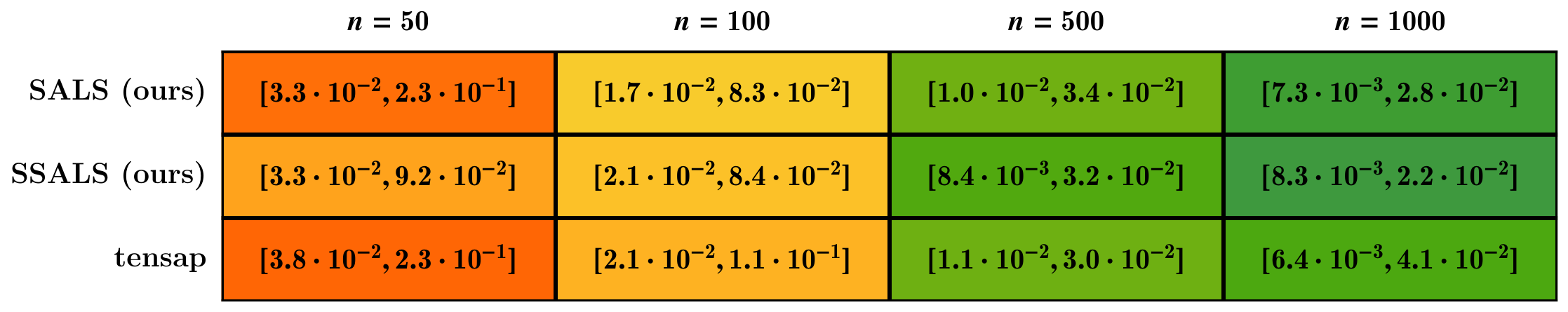}
    \caption{%
    Relative $L^2$-approximation error for the quantity of interest for $k=2$ and a sampling distribution $w^{-1}\rho = \mcal{N}\pars{0, 2I}$ in Section~\ref{sec:logaffine_experiments}.
    The relative error in the $L^2$-norm is estimated on a test set of $1\,000$ independent samples.
    The experiments are performed $10$ times and the $5\%$ and $95\%$ quantiles are displayed.
    All algorithms use the same samples to compute the empirical approximation (in each column) and the errors are always computed on the same test set.
    SALS and SSALS are compared to the \texttt{TreeBasedTensorLearning} procedure with basis adaptation from \texttt{tensap}.%
    }
    \label{tbl:darcy_lognormal_2}
\end{table}

\begin{table}[ht]
    \centering
    \includegraphics[width=\textwidth]{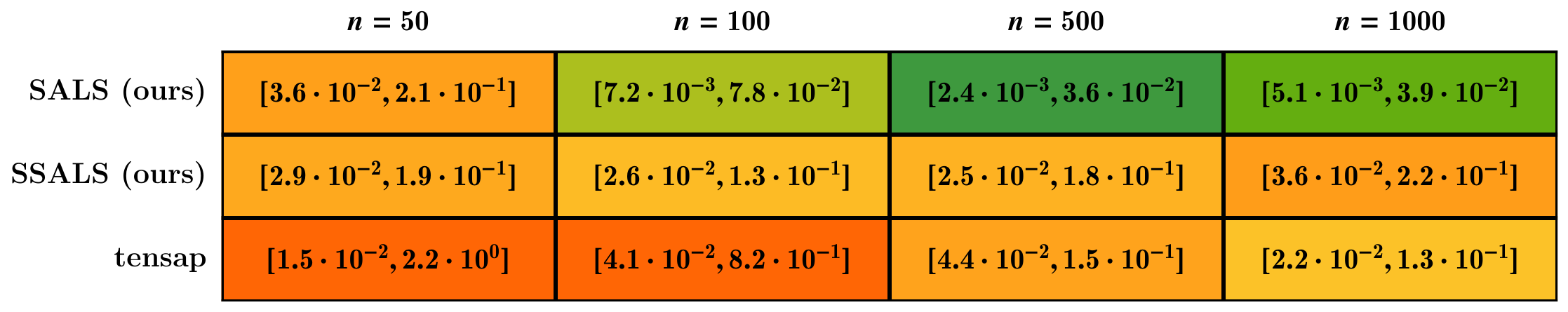}
    \caption{%
    Relative $L^2$-approximation error for the quantity of interest for $k=2$ and a sampling distribution $w^{-1}\rho = \mcal{N}\pars{0, 4I}$ in Section~\ref{sec:logaffine_experiments}.
    The relative error in the $L^2$-norm is estimated on a test set of $1\,000$ independent samples.
    The experiments are performed $10$ times and the $5\%$ and $95\%$ quantiles are displayed.
    All algorithms use the same samples to compute the empirical approximation (in each column) and the errors are always computed on the same test set.
    SALS and SSALS are compared to the \texttt{TreeBasedTensorLearning} procedure with basis adaptation from \texttt{tensap}.%
    }
    \label{tbl:darcy_lognormal_2_var4}
\end{table}

\subsection{Discussion}

% \paragraph{Polynomial recovery}
% \begin{itemize}
%     \item It would be interesting, if we could recover the block-sparse tensor train structure for polynomials~\cite{goette2021blocksparse}.
%     \item But this is impossible, since the sparsity structure depends on the parameterisation and both of our algorithms require a specific parameterisation which may be different from the one necessary for the block-sparsity.
% \end{itemize}
% \paragraph{Sample complexity}

The numerical results illustrate that the obtained accuracy of the newly proposed sparse ALS algorithms SALS and SSALS is comparable to the highly optimised algorithm implemented in \texttt{tensap}, which we consider as base line.
To understand the constraints of our sparse approach, recall the sample complexity bound 
\begin{equation}
    n \ge C \delta^{-2} r\max\braces{\log^3\pars{r}\log\pars{D}, -\log\pars{p}}
\end{equation}
from Theorem~\ref{thm:sparse_RIP_weighted} with $D$ denoting the dimension of the full tensor product space.
Given a fixed sample size $n$, stability parameter $\delta$ and probability $p$, this provides a heuristic upper bound on the weighted sparsity $r$ that can be achieved with our adaptive algorithms.
To make this concrete, let $\delta=\sqrt{C}$ and $p\ge\frac12$.
Then
\begin{equation}
    r \le \frac{n}{\log\pars{D}} = \frac{n}{20\log\pars{20}} \le \frac{n}{50},
\end{equation}
indicating that our adaptive algorithms are restricted to model classes with $r\le \frac{n}{50}$.
We suspect that this bound is implicitly imposed by the cross-validation inside the microstep.
Since $1 \le r$, this argument provides a theoretical explanation for the poor performance in the small-data regime.
For $n=50$, the bound $r \le 1$ allows only to recover the mean since $B_0 \equiv 1$ is the only basis function with $\norm{B_j}_{L^\infty}\le 1$.
This indicates that the problem in subsection~\ref{sec:affine_experiment} is almost trivial to solve.
In general, it can be seen from the numerical experiments that the errors decrease when the sample size increases.
This is to be expected, since the probability of the restricted isometry property increases with the number of samples.

We should also note that, although the experiments seem to work very well, the automatic rank-adaptivity of SALS may fail in pathological scenarios.
It is probably easy to construct an example where the automatic rank adaptation of the sparse QR can not increase the rank in a productive way.
In this case, the semi-sparse ALS should nevertheless succeed, since the rank is adapted randomly.

In comparison to the baseline \texttt{tensap}, the shown results are quantitatively similar.
We consider this very promising since the proposed algorithms only require simple modifications of the standard ALS method.
Note that the \texttt{tensap} algorithm adapts the basis functions with strategies based on leave-one-out cross validation~\cite{Cawley2004,Chapelle2002} or slope heuristics~\cite{Michel2022}, and adapts the ranks based on a strategy similar to \cite{Grasedyck2019SALSA}.

In general, the regularity of the considered function is encoded in $\bfomega$ in the weighted Stechkin lemma.
% For very regular functions, this leads to fast convergence.
Since the target functionals in all experiments are (anisotropically) holomorphic functions, the best $n$-term sets $J_n$ are (anisotropic) balls in the index space.
These are downward-closed sets matching exactly the basis selection strategy of \texttt{tensap}.
It hence should almost be impossible to improve practical results on the selected model problems.
% However, for high ranks or gaps in the sparsity pattern we would expect our new algorithms to behave favourably.

\section*{Acknowledgements}
This project is funded by the ANR-DFG project COFNET (ANR-21-CE46-0015).
ME acknowledges the partial support by the DFG SPP 2298 (Mathematics of Deep Learning).
This work was partially conducted within the France 2030 framework programme, Centre Henri Lebesgue ANR-11-LABX-0020-01.

Our code makes extensive use of the Python packages: \texttt{numpy}~\cite{numpy}, \texttt{scipy}~\cite{scipy}, and \texttt{matplotlib}~\cite{matplotlib}.

{
    \emergencystretch=3em
    \printbibliography
}

\appendix

\section{Basic harmonic summation formulas}

\begin{lemma}
\label{lem:s-harmonic-partial}
    For any $s>0$ and $n\in\mathbb N$, we have $\frac{n^{s+1}}{s+1} \le \sum_{k=1}^n k^s \le \frac{\pars{n+1}^{s+1}}{s+1}$.
\end{lemma}
\begin{proof}
    Both estimates rely on estimating the sum by an integral of the function $f\pars{x} := x^s$.
    Since this function is convex, we know that the trapezoidal rule overestimates the integral and
    \begin{equation}
        \sum_{k=1}^n k^s
        = \frac{1}{2}\pars{f\pars{1} + f\pars{n}} + \sum_{k=1}^{n-1} \frac{f\pars{k}+f\pars{k+1}}{2}
        \ge \frac{n^s+1}{2} + \int_1^n f\pars{x} \dx
        % = \frac{1+n^s}{2} + \bracs*{\frac{x^{s+1}}{s+1}}_1^n
        = \frac{n^s+1}{2} + \frac{n^{s+1}-1}{s+1} .
    \end{equation}
    This gives the lower bound.
    For the upper bound, observe that $f$ is increasing and therefore a left Riemann sum underestimates the integral.
    Consequently,
    \begin{align}
        \sum_{k=1}^n k^s
        &\le \int_1^{n+1} f\pars{x} \dx
        = \frac{\pars{n+1}^{s+1}-1}{s+1}
        \le \frac{\pars{n+1}^{s+1}}{s+1} . \qedhere
    \end{align}
\end{proof}

\begin{lemma}
\label{lem:s-harmonic-rate}
    For any $s>1$ and $s\in\mathbb N$, we have $\tfrac{1}{s-1} k^{1-s} \le \sum_{j=k}^\infty j^{-s} \le \tfrac{s}{s-1} k^{1-s}$.
\end{lemma}
\begin{proof}
    Both estimates rely on estimating the sum by an integral of the function $f\pars{x} := x^{-s}$.
    Since the function is decreasing, we know that a left Riemann sum overestimates the integral and
    \begin{equation}
        \sum_{j=k}^\infty j^{-s}
        \ge \int_{k}^\infty f\pars{x}\dx
        = \frac{1}{s-1}k^{1-s} .
    \end{equation}
    For the same reason, a right Riemann sum underestimates the integral and
    \begin{align}
        \sum_{j=k}^\infty j^{-s}
        &\le f\pars{k} + \int_k^{\infty} f\pars{x} \dx
        = k^{-s} + \frac{k^{1-s}}{s-1}
        \le \pars*{1 + \frac{1}{s-1}} k^{1-s}
        = \frac{s}{s-1} k^{1-s} . \qedhere
    \end{align}
\end{proof}
\section{Example of a hierarchical basis}
\label{app:sparse_grid_example}

For $k=1$ and $p=2$ the equivalence of Example~\ref{ex:sparse_grids} is easy to see.
In this case, we can use the basis
$$
    \phi_{\ell ,j}(x) \propto \max\braces{1 - \abs{2^\ell {x}  - j}, 0},
    \qquad
    \norm{\phi_{\ell ,j}}_{L^2} = 1,
$$
for any $\ell\in\mbb{N}$ and odd $0<j<2^\ell$.
\todo{Move this to the appendix.}
If we represent $v\in W^{1,2}$ by successive $L^2$-projections onto the spaces $V_\ell := \operatorname{span}\braces{\phi_{\ell, j} : 0<j<2^\ell\text{ odd}}$, it is easy to see that
$$
    \norm{v}_{L^2 }^2
    = \norm*{\sum_{\ell\in\mbb{N}} \sum_j \bfv_{\ell, j} \phi_{\ell, j}}_{L^2\pars{\lambda}}^2
    = \sum_{\ell\in\mbb{N}} \norm*{\sum_j \bfv_{\ell, j} \phi_{\ell, j}}_{L^2\pars{\lambda}}^2
    = \sum_{\ell\in\mbb{N}} \sum_j \abs{\bfv_{\ell, j}}^2,
    % = \sum_{l\in\mbb{N}} \sum_{j=0}^{2^l} \abs{v_{lj}}^2 \norm*{\phi_{lk}}_{L^2}^2
    % = \frac{2}{3} \sum_{l\in\mbb{N}} \sum_{j=0}^{2^l} 2^{-l} \abs{v_{lj}}^2 ,
$$
where the second equality follows due to the hierarchical projection and the third follows since $\phi_{\ell, j}$ have disjoint support for fixed $\ell$.
Moreover, it is easy to see that all $\phi_{j,\ell}$ are orthogonal with respect to the $H^1_0$ semi inner product.
This implies that
$$
    \norm{v}_{H^1_0\pars{\lambda}}^2
    = \sum_{\ell\in\mbb{N}} \sum_j \abs{\bfv_{\ell ,j}}^2 \norm{\phi_{\ell j}}_{H^1_0\pars{\lambda}}^2
    = 3 \sum_{\ell\in\mbb{N}} \sum_j 2^{2\ell} \abs{\bfv_{\ell j}}^2 .
$$
\section{Best $n$-term rates in higher dimensions}
\label{sec:n-term}

Recall that we consider isotropic weight sequences of the form $\bar{\bfomega}\pars{a} = \bfomega\pars{a}^{\otimes M}$, where $a\in\pars{0, \infty}$ determines growth of $\bfomega\pars{a}$.
To obtain worst-case rates for the approximation, we apply Lemmas~\ref{lem:weighted_stechkin} and~\ref{lem:trade_sparsity}
$$
    \norm{\bfu - P_{J_n\bfu } \bfu}_{\ell^2}
    \le \norm{P_{J_{n+1}} \bar\bfomega\pars{a}}_{\ell^2}^{-1} \norm{\bfu}_{\ell^1_{\bar\bfomega\pars{a}}}
    \le \norm{P_{J_{n+1}} \bar\bfomega\pars{a}}_{\ell^2}^{-1} \norm{\bar\bfomega\pars{A}^{-1}}_{\ell^2_{\bar\bfomega\pars{a}}} \norm{\bfu}_{\ell^2_{\bar\bfomega\pars{A}}}
$$
and compute an upper bound for the decay rate $\varepsilon(n) := \norm{P_{J_{n+1}} \bar\bfomega\pars{a}}_{\ell^2}^{-1}$.
Then, for fixed $a$, we choose the parameter $A>a$ as small as possible while ensuring that $\norm{\bar\bfomega\pars{A}^{-1}}_{\ell^2_{\bar\bfomega\pars{a}}}$ is finite.

\subsection{Exponential decay (analytic regularity)}
First, consider the weight sequence $\bfomega\pars{a}_j = f_a(j)$ with $f_a(x) := \exp(ax)$ and note that $\bar\bfomega\pars{a}_j = \bar{f}_a(j) := \prod_{m=1}^M f_a(j_m)$.
% Since $\bar{f}(x) = \exp\pars{a\norm{x}_1} \asymp \exp\pars{Ca\norm{x}_2} =: \bar{f}_{\mathrm{r}}(C\norm{x}_2)$ is (almost) radial and $\bar{f}_{\mathrm{r}}$ is monotonic, it holds that
% $$
% 	\bar{f}(x) \le \bar{f}(y) \quad\Leftrightarrow\quad \norm{x}_{2} \le C \norm{y}_2 .
% $$
To obtain a worst-case bound for the best $n$-term approximation of sequences $\bfu$ with $\norm{\bfu}_{\ell^2_{\bar\bfomega\pars{A}}} = 1$, we seek sets $J^*_n$ of size $n$ that maximise the factor $\norm{P_{J^*_n} \bar\bfomega\pars{a}}_{\ell^2}^{-1}$.
Finding such a set is equivalent to finding a set that minimises $\norm{P_{J^*_n} \bar\bfomega\pars{a}}_{\ell^2}^2$.
Since $\bar{f}_a(j) = \exp(a\norm{j}_1)$ is monotonic in $\norm{j}_1$, we can define for every $R\in \mbb{N}$ the set $J^\circ_R := \braces{j\in\mbb{N}^M : \norm{j}_1 \le R}$, which minimises
$$
    d(R) := \norm{P_{J^\circ_R} \bar\bfomega\pars{a}}_{\ell^2}^2 = \sum_{\norm{j}_1 \le R} \bar{f}_a(j)^2
$$
over all sets $J^\circ_R$ with cardinality bounded by
$$
    n(R) := \abs{J^\circ_R} = \sum_{\norm{j}_1 \le R} 1 .
$$
This means that for every $R\in\mbb{N}$ we can find a set of size $n(R)$ which results in the error bound $d(R)^{-1/2}$.
Solving this relation for $n$ gives an expression for $\varepsilon(n) = d(R(n))^{-1/2}$.
Note that this expression is technically only correct if $n=n(R)$ for some $R\in \mbb{N}$.
To obtain an explicit bound that is valid for all values of $n$,
we derive monotonic lower and upper bounds $\underline{d}\pars{R}\le d\pars{R}$ and $\overline{n}\pars{R}\ge n\pars{R}$ and define the inverse $\underline{R}\pars{n} := \overline{n}^{-1}\pars{n}$ as well as $\overline{\varepsilon}(n) := \underline{d}(\underline{R}(n))^{-1/2}$.
Since these bounds are monotonic, it holds that $\underline{R}(n) \le R(n)$ and thus
$$
    \varepsilon(n) = d(R(n))^{-1/2} \le \underline d(\underline R(n))^{-1/2} = \overline{\varepsilon}(n)
$$
if $n = n(R)$ for some $R$.
Moreover, since $n(R)$ increases monotonically with $R$, we can choose for any $n\in\mbb{N}$ a value $R\in\mbb{N}$ such that $n(R) \le n \le n(R+1) \le \overline{n}(R+1)$.
Then we can bound
$$
    \frac{\varepsilon\pars{n}}{\overline{\varepsilon}\pars{n}}
    \le \frac{\varepsilon\pars{{n}\pars{R}}}{\overline{\varepsilon}\pars{\overline{n}\pars{R+1}}}
    = \pars*{\frac{\underline{d}\pars{\underline{R}\pars{\overline{n}\pars{R+1}}}}{d\pars{R\pars{n(R)}}}}^{1/2}
    = \pars*{\frac{\underline{d}\pars{R+1}}{d\pars{R}}}^{1/2}
    \le \pars*{\frac{\underline{d}\pars{R+1}}{\underline{d}\pars{R}}}^{1/2} .
$$
% Moreover, since $\overline{n}$ is monotonic, we can choose for any $n\in\mbb{N}$ a value $R\in\mbb{N}$ such that $\overline{n}(R) \le n \le \overline{n}(R+1)$.
% Then, due to the monotonicity of $\varepsilon$ and $\overline{\varepsilon}$, we can bound
% $$
%     \frac{\varepsilon\pars{n}}{\overline{\varepsilon}\pars{n}}
%     \le \frac{\varepsilon\pars{\overline{n}\pars{R}}}{\overline{\varepsilon}\pars{\overline{n}\pars{R+1}}}
%     \le \frac{\overline{\varepsilon}\pars{\overline{n}\pars{R}}}{\overline{\varepsilon}\pars{\overline{n}\pars{R+1}}}
%     = \pars*{\frac{\underline{d}\pars{\underline{R}\pars{\overline{n}\pars{R+1}}}}{\underline{d}\pars{\underline{R}\pars{\overline{n}\pars{R}}}}}^{1/2}
%     = \pars*{\frac{\underline{d}\pars{R+1}}{\underline{d}\pars{R}}}^{1/2} .
% $$
Hence, if $c_\varepsilon := \big(\sup_{R\in\mbb{N}} \frac{\underline{d}\pars{R+1}}{\underline{d}\pars{R}}\big)^{1/2}$ remains bounded, we obtain
for any $n\in\mbb{N}$ the bound
$$
    \varepsilon\pars{n} \le c_\varepsilon \overline{\varepsilon}\pars{n} .
$$
% For all $n(R-1) \le n \le n(R)$ the sets $J^*(n)$ are arbitrary sets satisfying $J^*(R-1) \subseteq J \subseteq J^*(R)$.
% Since $\bar{\bfomega}_j = \bar\bfomega_k$ for all $j,k\in J(R)\setminus J(R-1)$, it follows that $d(n)$ increases linearly with $n = n(R-1), \ldots, n(R)$.
To find analytic expressions for $\underline{d}$ and $\overline{n}$, we interpret the sums in $d(R)$ and $n(R)$ as Riemann sums % of the two integrals
$$
    d(R)
    \gtrsim \int_{\substack{\norm{x}_1 \le R\\x\ge 0}} \bar{f}_a(x)^2 \dx
    % \ge \int_{\substack{\norm{x}_2 \le M^{-1/2}R\\x\ge 0}} \bar{f}_a(x)^2 \dx
\qquad
\text{and}
\qquad
    n(R)
    \lesssim \int_{\substack{\norm{x}_1 \le R\\x\ge 0}} 1 \dx .
    % \le \int_{\substack{\norm{x}_2 \le R\\x\ge 0}} 1 \dx .
$$
% These integrals can be computed by symmetry and a change of variables to polar coordinates, noting that
% \begin{align}
%     d(R)
%     &\gtrsim \int_{\substack{\norm{x}_2 \le M^{-1/2} R\\x\ge 0}} \bar{f}_a(x)^2 \dx
%     = 2^{-M} \int_{\norm{x}_2 \le M^{-1/2} R} \exp(2a\norm{x}_1) \dx \\
%     &\gtrsim \int_{\norm{x}_2 \le M^{-1/2} R} \exp(2a\norm{x}_2) \dx
%     \gtrsim \int_0^{M^{-1/2} R} \exp(2ar) r^{M-1} \dx[r]
%     \gtrsim \exp(2aM^{-1/2}R) .
% \end{align}
% and
% $$
%     n(R)
%     \lesssim \int_{\substack{\norm{x}_2 \le R\\x\ge 0}} 1 \dx
%     \lesssim \frac{R^M}{\Gamma\pars{\tfrac{M}{2}+1}} .
% $$

To compute the integrals, recall that the volume and surface area of the $M$-dimensional $\ell^1$-ball are given by 
$$
    V_M(R) = 2^M \frac{R^M}{M!}
    \qquad\text{and}\qquad
    A_M(R) = 2^M \sqrt{M} \frac{R^(M-1)}{(M-1)!} .
$$
Utilising the symmetry of the $\ell^1$-ball, this immediately yields
$$
    n(R)
    \lesssim \int_{\substack{\norm{x}_1 \le R\\x\ge 0}} 1 \dx
    = 2^{-M} \int_{\norm{x}_1 \le R} 1 \dx
    = 2^{-M} V_M(R)
    = \frac{R^M}{M!} .
$$
To bound $d(R)$, note that Fubini's theorem implies
\begin{align}
    \int_{\substack{\norm{x}_1\le R\\x\ge0}} \exp\pars{a\norm{x}_1} \dx
    &= \int_0^R\int_{\substack{\norm{x}_1=r\\x\ge0}} \exp\pars{a\norm{x}_1} \dx\dx[r]
    = \int_0^R \exp\pars{ar} \int_{\substack{\norm{x}_1=r\\x\ge0}} 1 \dx\dx[r] \\
    &= \int_0^R \exp\pars{ar} 2^{-M} \int_{\norm{x}_1=r} 1 \dx\dx[r]
    = \int_0^R \exp\pars{ar} 2^{-M} A_M(r) \dx[r] \\
    &= \int_0^R \exp\pars{ar} \sqrt{M}\frac{r^{M-1}}{(M-1)!} \dx[r]
    \overset{(\ast)}{=} \sqrt{M} (-a)^{-M} \pars*{1 - \exp(aR)\sum_{k=0}^{M-1}\tfrac{(-aR)^k}{k!}} \\
    &= \sqrt{M} a^{-M} \exp\pars{aR} \abs*{\exp\pars{-aR} - \sum_{k=0}^{M-1}\tfrac{(-aR)^k}{k!}},
\end{align}
where the equality $(\ast)$ follows from the definition of the incomplete gamma function.
% Note that the absolute value in the last line is the remainder of a Taylor series expansion of $x\mapsto\exp\pars{x}$ for the value $x=-aR$.
% Using the Lagrange form of the remainder, we can write
% $$
%     \abs*{\exp\pars{-aR} - \sum_{k=0}^M\tfrac{(-aR)^k}{k!}}
%     = \exp(\xi_M)\frac{(aR)^{M+1}}{(M+1)!}
% $$
% for some $\xi_M\in(-aR, 0)$.
% This yields
% \begin{align}
%     \int_{\substack{\norm{x}_1\le R\\x\ge0}} \exp\pars{a\norm{x}_1} \dx
%     &\le \tfrac{R^{M+1}}{(M+1)!}\exp\pars{aR}
% \intertext{and, when $R$ is large enough,}
%     \int_{\substack{\norm{x}_1\le R\\x\ge0}} \exp\pars{a\norm{x}_1} \dx
%     &\gtrsim \tfrac{c}{a} R^{M}\exp\pars{aR} .
% \end{align}
Note that the last line approaches $\sqrt{M} a^{-M}\exp(aR)\frac{(aR)^{M-1}}{(M-1)!} = \sqrt{M} \frac{R^{M-1}}{a(M-1)!}\exp(aR)$ as $R$ increases.
For the sake of simplicity, we hence compute the rates only up to asymptotic equivalence.
This yields the bounds
\begin{itemize}
    \item $\overline{n}(R) = c_n \frac{R^M}{M!}$,
    \item $\underline{R}(n) = c_R n^{1/M}$ with $c_R := \pars{\tfrac{M!}{c_n}}^{1/M}$ and
    \item $\underline{d}(R) = c_d \sqrt{M} \frac{R^{M-1}}{2a(M-1)!}\exp(2aR)$.
\end{itemize}
Moreover, it holds that $c_\varepsilon = \sup_{R\in\mbb{N}_{\ge1}} \pars*{1 + \frac1R}^{(M-1)/2} \exp(2a) < \infty$ and, consequently,
$$
    \varepsilon(n)
    \le c_\varepsilon \overline{\varepsilon}(n)
    = c_\varepsilon c_d n^{-(M-1)/(2M)} \exp\pars{-c_R a n^{1/M}} .
$$

Finally, observe that $\norm{\bar{\bfomega}\pars{A}^{-1}}_{\ell^2_{\bar\bfomega\pars{a}}} < \infty$ is valid for any $a<A$.

\subsection{Algebraic decay (mixed Sobolev regularity)}
Consider the weight sequence $\bfomega\pars{a}_j = g_a(j)$ with $g_a(x) := (x+1)^a$ and note that $\bar\bfomega\pars{a}_j = \bar{g}_a(j) := \prod_{m=1}^M g_a(j_m)$.
% g(x) = (x+1)^a = exp(ln(x+1)a) = exp(ay) = f(y) with y = ln(x+1)
Since $g_a(x) = f_a(\ln(x+1))$, this case can be reduced to the case of exponential decay.
The set of indices which minimise the decay rate are $J^\circ_R := \braces{j\in\mbb{N}^M : \norm{\ln\pars{j+1}}_1 \le R}$.
Replacing the (Riemann) sums by integrals and performing the change of variables $y := \ln\pars{x+1}$, we obtain
% x = exp(y) - 1 >= 0 --> exp(y) >= 1 --> y >= 0
\begin{align}
    d(R)
    &\gtrsim \int_{\substack{\norm{\ln\pars{x+1}}_1 \le R \\ x\ge 0}} \bar{g}_a(x)^2 \dx
    = \int_{\substack{\norm{y}_1 \le R \\ y\ge 0}} \bar{f}_a\pars{y}^2 \exp(\norm{y}_1) \dx[y]
    = \int_{\substack{\norm{y}_1 \le R \\ y\ge 0}} \exp\pars{(2a+1)\norm{y}_1} \dx[y], \\
    n(R)
    &\lesssim \int_{\substack{\norm{\ln\pars{x+1}}_1 \le R \\ x\ge 0}} 1 \dx
    = \int_{\substack{\norm{y}_1 \le R \\ y\ge 0}} \exp(\norm{y}_1) \dx[y] .
    % \asymp 2^{-M} \int_{B(0, R)} \exp(C\norm{x}_2) \dx
    % \lesssim \exp(R) .
\end{align}
% where the last inequality follows from the same arguments as in the exponential case.
% Similarly, we obtain
% \begin{align}
%     n(R)
%     \lesssim \int_{\substack{\norm{\ln\pars{x+1}}_1 \le R \\ x\ge 0}} 1 \dx
% 	= \int_{\substack{\norm{y}_1 \le R \\ y\ge 0}} \exp(\norm{y}_1) \dx[y]
% 	% \asymp 2^{-M} \int_{B(0, R)} \exp(C\norm{x}_2) \dx
% 	\lesssim \exp(R) .
% \end{align}

This yields the bounds
\begin{itemize}
    \item $\overline{n}(R) = c_n \sqrt{M} \frac{R^{M-1}}{(M-1)!}\exp(R)$,
    \item $R(n) = (M-1) W\pars*{c_R n^{1/(M-1)}}$ with $c_R := \frac{1}{M-1} \pars*{\tfrac{(M-1)!}{c_n\sqrt{M}}}^{1/(M-1)}$ and
    \item $\underline{d}(R) = c_d \sqrt{M} \frac{R^{M-1}}{(2a+1)(M-1)!}\exp((2a+1)R)$.
\end{itemize}
Moreover, it holds that $c_\varepsilon = \sup_{R\in\mbb{N}_{\ge1}} \pars*{1 + \frac1R}^{(M-1)/2} \exp(2a+1) < \infty$.
% {\color{gray}
% To obtain a decay rate, we define $y := c_R n^{1/(M-1)}$ and note that
% \begin{align}
%     \underline{d}\pars{\underline{R}\pars{n}}
%     &\propto W\pars{y}^{M-1} \exp((2a+1) (M-1) W\pars{y}) \\
%     % = W\pars{y}^M \exp((2a+1) M W\pars{y})
%     &= y^{(2a + 1)(M-1)} W(y)^{-2a(M-1)} \\
%     % = y^{(2a + 1)M} W(y)^{2aM}
%     &\propto n^{2a + 1} W(y)^{-2aM} \\
%     &\gtrsim n^{2a + 1} .
% \end{align}
% This yields the bound
% $$
%     \varepsilon(n) \lesssim n^{-(a + 1/2)} .
% $$
% }
To obtain a decay rate, we define $y := c_R n^{1/(M-1)}$ and recall that~\cite{Hoorfar2008}
$$
    \ln\pars{y} - \ln\ln\pars{y} \le W(y) \le \ln\pars{y} - \tfrac12\ln\ln\pars{y} \le \ln\pars{y}
$$
for every $y\ge e$.
This implies
\begin{align}
    \underline{d}\pars{\underline{R}\pars{n}}
    &\propto W\pars{y}^{M-1} \exp((2a+1) (M-1) W\pars{y}) \\
    &= y^{(2a + 1)(M-1)} W(y)^{-2a(M-1)} \\
    &\ge y^{(2a + 1)(M-1)} \ln\pars{y}^{-2a(M-1)} \\
    &\propto n^{2a + 1} \ln\pars{n}^{-2a(M-1)},
\end{align}
which yields the bound
$$
    \varepsilon(n) \lesssim n^{-(a + 1/2)} \ln\pars{n}^{a(M-1)}.
$$
Finally, observe that $\norm{\bar{\bfomega}\pars{A}^{-1}}_{\ell^2_{\bar\bfomega\pars{a}}} < \infty$ is valid for any $a<A-\tfrac{1}2$.

% % \norm{\ln(x+1)} \le R <--> \ln(x+1) \in B(0, R) <--> x \in \exp(B(0,R)) - 1
% Using a similar argument as in the exponential case, we obtain
% \begin{align}
% 	d(R)
%     % = \sum_{j\in J} g(j)^2
% 	&\asymp \int_{U(R)} \bar{g}(x)^2 \dx
% 	= \int_{B(0, C^{-1}R) \cap Q} \bar{f}(2x) \exp(\norm{x}_1) \dx
% 	% = 2^{-M} \int_{B(0, R)} \exp((a+1)\norm{x}_1) \dx
% 	% \asymp \int_0^R \exp(C (a+1) r) r^{M-1} \dx[r]
% 	% = C a^{-M} (\Gamma(M, -(a+M)R) - (M-1)!)
% 	\asymp \exp(2(a+\tfrac{1}{2})R), \\
%     n(R)
%     &\asymp \int_{U} 1 \dx
% 	= \int_{B(0, C^{-1}R) \cap Q} \exp(\norm{x}_1) \dx
% 	% \asymp 2^{-M} \int_{B(0, R)} \exp(C\norm{x}_2) \dx
% 	\asymp \exp(R) .
% \end{align}
% % We again obtain a parametric relationship.
% % $$
% % 	\text{The best }n(R)\text{-term approximation error decays like }d(R)^{-1} .
% % $$
% % d(R)^(-1/2) = exp(-C(a+1/2)R)
% % n(R) =  exp(CR) -->  R(n) = ln(n)/C
% % --> d(R(n))^(-1/2) = exp(-C(a+1/2)R(n)) = exp(-C(a+1/2)ln(n)/C) = n^(-(a+1/2))
% Solving for $n$ yields
% $$
%     \varepsilon(n) \asymp n^{-(a+1/2)} .
% $$
\section{The advantage of low ranks for approximation}
\label{app:low-rank_advantage}

To illustrate the advantage of this new format, consider approximating the rank-$1$ function $x\mapsto \exp\pars{x_1 + \ldots + x_M}$ by Legendre polynomials.
To solve this approximation problem by means of an ALS-type algorithm, a sequence of microsteps have to be performed that read
\begin{equation}
    \operatorname*{minimise}_{\norm{C}_{\ell^0_{\boldsymbol{\beta}}} \le r}\ \norm{F - MQC}_{\ell^2}^2.
\end{equation}
Here, the vector $F$ and operator $M$ are defined as in~\eqref{eq:F_and_M} with $B = \operatorname{vec}(b\otimes \cdots \otimes b)$ given by a vector of tensor product Legendre polynomials $b:[-1,1]\to\mbb{R}^d$ of degree at most $d-1$.
The operator $Q$ maps the core tensor $C$ to the full tensor and corresponds to a choice of basis $Q^\intercal B$ for the least squares problem of the microstep.
Note that the weighted sparsity constraint $\norm{C}_{\ell^0_{\boldsymbol{\beta}}} \le r$ is less restrictive the better the sought function $u$ can be expressed in the basis $Q$.
It is therefore instructive to compare the basis $Q$ that is chosen in the $k$\textsuperscript{th} microstep of sparse ALS (Algorithm~\ref{alg:sparse_als}) to the minimal basis that is chosen by a classical ALS.
\paragraph{Sparse ALS}
In the sparse ALS, $Q\in\mcal{Q}_{R,k}$ is (up to reshaping) an orthogonal matrix where every column is a standard basis vector (cf.~Theorem~\ref{thm:sparse_tt}).
This means that the basis functions $\tilde{B}^{\mathrm{sparse}} := Q^\intercal B$ in this linear least squares problem are of the form
\begin{equation}
    \tilde{B}^{\mathrm{sparse}}_j\pars{x} := B_{\alpha^{(j)}}\pars{x}
\end{equation}
for some multi-indices $\alpha^{(j)}\in [d]^M$,
i.e.\ that every $\tilde{B}^{\mathrm{sparse}}_j$ is a product Legendre polynomial of potentially high degree.
Since the sought function
$$
    u\pars{x}
    = \exp\pars{x_1+\ldots+x_M}
    = \exp\pars{x_1}\cdots\exp\pars{x_M}
    = u_1\pars{x_1}\cdots u_M\pars{x_M}
$$
is of rank $1$, the best approximation $v = v_1\otimes\cdots\otimes v_M$ is of rank $1$ as well.
But the number of terms in this approximation is exponentially large.
This must be the case for any approximation with small error, since the approximation error $\norm{u-v}_{L^2}$ is equivalent to the approximation error the individual factors $$
    \max_{k} \norm{u_k - v_k}_{L^2}
    \lesssim \norm{u - v}_{L^2}
    \lesssim \norm{u_1 - v_1}_{L^2} + \cdots + \norm{u_M - v_M}_{L^2}
    \lesssim \max_{k} \norm{u_k - v_k}_{L^2} .
$$
Due to this error equivalence   and the symmetry of $u$, every factor $u_k$ must be approximated to the same accuracy by a polynomial $v_k$ of uniform degree $g-1$.
% Suppose that each factor is approximated to the same accuracy with a polynomial of degree $g$.
Thus $\tilde{B}^{\mathrm{sparse}}$ has to be the product basis $\tilde{B}^{\mathrm{sparse}} = \tilde{b}^{\otimes (k-1)}\otimes b\otimes \tilde{b}^{\otimes (M-k)}$, where $\tilde{b} : [-1,1]\to\mbb{R}^{g}$ ith the vector of Legendre polynomials of degree at most $g-1$.
This means that the basis has to be $(g^{M-1}d)$-dimensional.

\paragraph{Standard ALS}
Suppose that every component tensor other than the $k$\textsuperscript{th} has been updated at least once.
Then the current approximation has the form
\begin{equation}
    QC = E_1 \otimes \cdots \otimes E_{k-1} \otimes \operatorname{vec}\pars{C} \otimes E_{k+1} \otimes \cdots\otimes E_M ,
\end{equation}
where the vectors $E_\ell$ are the coefficients of one-dimensional Legendre polynomial approximations to the exponential function
$$
    \widetilde{\exp}_\ell\pars{x} := E_{\ell}^\intercal b\pars{x}
$$
on the interval $\bracs{-1, 1}$.
The basis function $\tilde{B}^{\mathrm{dense}} := Q^\intercal B$ for the microstep are hence give by
\begin{equation}
    \tilde{B}^{\mathrm{dense}}_j\pars{x} := b_j\pars{x_k} \prod_{\ell\ne k} \widetilde{\exp}_\ell\pars{x_\ell} .
\end{equation}

\paragraph{Comparison}
In the preceding two paragraphs we have seen that the basis dimension in the sparse ALS is exponentially larger than in the standard ALS.
From a computational point of view, this drastically increases the complexity of the micro steps.
From a statistical point of view this also decreases the probability of the RIP. 
Assuming the approximations $\widetilde\exp_\ell \approx \exp$ are sufficiently good, it holds for every $j\geq 4$ that 
\begin{equation}
    \norm{\widetilde{\exp}_\ell}_{L^\infty\pars{\bracs{-1, 1}}}
    \approx \norm{\exp}_{L^\infty\pars{\bracs{-1, 1}}}
    = e
    \le \sqrt{2j+1}
    = \norm{b_j}_{L^\infty\pars{\bracs{-1,1}}} .
\end{equation}
This means that $\norm{\tilde{B}^{\mathrm{dense}}_j}_{L^\infty} \le \norm{\tilde B^{\mathrm{sparse}}_j}_{L^\infty}$ (approximately) for the majority of indices $1\le j \le r_{k-1}dr_k$.
Since Theorem~\ref{thm:sparse_RIP_weighted} requires $\boldsymbol\beta_j\ge\norm{\tilde{B}_j^{\mathrm{sparse}}}_{L^\infty}$ for the sparse ALS and $\boldsymbol\beta_j\ge\norm{\tilde{B}_j^{\mathrm{dense}}}_{L^\infty}$ for the standard ALS, the sparsity constraint $\norm{C}_{\ell^0_{\boldsymbol{\beta}}} \le r$ is less restrictive for the standard ALS and the same approximation error can be achieved with a smaller value of $r$.
In this special case, rounding would provide a better basis for the sparse approximation, which indicates that reducing the rank may increase the practical performance of the (then less sparse) ALS.
In general, however, the basis in the low-rank representation is not uniquely defined and has to be adapted before performing the sparse approximation.
% But since sparsity depends on the chosen basis $Q$, we have to be careful in order for the sparsity regularisation to still make sense.
% \input{content/tensor_networks}
% \input{content/notes}

\end{document}